\documentclass[11pt,a4paper]{article}
\usepackage[T1]{fontenc}
\usepackage[utf8]{inputenc}
\usepackage[english]{babel}
\usepackage{lmodern,microtype,amsmath,amssymb,amsthm,mathtools,mathrsfs}
\usepackage[margin=2.15cm,headheight=14pt]{geometry}
\usepackage{booktabs,longtable,array,enumitem,fancyhdr,xcolor}
\usepackage[colorlinks=true,linkcolor=black,citecolor=black,urlcolor=blue!50!black,hypertexnames=false]{hyperref}
\hypersetup{pdftitle={Small-time Wasserstein action for KPZ and Phi4 3},
 pdfauthor={Obayda Julien Assaad}}
\allowdisplaybreaks[2]
\setlist{itemsep=.1em,topsep=.25em}
\numberwithin{equation}{section}
\newcommand{\E}{\mathbb E}\newcommand{\T}{\mathbb T}
\newcommand{\R}{\mathbb R}
\newcommand{\dd}{\,\mathrm d}\newcommand{\Id}{\mathrm{Id}}
\newcommand{\Tr}{\operatorname{Tr}}
\newcommand{\Law}{\operatorname{Law}}\DeclareMathOperator*{\esslim}{ess\,lim}\newcommand{\cI}{\mathcal I}

\newtheorem{theorem}{Theorem}[section]
\newtheorem{proposition}[theorem]{Proposition}
\newtheorem{lemma}[theorem]{Lemma}
\newtheorem{corollary}[theorem]{Corollary}
\theoremstyle{definition}
\theoremstyle{remark}\newtheorem{remark}[theorem]{Remark}
\title{Small-time Wasserstein action\\for KPZ and $\Phi^4_3$}
\author{Obayda Julien Assaad}
\date{}
\begin{document}
\maketitle

\begin{abstract}
We determine the leading small-time Wasserstein action of the full
spatial laws of centered KPZ and dynamical $\Phi^4_3$, both started from
zero. The leading term is Gaussian, with an explicit heat-trace
constant. A common comparison also determines the ratio of length to
displacement. For KPZ, a first-chaos cancellation yields a quadratic
coupling error. For $\Phi^4_3$, two weighted entropies and an adjoint
estimate give comparison bounds uniform in the spectral cutoff.
Integrating a Gaussian rotation gives a differentiable realization of
the full law; the comparison holds at every positive time.
\end{abstract}

\section{Introduction}\label{se:statement}

How fast does the law of a singular stochastic field leave its initial
state? We study this question in the Wasserstein metric for the centered
KPZ profile and the dynamical $\Phi^4_3$ field, both started from zero.
We study the action of the entire spatial law. Its leading small-time
term equals that of the Gaussian linearization; we bound the nonlinear error.

The main difficulty concerns differentiation. A coupling may show that
two laws are close without comparing their speeds. For example,
$\delta_{n^{-1}\sin(nt)}$ converges uniformly to $\delta_0$, whereas its
squared speed is $\cos^2(nt)$. In a singular equation, direct
differentiation also produces rough products and divergent counterterms.
The proof must preserve their cancellations before estimating them.

\subsection{A comparison at the level of the law}

Let $\rho_t$ be a locally absolutely continuous curve in
$\mathcal P_2(H)$, where $H$ is a Hilbert space. Its action
$\mathcal A_\rho(t)$ is the squared norm of its minimum current.
Thus $\mathcal A_\rho(t)=|\dot\rho_t|_{W_2(H)}^2$ almost everywhere.
Projecting any representing current onto the closure of cylinder
gradients gives the minimum current \cite{BB,AGS,Lisini}.
For the zero-start Ornstein--Uhlenbeck law $\gamma_t$ with generator
$A>0$, that current is $B_tx$, where $B_t=A(e^{2tA}-I)^{-1}$.

Subtract this known current from the derivative of the nonlinear law.
The resulting functional has norm
\[
 e(t)=\sup_{\|\nabla\varphi\|_{L^2(\rho_t;H)}\le1}
 \left|\frac d{dt}\int\varphi\,d\rho_t
       -\int\langle\nabla\varphi,B_tx\rangle_H\,d\rho_t\right|.
\]
The tests are smooth cylinder functions. We call $e(t)$ the
\emph{relative derivative error}. Theorem~\ref{common:principle} gives
\begin{equation}\label{intro:comparison}
 |\sqrt{\mathcal A_\rho(t)}-\sqrt{\mathcal A_0(t)}|
 \le e(t)+\frac{W_2(\rho_t,\gamma_t)}{2t}.
\end{equation}
The proof uses orthogonal projection and $\|B_t\|\le(2t)^{-1}$.
We must therefore bound $W_2(\rho_t,\gamma_t)$ and $e(t)$.

The Gaussian action depends only on the spectrum. Suppose
$N_A(L)\sim C_A L^\alpha$, the squared transport norm weights mode $j$
by $\lambda_j^{-s}$, and the noise amplitude is $\sqrt{2a}$.
Put $\beta=1+\alpha-s\in(1,2)$. If
$W_2(\rho_t,\gamma_t)=O(t^\theta)$ and
$e(t)=O(t^{\theta-1})$, with $\psi=\theta-1+\beta/2>0$, then
\begin{equation}\label{intro:universal}
 t^\beta|\mathcal A_\rho(t)-\mathcal A_0(t)|=O(t^\psi),\qquad
 \esslim_{t\downarrow0}t^\beta\mathcal A_\rho(t)
 =aC_A\alpha2^{-\beta}\Gamma(\beta)[\zeta(\beta)-1].
\end{equation}
The limit is ordinary when the derivative bound holds at every time,
as in both applications below.
The Gaussian action is $a\sum_{n\ge2}\Tr(A^{1-s}e^{-2ntA})$.
Summing the leading heat-trace terms gives $\zeta(\beta)-1$
(Proposition~\ref{common:gaussian}).
The comparison applies to both equations and is stable under $W_2$ convergence.
It also determines the ratio of length to distance
(Corollary~\ref{common:integrated}):
\begin{equation}\label{intro:ratio}
 \lim_{t\downarrow0}
 \frac{\operatorname{Length}(\rho|_{[0,t]})}{W_2(\rho_t,\delta_0)}
 =\sqrt{\frac{(\beta-1)(\zeta(\beta)-1)}{2-\beta}},
\end{equation}
independent of $a$, $C_A$ and of the nonlinearity. It equals
$\sqrt{\zeta(3/2)-1}\approx1.270$ for KPZ and
$\sqrt{3(\zeta(7/4)-1)}\approx1.699$ for $\Phi^4_3$.

\subsection{The two full-field results}

For KPZ on $\mathbb T=\mathbb R/\mathbb Z$, we take the flat-start
Cole--Hopf solution of
\[
 \partial_t h=\nu\partial_x^2h+\tfrac g2:(\partial_xh)^2:+\sigma\xi,
 \qquad h(0)=0,
\]
subtract its spatial average, and use $H=L^2_0(\mathbb T)$.
For $\Phi^4_3$ on $\mathbb T^3$, the cubic coefficient is $g>0$,
the noise amplitude is $\sqrt{2a}$, and $A=\mu-\nu\Delta$.
We use $H=H_A^{-3/4}$ and the spectral scheme with fixed counterterms
in \eqref{eq:model}--\eqref{eq:constants}.

Theorems~\ref{kpz:action-limit} and~\ref{th:velocity} prove
\begin{align}
 \lim_{t\downarrow0}t^{3/2}\mathcal A_{\rm KPZ}(t)
 &=\frac{\sigma^2[\zeta(3/2)-1]}{16\sqrt{2\pi\nu}},
       \label{intro:kpz}\\
 \lim_{t\downarrow0}t^{7/4}\mathcal A_{\Phi^4}(t)
 &=\frac{a\Gamma(7/4)[\zeta(7/4)-1]}
           {2^{15/4}\pi^2\nu^{3/2}}.\label{intro:phi4}
\end{align}
For KPZ, the normalized difference from the Gaussian action is
$O_G(g^2t^{1/2})$, uniformly for $|g|\le G$.
For $\Phi^4_3$, it is $O_\varpi(t^{1/8-\varpi})$ for every
$0<\varpi<1/8$, uniformly in the Galerkin cutoff.
Both full-field bounds hold at every positive time in the stated range.
The $\Phi^4_3$ limit concerns the full field; for $X_N$ it requires
$N\sqrt t\to\infty$.
Corollary~\ref{common:integrated} gives the resulting length and
integrated-action asymptotics.

\subsection{Why two ordinary entropies suffice for the quartic field}
\label{se:strategy}

In KPZ, the logarithmic heat representation gives a coupling expansion.
Its linear term is Gaussian. In the exact action difference, the
Gaussian current pairs only with the first chaos of the remainder.
That component vanishes at first order in the coupling.
A third coupling derivative bounds the remainder, giving the quadratic
error by Proposition~\ref{out:second}.

For $\Phi^4_3$, write $u=X_N-Z_N$. Two singular terms obstruct an
ordinary energy estimate: the additive Wick cube and the quadratic
Gaussian potential. A noise-dependent corrector $q$ removes the first.
Positive forward and backward weights $h,k$ handle the second.
For the corrected state $r=u+q$, set $v=r/h$ and use
\[
 \mathcal E_2=\frac12\int hk\,v^2,
 \qquad \mathcal E_4=\frac14\int hk\,v^4.
\]
The weights are exponentials of heat primitives of the Wick square,
as in the multiplicative transformation of Jagannath and
Perkowski~\cite{JP}.
In the relative entropy identity, the rough potential cancels.
Only the errors in the equations for $h,k$ remain.

The quadratic entropy gives
$v\in L^\infty_tL^2_x\cap L^2_tH^1_x\cap L^4_{t,x}$, with a
positive small-time gain. The quartic entropy gives
$v^2\in L^\infty_tL^2_x\cap L^2_tH^1_x$.
Parabolic Sobolev embedding then gives $u^3\in L^2_{t,x}$.
Heat regularity for the corrected state then supplies the spatial
derivatives needed in the differentiated equation.
Only quadratic and quartic tests are used in the nonlinear estimates.

The spectral projection requires care. Quartic testing produces
$\|P_N(r^3)\|_2^2$, not $\|r\|_6^6$. The needed coercivity follows
from the exact identity
\[
 \|r\|_4^4=\langle P_N(r^3),r\rangle
 \le\|P_N(r^3)\|_2\|r\|_2,
 \qquad r\in E_N.
\]
The quadratic estimate controls $\|r\|_2$; this inequality then absorbs
the lower-order quartic terms.

A bounded rotation of the Gaussian noise makes the endpoint derivative
equal to $tB_tZ_t$, without changing the law.
The remaining derivative comes from $u$.
Testing its equation against the adjoint response $p$ uses the cubic dissipation:
\[
 \left|\int_0^T\!\int u^3p\right|
 \le\|u\|_{L^4_{t,x}}^2\|up\|_{L^2_{t,x}}.
\]
The state estimates bound the first factor; the response energy bounds
the second. This bounds the relative derivative uniformly in $N$.
Integrating the rotation gives a differentiable realization of the full law
(Proposition~\ref{ma:full-differentiation}). The derivative bound therefore
passes to the full field at every positive time.

\subsection{Contribution and relation to existing work}

We identify the leading action of both spatial laws and bound its
nonlinear correction. The comparison uses two bounds: $W_2(\rho_t,\gamma_t)$ and $e(t)$.
For $\Phi^4_3$, an additive corrector, positive weights, two ordinary
entropies and the adjoint estimate bound both errors uniformly in the spectral cutoff.

The exponential weights come from the inverse Cole--Hopf transformation
of Jagannath and Perkowski~\cite{JP}, who reduced the periodic
$\Phi^4_3$ equation to a well-posed random PDE; for linear equations the
transformation appears in \cite{HL,GPkpz}. The renormalized gradient
square and its relation to the sunset counterterm already occur in
\cite[Section~2 and Appendix~A]{JP}. Proposition~\ref{cal:sunset}
records the corresponding exact covariance identity for the present
spectral scheme. The gradient square comes from the spatial chain rule
for the Laplacian.

Here the forward and backward weights have explicit equation defects.
Their simultaneous use in the quadratic and quartic relative entropies
gives the state bounds. The same weights control the adjoint of the
linearized flow while retaining its cubic dissipation. The projection
estimates of Appendix~\ref{cal:projection-section} make these estimates
uniform in the Galerkin cutoff. These bounds supply the derivative of
the full law required by the transport comparison.

Several ingredients are classical. Minimum currents and their gradient
projection belong to Wasserstein geometry \cite{BB,AGS,JKO,Lisini};
Gaussian transport is treated in \cite{Gelbrich,Takatsu}.
The entropy identity used here is generalized relative entropy
\cite[Lemma~1.1]{MMP}; see also \cite{ConstantinGRE}. The positive quotient is related
to Doob's $h$-transform \cite{Doob}. Heat regularity, interpolation and
paraproduct estimates are standard \cite{Bony,BCD}.
Our contribution is the comparison of actions and the cancellation
and projection estimates needed for these two equations.

KPZ and its Gaussian linear reference originate in
\cite{KPZ,EW}. We use the Cole--Hopf interpretation
\cite{HairerKPZ,GPkpz}, Walsh integration \cite{Walsh},
positivity \cite{Mueller,ChenKim}, and inverse moments
\cite{MuellerNualart}. The short-time statistics studied in
\cite{GLRHC} concern a different observable from the transport
action of the full law.

For stochastic quantization \cite{PW}, the proof uses the
Gaussian-remainder decomposition \cite{DPD} and paracontrolled calculus
\cite{GIP}; regularity structures give another construction
\cite{Hairer}. The periodic solution theory is supplied by \cite{CC,MW}.
Related field constructions appear in \cite{GHglobal,GHPDE,AK}, and
approximation results in \cite{HM,ZZ}. Strong Feller properties
\cite{HMat} and densities \cite{GL} address other aspects of law regularity.
Our Gaussian estimates use the diagram methods of \cite{MWX},
chaos and Malliavin calculus \cite{Nelson,Gross,Janson,Nualart},
and Hilbert-space stochastic evolution \cite{DPZ}.
Cameron--Martin translation and Gaussian isoperimetry \cite{CM,Borell}
give moments of the response from its bounds under noise shifts, as in \cite{CLL}.

Section~\ref{common:section} proves the comparison and its weak closure.
Sections~\ref{kpz:application} and~\ref{app:phi4} apply it to KPZ and
$\Phi^4_3$. The appendices establish the product, entropy and projection
estimates used in these applications.

\section{A common principle for the leading transport action}
\label{common:section}

We compare the current of $\rho_t$ with the Gaussian current $B_tx$.
The spectrum of $A$ determines the leading action.

Let $A\ge\lambda_*I$, $\lambda_*>0$, be self-adjoint with compact resolvent on a
separable real Hilbert space, with eigenvalues $(\lambda_j)$ counted
with multiplicity. Write $H_s$ for the completion of its eigenvector
span in the norm $\|x\|_{H_s}=\|A^{-s/2}x\|$, and assume
\begin{equation}\label{common:weyl}
 N_A(L):=\#\{j:\lambda_j\le L\}\sim C_A L^\alpha,
 \quad C_A,\alpha>0,
 \qquad \beta=1+\alpha-s\in(1,2).
\end{equation}
Fix $a>0$. The zero-start Gaussian reference and its linear current are
\begin{equation}\label{common:reference}
 \gamma_t=\mathcal N\!\left(0,aA^{-1}(I-e^{-2tA})\right),
 \qquad B_t=\frac{A}{e^{2tA}-I},\qquad \|B_t\|\le\frac1{2t}.
\end{equation}
The covariance is written in the original eigenvector coordinates;
as a covariance on $H_s$ its eigenvalues are
$a\lambda_j^{-1-s}(1-e^{-2t\lambda_j})$.
Since $1+s>\alpha$, these define laws in $\mathcal P_2(H_s)$.
All Wasserstein distances, gradients and vector norms below use $H_s$.

For $\rho\in\mathcal P_2(H_s)$, let $T_\rho$ be the $L^2(\rho;H_s)$ closure of smooth bounded
cylinder gradients. The minimum current $v_t\in T_{\rho_t}$ satisfies
\[
 \frac d{dt}\int\varphi\,d\rho_t
 =\int\langle v_t,\nabla\varphi\rangle\,d\rho_t,
 \qquad \mathcal A_\rho(t)=\|v_t\|_{L^2(\rho_t)}^2.
\]
It exists whenever the derivative functional is bounded in
$\|\nabla\varphi\|_{L^2(\rho_t)}$. Any other representing current differs
from $v_t$ by an element of $T_{\rho_t}^\perp$.
For locally $AC^2$ curves, $\mathcal A_\rho(t)=|\dot\rho_t|_{W_2}^2$
almost everywhere \cite[Sections~8.3--8.4]{AGS}.

\begin{proposition}[The Gaussian action]\label{common:gaussian}
Under \eqref{common:weyl}, the realization
\begin{equation}\label{common:realization}
 Z_t=\sum_j\sqrt{a\lambda_j^{-1}(1-e^{-2t\lambda_j})}\,\xi_j e_j,
\end{equation}
with independent standard real Gaussians $\xi_j$, is locally $C^1$ in
$L^2(\Omega;H_s)$ on $(0,\infty)$ and satisfies $\partial_tZ_t=B_tZ_t$.
Hence $B_tx$ is the minimum current of $\gamma_t$, and its action is a
sum of heat traces at the times $2nt$, $n\ge2$:
\begin{equation}\label{common:gaussian-sum}
 \mathcal A_0(t)
 =a\sum_j\lambda_j^{1-s}
              \frac{e^{-4t\lambda_j}}{1-e^{-2t\lambda_j}}
 =a\sum_{n\ge2}\Tr\bigl(A^{1-s}e^{-2ntA}\bigr).
\end{equation}
Moreover
\begin{equation}\label{common:gaussian-limit}
 \mathcal A_0(t)\le Ct^{-\beta},\qquad
 t^\beta\mathcal A_0(t)\longrightarrow
 K:=aC_A\alpha\,2^{-\beta}\Gamma(\beta)
                              \bigl(\zeta(\beta)-1\bigr).
\end{equation}
For a finite spectral compression with action $\mathcal A_{0,N}$, the
same limit holds along $t\downarrow0$ if the retained spectrum contains
every $\lambda_j\le L_N$ and $tL_N\to\infty$.
\end{proposition}

\begin{proof}
\emph{The Gaussian current.}
Put $q_j=a\lambda_j^{-1}(1-e^{-2t\lambda_j})$ and
$b_j=\lambda_j/(e^{2t\lambda_j}-1)$. Then
\[
 q_j'=2b_jq_j,\qquad (\sqrt{q_j})'=b_j\sqrt{q_j},\qquad
 \sup_{t\ge\tau}\sum_j\lambda_j^{-s}b_j^2q_j
 \le C_\tau\sum_j\lambda_j^{1-s}e^{-2\tau\lambda_j}<\infty.
\]
Thus $\partial_tZ_t=B_tZ_t$ in $L^2(\Omega;H_s)$, locally continuously.
This current is tangent by the approximation below, hence minimal.
Its squared norm is the first sum in \eqref{common:gaussian-sum};
Tonelli and $e^{-4x}/(1-e^{-2x})=\sum_{n\ge2}e^{-2nx}$ give the second.

\emph{The leading term.}
Set
\[
 m_t=(2t)^\alpha\sum_j\delta_{2t\lambda_j},\qquad
 f(x)=x^{1-s}\frac{e^{-2x}}{1-e^{-x}}.
\]
The Weyl law gives
$m_t((0,x])=(2t)^\alpha N_A(x/(2t))\to C_Ax^\alpha$ and
$m_t((0,x])\le Cx^\alpha$ uniformly. Thus
$m_t\rightharpoonup C_A\alpha x^{\alpha-1}\,dx$ vaguely, and
$(2t)^\beta\mathcal A_0(t)/a=\int f\,dm_t$.
For $0<\delta\le1$ and $R\ge1$, dyadic decomposition gives
\begin{align}
 \sup_{t>0}\int_{(0,\delta]} f\,dm_t
 &\le C\sum_{j\ge0}(2^{-j}\delta)^{\alpha-s}
 \le C\delta^{\beta-1},\label{common:small-tail}\\
 \sup_{t>0}\int_{[R,\infty)} f\,dm_t
 &\le C\sum_{j\ge0}(2^jR)^\beta e^{-2^{j+1}R}
 \le CR^\beta e^{-2R}.\label{common:large-tail}
\end{align}
Here $f(x)\lesssim x^{-s}$ near zero; on each dyadic annulus its
power factor varies by a fixed constant, for either sign of $s$.
These bounds prove uniform integrability at both spectral ends and
the bound in \eqref{common:gaussian-limit}. Vague convergence yields
\[
 \lim_{t\downarrow0}t^\beta\mathcal A_0(t)
 =aC_A\alpha2^{-\beta}
   \sum_{n\ge2}\frac{\Gamma(\beta)}{n^\beta}=K.
\]
Thus $\beta>1$ controls the scaled spectral origin, whereas $\beta<2$
ensures that the Gaussian reference is an $H_s$-valued law.
For a spectral compression, with $R=2tL_N\ge1$,
$0\le t^\beta(\mathcal A_0-\mathcal A_{0,N})\le CR^\beta e^{-2R}\to0$
by \eqref{common:large-tail}.
\end{proof}

\begin{remark}[Reading the constant]\label{common:replica}
By \eqref{common:gaussian-sum}, $\zeta(\beta)-1=\sum_{n\ge2}n^{-\beta}$
is the Dirichlet series of the heat-trace times $2nt$; the subtracted
term is the absent time $n=1$. Each trace contributes
$C_A\alpha\Gamma(\beta)(2nt)^{-\beta}$ by Karamata's theorem, which
is a second proof of \eqref{common:gaussian-limit}.
\end{remark}

\begin{theorem}[Relative derivative and leading action]
\label{common:principle}
Assume \eqref{common:weyl}. Let
$\rho\in AC^2_{\rm loc}((0,t_0];\mathcal P_2(H_s))$, $t_0\le1$.
Suppose that, almost everywhere in time and on a countable cylinder
test core,
\begin{equation}\label{common:dual-defect}
 \left|\frac{d}{dt}\int\varphi\,d\rho_t
       -\int\langle\nabla\varphi,B_tx\rangle\,d\rho_t\right|
 \le e(t)\|\nabla\varphi\|_{L^2(\rho_t)}.
\end{equation}
At each time satisfying \eqref{common:dual-defect}, the minimum current is
$v_t(x)=B_tx+r_t(x)$, where
$r_t\in T_{\rho_t}$ and $\|r_t\|_{L^2(\rho_t)}\le e(t)$, and
\begin{equation}\label{common:speed-comparison}
 \left|\sqrt{\mathcal A_\rho(t)}-\sqrt{\mathcal A_0(t)}\right|
 \le e(t)+\frac{W_2(\rho_t,\gamma_t)}{2t}.
\end{equation}
In particular, if
\begin{equation}\label{common:inputs}
 W_2(\rho_t,\gamma_t)\le Ct^\theta,\qquad
 e(t)\le Ct^{\theta-1},\qquad
 \psi:=\theta-1+\frac\beta2>0,
\end{equation}
then
\begin{equation}\label{common:calibration}
 t^\beta|\mathcal A_\rho(t)-\mathcal A_0(t)|
 \le C(t^\psi+t^{2\psi}),\qquad
 \esslim_{t\downarrow0}t^\beta\mathcal A_\rho(t)=K.
\end{equation}
If the derivative pairings and \eqref{common:inputs} hold at every time,
the limit in \eqref{common:calibration} is ordinary.
For spectral projections as in Proposition~\ref{common:gaussian}, replace
$\gamma_t,\mathcal A_0$ by $\gamma_{N,t},\mathcal A_{0,N}$.
If $C$ in \eqref{common:inputs} is independent of $N$, so are the comparison constants.
The limit is $K$ when $tL_N\to\infty$.
\end{theorem}

\begin{proof}
\emph{The relative current.}
Let $P_m\to I$ be finite-rank orthogonal projections and
$\chi\in C_c^\infty(\R)$ equal one near zero. For every second-moment law,
\[
 \varphi_{m,R}(x)=\tfrac12\chi(\|P_mx\|^2/R^2)
                    \langle B_tP_mx,P_mx\rangle,
 \qquad
 \nabla\varphi_{m,R}\xrightarrow[R\to\infty]{L^2(\rho)}P_mB_tP_mx
 \xrightarrow[m\to\infty]{L^2(\rho)}B_tx,
\]
since $|\nabla\varphi_{m,R}(x)|\le C\|B_t\|\|x\|$.
Hence $B_tx\in T_\rho$. Riesz applied to \eqref{common:dual-defect} gives
$r_t\in T_{\rho_t}$, $\|r_t\|_2\le e(t)$, and
\[
 \frac d{dt}\int\varphi\,d\rho_t
 =\int\langle B_tx+r_t,\nabla\varphi\rangle\,d\rho_t.
\]
Thus $B_tx+r_t$ is admissible and tangent, hence minimal. The exact
relative dual norm is $\|r_t\|_2$; $e(t)$ is an upper bound.

\emph{Comparison of the actions.}
By Proposition~\ref{common:gaussian}, $B_tx$ is the Gaussian minimum
current. For every coupling $(X,Z)$ of $(\rho_t,\gamma_t)$, the
triangle inequality gives
\begin{equation}\label{common:coupled-comparison}
 \left|\sqrt{\mathcal A_\rho(t)}-\sqrt{\mathcal A_0(t)}\right|
 \le\|r_t\|_{L^2(\rho_t)}+\|B_t(X-Z)\|_{L^2(\Omega;H_s)}.
\end{equation}
Taking an optimal coupling and $\|B_t\|\le(2t)^{-1}$ proves
\eqref{common:speed-comparison}.

\emph{Conclusion.}
By \eqref{common:speed-comparison}, \eqref{common:inputs} and
\eqref{common:gaussian-limit},
\[
 t^\beta|\mathcal A_\rho-\mathcal A_0|
 \le Ct^\psi\bigl(2\sqrt{t^\beta\mathcal A_0}+Ct^\psi\bigr)
 \le C(t^\psi+t^{2\psi}).
\]
\end{proof}

The two applications use the following exponents:
\[
\begin{array}{c|ccccc}
 &\alpha&s&\beta&\theta&\psi\\\hline
 \mathrm{KPZ}\text{ in }L^2_0(\mathbb T)
 &1/2&0&3/2&1/2&1/4\\
 \Phi^4_3\text{ in }H^{-3/4}(\mathbb T^3)
 &3/2&3/4&7/4&1/4-\varpi&1/8-\varpi
\end{array}
\]
Here $0<\varpi<1/8$. For KPZ, the second-order identity below improves
the gain from $1/4$ to $1/2$ (Theorem~\ref{kpz:action-limit}).

\subsection{Discarding variations that leave the law fixed}\label{common:ward-section}

A realization can change while its law stays fixed. Estimating its full
derivative may therefore retain irrelevant directions.
For a realization $Y$ and a
square-integrable variation $V$ define
\[
 \|V\|_{\mathrm{obs},Y}
 =\sup_{\|\nabla\varphi(Y)\|_2\le1}
       |\E\langle V,\nabla\varphi(Y)\rangle|.
\]
This seminorm is the norm of the tangent projection \cite{AGS}.
It vanishes exactly when $V$ leaves the law unchanged to first order.
At a spectral cutoff, let $v'=Lv+F$, $v(0)=a$, and subtract a terminal
corrector $c$. The observed variation is the complete balance
\[
 v(T)-c=J(T,0)a+\int_0^TJ(T,s)F_s\dd s-c.
\]
For any differentiable path $b$, variation of constants gives
$J(T,0)b(0)+\int_0^TJ(T,s)(b'-Lb)\dd s-b(T)=0$.
Thus the initial, source and terminal contributions of a coordinate
change cancel together. The second cancellation comes from an
orthogonal rotation of Gaussian noise.

Let $W$ be isonormal over a real Hilbert space $\mathcal K$.
For orthogonal $O$, the map $\Gamma(O)I_n(f)=I_n(O^{\otimes n}f)$
preserves joint laws and products. For bounded skew-adjoint $\mathsf A$,
write $\mathscr W_{\mathsf A}F=\partial_\varepsilon
\Gamma(e^{\varepsilon\mathsf A})F|_{\varepsilon=0}$ when this derivative
exists strongly in $L^2$.

\begin{lemma}[Gaussian invariance]\label{common:ward}
Let $Y_t$ be a locally $C^1$ realization in $L^2(\Omega;H_s)$ of $\rho_t$,
and suppose $\mathscr W_{\mathsf A_t}Y_t$ exists for deterministic bounded
skew-adjoint $\mathsf A_t$. Then \eqref{common:dual-defect} holds with
\begin{equation}\label{common:ward-defect}
 D_t=t\partial_tY_t-tB_tY_t+\mathscr W_{\mathsf A_t}Y_t,
 \qquad e(t)=t^{-1}\|D_t\|_{\mathrm{obs},Y_t}\le t^{-1}\|D_t\|_2.
\end{equation}
\end{lemma}
\begin{proof}
For a smooth bounded cylinder test, invariance and the chain rule give
\[
 0=\left.\partial_\varepsilon\mathbb E
    \varphi(\Gamma(e^{\varepsilon\mathsf A_t})Y_t)\right|_0
   =\mathbb E\langle\nabla\varphi(Y_t),\mathscr W_{\mathsf A_t}Y_t\rangle.
\]
Consequently
\[
 t\partial_t\E\varphi(Y_t)-t\E\langle B_tY_t,\nabla\varphi(Y_t)\rangle
 =\E\langle D_t,\nabla\varphi(Y_t)\rangle,
 \qquad \|D_t\|_{\mathrm{obs},Y_t}\le\|D_t\|_2.
\]
\end{proof}

For the Gaussian number operator $\mathcal NI_n=nI_n$, one has
\begin{equation}\label{common:ward-number}
 F\in\operatorname{Dom}(\mathcal N)\ \Longrightarrow\
 \mathscr W_{\mathsf A}F=d\Gamma(\mathsf A)F,
 \qquad\|d\Gamma(\mathsf A)F\|_2\le\|\mathsf A\|\,\|\mathcal NF\|_2.
\end{equation}
Indeed, for $F=\sum_{n\ge0}I_n(f_n)$,
\[
 d\Gamma(\mathsf A)I_n(f_n)=I_n\Bigl(\sum_{k=1}^n\mathsf A^{(k)}f_n\Bigr),
 \qquad
 \|d\Gamma(\mathsf A)F\|_2^2
 \le\|\mathsf A\|^2\sum_{n\ge0}n^2n!\|f_n\|^2.
\]
Finite-chaos truncation proves \eqref{common:ward-number}.
Appendix~\ref{ma:section} constructs $\mathsf A_t$ and differentiates
the finite-cutoff solution along its rotations.

\paragraph{An exact second-order comparison.}
Since $B_tZ_t$ lies in the first chaos, the linear error depends only
on $u_t^{(1)}$ and $D_t^{(1)}$.

\begin{proposition}[Second-order comparison]\label{out:second}
Let $Y_t=Z_t+u_t$ be a realization as in Lemma~\ref{common:ward}, where
$Z_t$ is the first-chaos realization of $\gamma_t$ and the rotation
$\mathsf A_t$ satisfies $(t\partial_t+\mathscr W_{\mathsf A_t})Z_t=tB_tZ_t$.
Let $D_t$ be the defect \eqref{common:ward-defect}, $r_t$ the residual
current of Theorem~\ref{common:principle}, and write $F^{(1)}$ for the
first-chaos component of $F$. Then
\begin{equation}\label{out:second-identity}
\begin{split}
 \mathcal A_\rho(t)-\mathcal A_0(t)
 ={}&2\E\bigl\langle B_tZ_t,\,B_tu_t^{(1)}+t^{-1}D_t^{(1)}\bigr\rangle\\
 &+\E\|B_tu_t\|^2+2t^{-1}\E\langle B_tu_t,D_t\rangle
 +\|r_t\|_{L^2(\rho_t)}^2,
\end{split}
\end{equation}
with $\|r_t\|_{L^2(\rho_t)}\le t^{-1}\|D_t\|_{L^2}$.
\end{proposition}
\begin{proof}
Testing the relative current against $B_tx\in T_{\rho_t}$ gives
$\E\langle B_tY_t,r_t(Y_t)\rangle=t^{-1}\E\langle B_tY_t,D_t\rangle$.
Hence
\[
 \mathcal A_\rho-\mathcal A_0
 =\E\langle B_t(Y_t-Z_t),B_t(Y_t+Z_t)\rangle
   +2t^{-1}\E\langle B_tY_t,D_t\rangle+\|r_t\|_2^2.
\]
Substitute $Y_t=Z_t+u_t$ and use
$\E\langle B_tZ_t,F\rangle=\E\langle B_tZ_t,F^{(1)}\rangle$.
This is \eqref{out:second-identity}.
\end{proof}

Since $\|B_t\|\le(2t)^{-1}$, Cauchy--Schwarz in
\eqref{out:second-identity} gives
\begin{equation}\label{out:second-bound}
\begin{split}
 |\mathcal A_\rho(t)-\mathcal A_0(t)|
 \le{}&\frac{C\sqrt{\mathcal A_0(t)}}{t}
       \bigl(\|u_t^{(1)}\|_2+\|D_t^{(1)}\|_2\bigr)\\
 &+\frac{C}{t^2}\bigl(\|u_t\|_2+\|D_t\|_2\bigr)^2,
\end{split}
\end{equation}
where the norms are in $L^2(\Omega;H_s)$.

\begin{corollary}[Length, action and excess length]\label{common:integrated}
Under the hypotheses of Theorem~\ref{common:principle}, including
\eqref{common:inputs}, if
$\rho_t\to\delta_0$ in $W_2$ as $t\downarrow0$, then the curve extends
to $AC^1$ at zero and, for $c>1$,
\begin{equation}\label{common:length}
 \operatorname{Length}(\rho|_{[0,t]})
 \sim\frac{\sqrt K}{1-\beta/2}t^{1-\beta/2},\qquad
 t^{\beta-1}\int_t^{ct}\mathcal A_\rho(u)\,du
 \longrightarrow K\frac{1-c^{1-\beta}}{\beta-1}.
\end{equation}
Moreover
\begin{equation}\label{common:distance}
 W_2(\rho_t,\delta_0)^2\sim
 \frac{aC_A\alpha\,\Gamma(\beta-1)}{2-\beta}(2t)^{2-\beta},
 \qquad
 \lim_{t\downarrow0}
 \frac{\operatorname{Length}(\rho|_{[0,t]})}{W_2(\rho_t,\delta_0)}
 =\sqrt{\frac{(\beta-1)(\zeta(\beta)-1)}{2-\beta}}.
\end{equation}
\end{corollary}
\begin{proof}
Since $\sqrt{\mathcal A_\rho(t)}\le Ct^{-\beta/2}\in L^1(0,t_0)$,
the curve extends to $AC^1$. Substitution $u=tv$ and dominated convergence give
\[
 \frac{\operatorname{Length}(\rho|_{[0,t]})}{t^{1-\beta/2}}
 =\int_0^1\sqrt{t^\beta\mathcal A_\rho(tv)}\,dv
 \longrightarrow\sqrt K\int_0^1v^{-\beta/2}\,dv,
 \qquad
 t^{\beta-1}\int_t^{ct}\mathcal A_\rho(u)\,du
 \longrightarrow K\int_1^cv^{-\beta}\,dv.
\]

For the distance, $W_2(\gamma_t,\delta_0)^2=\E\|Z_t\|_{H_s}^2$. With
the measures $m_t$ of Proposition~\ref{common:gaussian},
\[
 \E\|Z_t\|_{H_s}^2
 =a\sum_j\lambda_j^{-1-s}(1-e^{-2t\lambda_j})
 =a(2t)^{2-\beta}\int x^{-1-s}(1-e^{-x})\,dm_t(x).
\]
The small-$x$ bound is \eqref{common:small-tail}; the large-$x$ tail is
bounded by $C\sum_{j\ge0}(2^jR)^{\beta-2}\le CR^{\beta-2}$.
Vague convergence and integration by parts give
\[
 \int_0^\infty x^{\beta-3}(1-e^{-x})\,dx
 =\frac1{2-\beta}\int_0^\infty x^{\beta-2}e^{-x}\,dx
 =\frac{\Gamma(\beta-1)}{2-\beta},
\]
which is the first relation in \eqref{common:distance} for $\gamma_t$.
The triangle inequality gives
$|W_2(\rho_t,\delta_0)-W_2(\gamma_t,\delta_0)|\le Ct^\theta$, and
$t^\theta=o(t^{1-\beta/2})$ because $\psi>0$.
Dividing the two asymptotics and using $\Gamma(\beta)=(\beta-1)\Gamma(\beta-1)$,
\[
 \frac{K}{(1-\beta/2)^2}\cdot
 \frac{2-\beta}{aC_A\alpha\,\Gamma(\beta-1)2^{2-\beta}}
 =\frac{(\beta-1)(\zeta(\beta)-1)}{2-\beta}.
\]
This proves \eqref{common:distance}.
\end{proof}

\begin{remark}[Excess length]\label{common:excess}
The ratio in \eqref{common:distance} depends only on $\beta$. It tends
to $1$ as $\beta\downarrow1$ and to infinity as $\beta\uparrow2$.
For every fixed $1<\beta<2$, it is strictly greater than one, since
$\zeta(\beta)>\int_1^\infty x^{-\beta}\dd x=(\beta-1)^{-1}$.
Its values are $\sqrt{\zeta(3/2)-1}\approx1.270$ for KPZ and
$\sqrt{3(\zeta(7/4)-1)}\approx1.699$ for $\Phi^4_3$ in $H^{-3/4}$.
The nonlinearity does not change it, because both relations in
\eqref{common:distance} are governed by the Gaussian reference at
leading order.
\end{remark}
\subsection{Closure of relative currents}

\begin{lemma}[Closure of relative currents]
\label{full:closure}
Let $H$ be a separable Hilbert space, $I\Subset(0,\infty)$ a compact
interval, and $B_t$ a norm-continuous family of bounded self-adjoint
operators on $H$. Suppose $\mu_n\to\mu$ in $C(I;\mathcal P_2(H))$.
Assume that $\mu_n$ is locally absolutely continuous and has an
admissible current $B_tx+r_n(t,x)$ satisfying
\[
 \|r_n(t)\|_{L^2(\mu_n(t);H)}\le a(t)
 \quad\text{for a.e. }t\in I,\qquad a\in L^2(I).
\]
Then $\mu\in AC^2(I;\mathcal P_2(H))$ and its minimum current is
\begin{equation}\label{full:relative-current}
 v_t(x)=B_tx+r_t(x),\qquad r_t\in T_{\mu_t},\qquad
 \|r_t\|_{L^2(\mu_t;H)}\le a(t)\quad\text{a.e.},
\end{equation}
where $T_{\mu_t}$ is the $L^2(\mu_t;H)$ closure of smooth bounded
cylinder gradients. In particular $|\dot\mu_t|^2=\|v_t\|_{L^2(\mu_t)}^2$
for almost every $t$.
\end{lemma}
\begin{proof}
For a smooth cylinder test $\Phi(t,x)$ compactly supported in the
interior of $I$, the continuity equation gives
\[
 \left|\int_I\!\int_H
 [\partial_t\Phi+\langle B_tx,\nabla\Phi\rangle]\,d\mu_n(t)\,dt\right|
 \le\int_I a(t)\|\nabla\Phi(t)\|_{L^2(\mu_n(t))}\,dt.
\]
For tests with bounded first and second spatial derivatives, an optimal
coupling $(X_n,X)$ gives, uniformly in $t\in I$,
\[
 \begin{split}
 \big|\E\langle B_tX_n,\nabla\Phi(X_n)\rangle
       -\E\langle B_tX,\nabla\Phi(X)\rangle\big|
 &\le C(1+\|X\|_2)W_2(\mu_n,\mu),\\
 \big|\|\nabla\Phi\|_{L^2(\mu_n)}-\|\nabla\Phi\|_{L^2(\mu)}\big|
 &\le\|\nabla^2\Phi\|_\infty W_2(\mu_n,\mu).
 \end{split}
\]
Thus the negative limiting integral defines a functional $F$ on the
closure $\mathscr G$ of these gradients in $L^2(dt\,d\mu_t;H)$, with
\begin{equation}\label{full:dual-bound}
 |F(g)|\le\int_I a(t)\|g_t\|_{L^2(\mu_t)}\,dt.
\end{equation}
Riesz gives a unique $r\in\mathscr G$ representing $F$.
Choose gradients $g_k$ with $\|g_k-r\|_{L^2(dt\,d\mu_t)}^2\le2^{-k}$.
Then $\sum_k\|g_k(t)-r_t\|_{L^2(\mu_t)}^2<\infty$ a.e., so
$r_t\in T_{\mu_t}$.
The space $\mathscr G$ is stable under smooth time multipliers.
Apply \eqref{full:dual-bound}, by density, to $g=\chi r$ with
$\chi\ge0$ smooth and compactly supported. It gives
\[
 \int_I\chi(t)\|r_t\|_{L^2(\mu_t)}^2\,dt
 \le\int_I\chi(t)a(t)\|r_t\|_{L^2(\mu_t)}\,dt,
\]
which proves the pointwise bound in \eqref{full:relative-current}.
The defining identity for $F$ is the limiting continuity equation;
time localization gives its derivative pairings on a countable cylinder
core for almost every $t$.

Put $b_n(t)=\|B_tx\|_{L^2(\mu_n(t))}$ and
$b(t)=\|B_tx\|_{L^2(\mu_t)}$. The inequality
$|b_n-b|\le\|B_t\|W_2(\mu_n(t),\mu_t)$ gives uniform convergence.
The admissible-current bound implies
$W_2(\mu_n(s),\mu_n(t))\le\int_s^t(b_n+a)$.
Passing to the limit proves $\mu\in AC^2$.
The quadratic approximation in Theorem~\ref{common:principle} gives
$B_tx\in T_{\mu_t}$. Thus $B_tx+r_t$ is tangent and minimal, and
$|\dot\mu_t|^2=\|B_tx+r_t\|_{L^2(\mu_t)}^2$ a.e. \cite{AGS}.
\end{proof}

\section{The centered KPZ profile}\label{kpz:application}

Let $\mathbb T=\mathbb R/\mathbb Z$, $H=L^2_0(\mathbb T;\mathbb R)$,
and $\Pi_\circ f=f-\int_{\mathbb T}f$. Fix $\nu>0$ and $\sigma\in\mathbb R\setminus\{0\}$.
For $g\ne0$, put $\kappa=g/(2\nu)$ and define the flat-start Cole--Hopf solution by
\[
 dZ_\kappa=\nu\partial_x^2Z_\kappa\,dt+\kappa\sigma Z_\kappa\,dW,
 \qquad Z_\kappa(0)=1,\qquad h_g=\kappa^{-1}\log Z_\kappa.
\]
Thus $h_g$ solves the renormalized KPZ equation with nonlinear coefficient
$g/2$. Its spatial centering removes the additive renormalization constant.
Write $\mu_{g,t}=\Law(\Pi_\circ h_g(t))$; at $g=0$ use the additive
stochastic heat equation. Subtracting the linear term of $\log Z_\kappa$ gives the nonlinear
remainder. Its quadratic term has no first chaos, so its pairing with
the Gaussian current vanishes in Proposition~\ref{out:second}.

\begin{theorem}\label{kpz:action-limit}
For every $G<\infty$ there are $C_G,t_0>0$ such that the curves
$\mu_{g,t}$, $|g|\le G$, are locally $AC^2$ on $(0,t_0]$.
Their actions satisfy, for every $0<t\le t_0$,
\[
 t^{3/2}|\mathcal A_g(t)-\mathcal A_0(t)|
 \le C_Gg^2t^{1/2},\qquad
 \lim_{t\downarrow0}t^{3/2}\mathcal A_g(t)
 =\frac{\sigma^2[\zeta(3/2)-1]}{16\sqrt{2\pi\nu}}.
\]
\end{theorem}

\begin{proof}
Appendix~\ref{kpz:analytic} constructs locally $C^1$ realizations
$Y_{g,t}=X_t^0+R_{g,t}$ on one Gaussian space, with $X_t^0$ Gaussian.
Lemma~\ref{common:ward} gives a current at every positive time and local
absolute continuity. Let $D_{g,t}$ be the defect in
\eqref{common:ward-defect}, and $F^{(1)}$ the first-chaos component of $F$.
Equations~\eqref{kpz:required-defect} and~\eqref{kpz:first-chaos-defect} give
\[
 \|R_{g,t}\|_2+\|D_{g,t}\|_2\le C_G|g|t^{1/2},\qquad
 \|R_{g,t}^{(1)}\|_2+\|D_{g,t}^{(1)}\|_2\le C_Gg^2t^{3/4}.
\]
Here $B_t=A(e^{2tA}-I)^{-1}$, $A=-\nu\partial_x^2$ on $H$, so
$\|B_t\|\le(2t)^{-1}$. Proposition~\ref{common:gaussian} gives
$\|B_tX_t^0\|_2\le Ct^{-3/4}$. Hence \eqref{out:second-bound} gives
\begin{align*}
 |\mathcal A_g(t)-\mathcal A_0(t)|
 &\le C t^{-7/4}\bigl(\|R_{g,t}^{(1)}\|_2+\|D_{g,t}^{(1)}\|_2\bigr)
       +C t^{-2}\bigl(\|R_{g,t}\|_2+\|D_{g,t}\|_2\bigr)^2\\
 &\le C_Gg^2t^{-1}.
\end{align*}
Multiplication by $t^{3/2}$ proves the comparison.

Counting both real modes at each positive frequency gives
$N_A(L)\sim L^{1/2}/(\pi\sqrt\nu)$. Formula~\eqref{common:gaussian-limit},
with $s=0$ and $a=\sigma^2/2$, gives the stated constant.
\end{proof}

\begin{remark}[Parity in the coupling]\label{kpz:parity}
The action is even in $g$. Indeed $Z_{-\kappa}(W)=Z_\kappa(-W)$, so
$h_{-g}(W)=-h_g(-W)$. Since $-W$ has the law of $W$ and $x\mapsto-x$ is
an isometry of $H$, the curves $\mu_{-g,t}$ and $\mu_{g,t}$ are
isometric images of each other, and $\mathcal A_{-g}=\mathcal A_g$.
\end{remark}

\begin{remark}[The Gaussian action to all orders]\label{kpz:gaussian-expansion}
For KPZ the normalized Gaussian action has an expansion with a single correction:
\begin{equation}\label{kpz:two-term}
 t^{3/2}\mathcal A_0(t)
 =\frac{\sigma^2[\zeta(3/2)-1]}{16\sqrt{2\pi\nu}}
 -\frac{\sigma^2}{4}t^{1/2}+O(t^M)\qquad\text{for every }M.
\end{equation}
Put $a=\sigma^2/2$, $\phi(z)=ze^{-4z}/(1-e^{-2z})$, and
$f(x)=\phi(4\pi^2\nu x^2)$. The function $f$ is even and Schwartz,
with $f(0)=1/2$. The spectral formula gives
$t^{3/2}\mathcal A_0(t)=a\sqrt t\sum_{k\ne0}f(\sqrt t\,k)$.
Poisson summation gives
\[
 \sum_{k\ne0}f(\sqrt t\,k)
 =t^{-1/2}\widehat f(0)-\tfrac12
   +t^{-1/2}\sum_{m\ne0}\widehat f(m/\sqrt t),\qquad
 t^{-1/2}\sum_{m\ne0}|\widehat f(m/\sqrt t)|
 \le C_Lt^{(L-1)/2}\sum_{m\ne0}|m|^{-L}.
\]
Take $L>1$ arbitrarily large.
Multiply by $a\sqrt t$ and use Proposition~\ref{common:gaussian}.
The correction $-\sigma^2\sqrt t/4$ is the removed zero mode.

The nonlinear bound $O_G(g^2t^{1/2})$ has the same time order as this
Gaussian correction. Its limiting coefficient remains open
(Section~\ref{se:outlook}).
\end{remark}

\section{The dynamical \texorpdfstring{$\Phi^4_3$}{Phi4 3} model}
\label{app:phi4}

On the unit torus $\mathbb T^3$, fix $a,\mu,\nu,g>0$ and let
$A=\mu-\nu\Delta$, $\lambda_k=\mu+4\pi^2\nu|k|^2$.
Let $P_N$ project onto $\Lambda_N=\{|k|_\infty\le N\}$, put
$E_N=\operatorname{Ran}P_N$ as a real vector space, and let
$W_N=P_NW$ for one cylindrical Brownian motion. The approximations are
\begin{equation}\label{eq:model}
 \begin{aligned}
 dX_N&=-(AX_N+F_N(X_N))\,dt+\sqrt{2a}\,dW_N,\\
 dZ_N&=-AZ_N\,dt+\sqrt{2a}\,dW_N,
 \end{aligned}
 \qquad X_N(0)=Z_N(0)=0,
\end{equation}
where $F_N(x)=gP_Nx^3+(-3gC_{1,N}+9g^2C_{2,N})x$ and
\begin{equation}\label{eq:constants}
 C_{1,N}=a\sum_{k\in\Lambda_N}\lambda_k^{-1},\qquad
 C_{2,N}=2a^2\!
 \sum_{\substack{p,q,p+q\in\Lambda_N}}
 \frac1{\lambda_p\lambda_q(\lambda_p+\lambda_q+\lambda_{p+q})}.
\end{equation}
Write $u_N=X_N-Z_N$, $\rho_{g,N,t}=\operatorname{Law}(X_N(t))$ and
$\rho_{0,N,t}=\operatorname{Law}(Z_N(t))$.
We use $H=H_A^{-3/4}$, with
$\|f\|_{H_A^r}^2=\sum_k\lambda_k^r|\widehat f(k)|^2$.

Write $\mathcal A_{g,N}$ for the action of $\rho_{g,N}$, and let $e_N(t)$
be the smallest constant in \eqref{common:dual-defect} for this curve.

Let $X=\lim_{N\to\infty}X_N$ as in Proposition~\ref{full:identification},
and set $\rho_{g,t}=\operatorname{Law}(X(t))$.

\begin{theorem}[$\Phi^4_3$ action]\label{th:velocity}
For every $0<\varpi<1/8$, there are $t_0,C>0$, depending only on
$a,\mu,\nu,g,\varpi$, such that
\begin{align}
 \sup_N e_N(t)^2&\le Ct^{-3/2-2\varpi},\label{eq:pointrelative}\\
 \sup_N t^{7/4}|\mathcal A_{g,N}(t)-\mathcal A_{0,N}(t)|
 &\le Ct^{1/8-\varpi}\qquad(0<t\le t_0).\label{eq:speedcomparison}
\end{align}
The full law curve lies in $AC^1([0,t_0];\mathcal P_2(H))$ and in
$AC^2_{\mathrm{loc}}((0,t_0];\mathcal P_2(H))$. Its action satisfies
\begin{equation}\label{eq:full-comparison}
 t^{7/4}|\mathcal A_g(t)-\mathcal A_0(t)|\le Ct^{1/8-\varpi}
 \qquad(0<t\le t_0).
\end{equation}
Consequently
\begin{equation}\label{eq:full-limit}
 \lim_{t\downarrow0}
 t^{7/4}\mathcal A_g(t)
 =K(a,\nu):=
 \frac{a\Gamma(7/4)[\zeta(7/4)-1]}
 {2^{15/4}\pi^2\nu^{3/2}}.
\end{equation}
Whenever $N(t)\sqrt t\to\infty$,
\begin{equation}\label{eq:universal}
 t^{7/4}\mathcal A_{g,N(t)}(t)\longrightarrow K(a,\nu).
\end{equation}
For each fixed $N$, the normalized action instead satisfies
$t^{7/4}\mathcal A_{g,N}(t)\to0$ as $t\downarrow0$.
\end{theorem}

The proof establishes the two bounds in Theorem~\ref{common:principle}:
\[
 \|u_N(t)\|_{L^2(\Omega;H)}\le Ct^\theta,\qquad
 e_N(t)\le Ct^{\theta-1},\qquad \theta=\tfrac14-\varpi.
\]
\paragraph{Conventions for the quartic estimates.}

Constants are independent of $N$ unless otherwise stated. A polynomial
bound in a model norm $M\ge1$ means $C(1+M)^b$ for fixed finite $C,b$.
We distinguish $L^p(\Omega;H)$ from time-space norms such as $L^p(0,T;H)$.
The spaces $\mathcal C^\alpha$ mean $B^\alpha_{\infty,\infty}$.
For a fixed smooth dyadic decomposition, $f\prec h$ denotes the low--high
paraproduct, $f\circ h$ the resonant product, and
$fh=f\prec h+h\prec f+f\circ h$. The high-frequency operator $H_J$
retains dyadic blocks with index at least $J$; $L_J=I-H_J$.
The analytic estimates below use only finite-$p$ bounds for $P_N$.
Fix the desired loss $\varpi\in(0,1/8)$. Rough coefficients require
regularity just below $-1/2$; their tests require regularity just above
$1/2$. Choose a small $\epsilon_0>0$ and a large finite time exponent $q$:

\begin{equation}\label{eq:indices}
\begin{gathered}
 0<\epsilon_0<\varpi/8,\qquad q^{-1}<\varpi/4,\qquad
 \alpha=\tfrac12+\epsilon_0,\quad \sigma=\tfrac12+2\epsilon_0,\\
 \beta=1-\epsilon_0/2,\quad r_*=1-\epsilon_0,\quad
 \gamma=\tfrac14-2\epsilon_0,\quad \theta=\tfrac14-\varpi,\quad
 \omega=\tfrac14+\varpi.
\end{gathered}
\end{equation}
Coefficients have regularity $-\alpha$, tests $\sigma$, and weights
$\beta$. The response has regularity $r_*<\beta$; the state gains
$\gamma>\theta$. The exponent $\omega$ is used in the local heat estimate.
These choices give $\varpi+2\theta<1/2$, $\theta>1/8$ and
$q>32>4/[\sigma(1-\sigma)]$.

In this section and Appendices~\ref{ga:section}--\ref{cal:projection-section},
$\alpha,\beta,\sigma$ always denote the regularity indices in
\eqref{eq:indices}. We write the spectral exponents as numbers:
Weyl exponent $3/2$, transport weight $3/4$, action exponent $7/4$.

The coordinate removes the cubic Gaussian source. The weights cancel
the rough potential in the energy identity. These two cancellations
control the state and its derivative. Appendix~\ref{po:section}
identifies the limiting field; the notation index is at the end.

\subsection{One entropy identity, two estimates}\label{le:section}

The rough potential is the obstruction common to the state and its
response. We remove it from the energy identity before estimating
the remaining terms. This is generalized relative entropy
\cite{MMP,ConstantinGRE}, used with approximate positive weights.
Write $P=P_N$ and $Q=I-P$.

Let $L=\nu\Delta-V$ and let $h,k>0$ satisfy
\begin{equation}\label{cal:gauges}
 h'=Lh+a_+h,\qquad -k'=Lk+a_-k.
\end{equation}
The defects $a_\pm$ measure the error in solving the forward and
backward equations. Their sum, rather than the potential $V$, enters
the quadratic energy.

\begin{lemma}[Relative entropy with defects]\label{cal:energy}
If $f'=Lf+F$, $v=f/h$, and $\Phi\in C^2(\mathbb R)$, then
\begin{equation}\label{cal:general-entropy}
\begin{split}
 \frac d{ds}\int hk\Phi(v)
 +\nu\int hk\Phi''(v)|\nabla v|^2
 ={}&\langle F,k\Phi'(v)\rangle\\
 &+\int hk\{(a_+-a_-)\Phi(v)-a_+v\Phi'(v)\}.
\end{split}
\end{equation}
In particular, with $m=k/h$,
\begin{equation}\label{cal:local-energy}
 \frac12\frac d{ds}\int mf^2+\nu\int hk|\nabla v|^2
 =\langle F,kv\rangle-\frac12\langle a_++a_-,hk v^2\rangle.
\end{equation}
For $\Phi(v)=v^4/4$, the dissipation is
$3\nu\int hk v^2|\nabla v|^2$ and the defect is
$-\langle3a_++a_-,hk v^4\rangle/4$.
If $f\in E_N$ solves $f'=\nu\Delta f-P(Vf)+P\mathcal F$, then
\begin{equation}\label{cal:projected-energy}
 \frac12\frac d{ds}\int mf^2+\nu\int hk|\nabla(f/h)|^2
 =\langle Q(mf),Vf\rangle+\langle P(mf),\mathcal F\rangle
 -\frac12\langle a_++a_-,hk(f/h)^2\rangle.
\end{equation}
The backward identity interchanges $h,k$.
\end{lemma}
\begin{proof}
The quotient satisfies
$v'=\nu\Delta v+2\nu\nabla\log h\cdot\nabla v-a_+v+F/h$.
Also
$(hk)'=\nu\operatorname{div}(k\nabla h-h\nabla k)
 +(a_+-a_-)hk$.
Integration by parts gives \eqref{cal:general-entropy}, since the
coefficient of $\nu\Phi'(v)\nabla v$ is
\[
 -(k\nabla h+h\nabla k)+2k\nabla h-(k\nabla h-h\nabla k)=0.
\]
For the projected equation,
$\nu\Delta f-P(Vf)=Lf+Q(Vf)$ and
$\langle Q(Vf),mf\rangle=\langle Vf,Q(mf)\rangle$.
\end{proof}

Put $\eta=3g$ and $W=Z_N^2-\E Z_N^2$.
The weights will be $h=e^{-\eta H_JY_+}$ and $k=e^{-\eta H_JY_-}$, with
$Y_\pm$ forward and backward heat primitives of the Wick square
(Appendix~\ref{ac:reference-section}). The chain rule gives
\[
 \frac{(\partial_t-\nu\Delta)h}{h}
 =-\eta(\partial_t-\nu\Delta)H_JY_+-\eta^2\nu|\nabla H_JY_+|^2 .
\]
The first term cancels the rough potential $\eta W$ at high frequency.
The second is a positive gradient square. Its mean is the divergent
second-order counterterm, up to bounded terms:

\begin{proposition}[The sunset counterterm is a chain-rule correction]
\label{cal:sunset}
Let $Z$ be the stationary Ornstein--Uhlenbeck field with noise cutoff
$\Lambda_N$, $W=Z^2-\E Z^2$, and
$Y_{\rm stat}(t)=\int_{-\infty}^te^{-(t-r)A}W(r)\dd r$. Put
$d_{pq}=\lambda_p+\lambda_q+\lambda_{p+q}$ and
\[
 C_{2,N}^{\rm full}=2a^2\sum_{p,q\in\Lambda_N}
 \frac1{\lambda_p\lambda_q d_{pq}}.
\]
Then
\begin{equation}\label{cal:sunset-identity}
 \nu\,\E|\nabla Y_{\rm stat}|^2
 =C_{2,N}^{\rm full}-\mu\,\E|Y_{\rm stat}|^2,
 \qquad
 \sup_N\bigl(\E|Y_{\rm stat}|^2+|C_{2,N}^{\rm full}-C_{2,N}|\bigr)<\infty.
\end{equation}
\end{proposition}
\begin{proof}
With $k=p+q$, the Wick covariance and the heat integrals give
\[
\begin{aligned}
 \E[\widehat W(t,k)\overline{\widehat W(s,k)}]
 &=2a^2\sum_{\substack{p+q=k\\p,q\in\Lambda_N}}
       \frac{e^{-(\lambda_p+\lambda_q)|t-s|}}{\lambda_p\lambda_q},\\
 \E|\widehat Y_{\rm stat}(k)|^2
 &=2a^2\sum_{\substack{p+q=k\\p,q\in\Lambda_N}}
 \frac{\int_0^\infty\!\int_0^\infty
 e^{-\lambda_k(r+s)- (\lambda_p+\lambda_q)|r-s|}\dd r\dd s}
      {\lambda_p\lambda_q}\\
 &=2a^2\sum_{\substack{p+q=k\\p,q\in\Lambda_N}}
 \frac1{\lambda_p\lambda_q\lambda_kd_{pq}}.
\end{aligned}
\]
Summation against $\lambda_k-\mu=4\pi^2\nu|k|^2$ proves the identity.
Moreover, \eqref{gm:heat} and \eqref{ga:full-correction} give
\[
 \E|Y_{\rm stat}|^2\lesssim\sum_k\langle k\rangle^{-5}<\infty,
 \qquad
 |C_{2,N}^{\rm full}-C_{2,N}|
 \lesssim\sum_{\substack{p,q\in\Lambda_N\\
              |p|_\infty\vee|q|_\infty>N/2}}
       \frac1{\lambda_p\lambda_qd_{pq}}\lesssim1.
\]
\end{proof}

\begin{remark}[The same mechanism in KPZ]\label{cal:sunset-kpz}
For KPZ, $h=\kappa^{-1}\log Z_\kappa$ and It\^o's formula for the
logarithm produce the Wick renormalization of $(\partial_xh)^2$. In
both models the counterterm is the correction generated by an
exponential of a heat primitive: of the solution itself for KPZ, of the
Gaussian potential for $\Phi^4_3$. The relative entropy of
Lemma~\ref{cal:energy} uses this correction to cancel the rough
potential before any estimate is made.
\end{remark}

\subsection{The state: a corrector, two entropies, and heat regularity}
\label{co:state-roadmap}

Write $u=X_N-Z_N$. Its equation contains a cubic Gaussian source and
a rough quadratic potential. The source must be removed before an
energy estimate can close. Appendix~\ref{po:section} constructs a
noise-dependent corrector $q$ and an invertible smoothing perturbation
$\mathcal T=I+K$ such that
\begin{equation}\label{co:corrected}
 r=u+q,\qquad w=\mathcal Tr=\mathcal Tu+q_0,\qquad
 (\partial_s+A)w=\mathcal R_s(w),\qquad w(0)=0.
\end{equation}
The leading part of $q$ is the heat integral of the Wick cubic.
The correction cancels that source exactly, including its prescribed
counterterm; see \eqref{state:additive-equation}.

Set $v=r/h$. The first entropy controls
\[
 E_2(T)=\sup_{s\le T}\|v(s)\|_2^2+
       \int_0^T\bigl(\|v\|_{H^1}^2+\|v\|_4^4\bigr)\dd s.
\]
Its absorption mechanism is one comparison. The hardest quadratic
pairing, after factoring out its coefficient, costs
$\|v\|_4^{3-\sigma}\|v\|_{H^1}^{\sigma}$. Its degree relative to the
two dissipations is
\[
 \frac{3-\sigma}{4}+\frac{\sigma}{2}
 =\frac{3+\sigma}{4}<1.
\]
Young's inequality absorbs it. The coefficient remainder is integrable
by the Gaussian time norms. The initial layer determines the rate:
\[
 \int_0^T s^{-1/2}\|v(s)\|_{H^\sigma}\dd s
 \lesssim T^{(1-\sigma)/2}E_2(T)^{1/2}.
\]
A second application of Young gives
$E_2(T)\le CM^CT^{1-\sigma}$; Lemma~\ref{state:quadratic} proves the
projection and coefficient bounds used here.

The second entropy controls $v^2$ in $L^\infty L^2\cap L^2H^1$
(Lemma~\ref{state:quartic}). Thus $v\in L^{20/3}$ and $u^3\in L^2$
in space-time. The equation for $w$ now has an $L^2H^{-\alpha}$ source.
Linear heat regularity and interpolation give
\begin{equation}\label{co:position-energy}
 \sup_{s\le t}\|w(s)\|_{H^{\epsilon_0}}^2+
 \int_0^t\bigl(\|w\|_{H^{1+\epsilon_0}}^2+\|u\|_4^4+\|w\|_4^4\bigr)\dd s
 \le CM^Ct^{2\gamma},\qquad t\le T_g.
\end{equation}
Here $M$ has all finite moments uniformly in the cutoff and horizon,
and Proposition~\ref{po:position} gives a deterministic $T_g>0$. Since $\|q(t)\|_2\le CM^Ct^\theta$ and
$\gamma>\theta$, reconstruction gives
\begin{equation}\label{eq:positionneeded}
 \|u(t)\|_2\le CM^C(t^\gamma+t^\theta)\le CM^Ct^\theta.
\end{equation}

\subsection{The response: retain the cubic dissipation}

Fix a reference path $X=X_N$ and write $P=P_N$, $Q=I-P$. Its physical
linearization is
\[
 L_Xf=\nu\Delta f-P(Vf),\qquad
 V=\mu+3gX^2-3gC_{1,N}+9g^2C_{2,N}.
\]
We use the same explicit weights, now with defects relative to this $V$.
Proposition~\ref{cal:construction} gives deterministic $R,T_*>0$ and
a reference event of probability at least $1/2$ on which
\begin{equation}\label{cal:reference-bounds}
 R^{-1}\le h,k\le R,\qquad
 \|h\|_{CC^\beta}+\|k\|_{CC^\beta}
 +\|X\|_{CC^{-\alpha}}+\|u\|_{CC^\varpi}\le R,
 \qquad \|k/h-1\|_\infty\le\delta_0,
\end{equation}
where $\delta_0<(2\sup_N\|P_N\|_{L^4\to L^4})^{-1}$.
This event is independent of any control introduced below.

For $m_+=k/h$, $m_-=h/k$, put $z_+=f/h$, $z_-=f/k$.
Lemma~\ref{cal:projection-form} and the defect bounds give, for every
$\varepsilon>0$,
\begin{equation}\label{cal:reference-form}
\begin{split}
 |\langle Q(m_\pm f),Vf\rangle|
 +\tfrac12|\langle a_++a_-,hkz_\pm^2\rangle|
 &\le\varepsilon\|z_\pm\|_{H^1}^2+B_\varepsilon(s)\|z_\pm\|_2^2,\\
 \int_0^T B_\varepsilon(s)\dd s&\le C_{R,\varepsilon}.
\end{split}
\end{equation}
Indeed,
\[
 |\langle a,z^2\rangle|
 \lesssim\|a\|_{C^{-\alpha}}\|z\|_2^{2-\alpha}\|z\|_{H^1}^{\alpha}
 \le\varepsilon\|z\|_{H^1}^2
     +C_\varepsilon\|a\|_{C^{-\alpha}}^{2/(2-\alpha)}\|z\|_2^2.
\]
The construction bounds this coefficient in $L^1(0,T)$, as well as
the spatially constant defects.

\begin{proposition}[Controlled states and response]\label{cal:control}
Work on the reference event of \eqref{cal:reference-bounds}--\eqref{cal:reference-form}. Let
$\mathfrak h\in L^2(0,T;L^2)$, put $H=\|\mathfrak h\|_{L^2L^2}$, and let
\begin{equation}\label{co:eq:controlled-difference}
 y'=L_Xy-3gP(Xy^2)-gPy^3+P\mathfrak h,\qquad y(0)=0.
\end{equation}
With $v=y/h$,
\begin{equation}\label{cal:state}
 \|v\|_{CL^2}^2+\int_0^T(\|v\|_{H^1}^2+\|v\|_4^4)\dd s
       \le C_R(1+H^{4/3}).
\end{equation}
Set
\[
 \mathcal E_I(f;y)=\|f\|_{C_IL^2}^2
       +\|f\|_{L^2_IH^{r_*}}^2+\|yf\|_{L^2_IL^2}^2.
\]
If $f'=L_Xf-6gP(Xyf)-3gP(y^2f)$, then, for every restart time $s$,
\begin{equation}\label{cal:response}
 \mathcal E_{[s,T]}(f;y)
       \le C_Re^{C_R(1+H^{4/3})}\|f(s)\|_2^2.
\end{equation}
Let $J_y$ denote this evolution. Its physical adjoint satisfies
\begin{equation}\label{cal:adjoint}
 \mathcal E_{[s,T]}(J_y(T,\cdot)^*\varphi;y)
       \le C_Re^{C_R(1+H^{4/3})}\|\varphi\|_2^2.
\end{equation}
The constants are uniform in $N,T\le\min(T_*,1)$ and $s$.
\end{proposition}
\begin{proof}
\emph{The controlled state.}
Put $B=\|v\|_4$ and $Y=\|v\|_{H^1}$. By
\eqref{cal:reference-form},
\begin{equation}\label{cal:reference-absorb}
 |\text{linear errors}|\le\varepsilon\|z\|_{H^1}^2
       +B_\varepsilon\|z\|_2^2,
 \qquad\int_0^T B_\varepsilon\le C_{R,\varepsilon}.
\end{equation}
The weights preserve $H^r$ for $r<\beta$, and
\[
 \langle P(my),y^3\rangle
 =\|y\|_4^4+\langle P((m-1)y),y^3\rangle
 \ge\tfrac12\|y\|_4^4\ge c_RB^4.
\]
Since $\alpha<\sigma<\beta$,
\[
\begin{aligned}
 |\langle P(my),Xy^2\rangle|
 &\le C_R\|v^2P(kv)\|_{B^\alpha_{1,1}}
 \le C_RB^2\|v\|_{H^\sigma}
 \le C_RB^{3-\sigma}Y^\sigma\\
 &\le\varepsilon(Y^2+B^4)+C_{R,\varepsilon},
 \qquad \frac{3-\sigma}{4}+\frac\sigma2<1.
\end{aligned}
\]
Also,
\[
 |\langle P(my),\mathfrak h\rangle|
 \le C_R\|\mathfrak h\|_2B
 \le\varepsilon B^4+C_{R,\varepsilon}\|\mathfrak h\|_2^{4/3}.
\]
Thus \eqref{cal:projected-energy} yields
\[
 \frac d{ds}\int hkv^2+c_R(Y^2+B^4)
 \le C_RB_\varepsilon\int hkv^2
      +C_R(1+\|\mathfrak h\|_2^{4/3}).
\]
Since $\int B_\varepsilon\le C_R$ and
$\int\|\mathfrak h\|_2^{4/3}\le T^{1/3}H^{4/3}$,
Gronwall gives \eqref{cal:state}.

\emph{The response.}
Put $z=f/h$, $F=\|z\|_2$ and $Z=\|z\|_{H^1}$.
The cubic term gives $\|yf\|_2^2$, with boundary error
\[
 |\langle Q(mf),y^2f\rangle|
 \le\varepsilon\|yf\|_2^2+C_{R,\varepsilon}B^2FZ
 \le\varepsilon\|yf\|_2^2+\varepsilon'Z^2
                       +C_{R,\varepsilon,\varepsilon'}B^4F^2,
\]
by \eqref{cal:boundary-four}. For the rough term,
\[
\begin{aligned}
 |\langle P(mf),Xyf\rangle|
 &\le C_R\|vzP(kz)\|_{B^\alpha_{1,1}}\\
 &\le C_R\bigl(\|v\|_{B^\sigma_{4/(1+\sigma),\infty}}
          \|z\|_{8/(3-\sigma)}^2+B\|z\|_{H^\sigma}\|z\|_4\bigr)\\
 &\le C_R\bigl(B^{1-\sigma}Y^\sigma F^{(5-3\sigma)/4}Z^{3(1+\sigma)/4}
       +BF^{5/4-\sigma}Z^{3/4+\sigma}\bigr).
\end{aligned}
\]
Young absorbs $Z^2$ and leaves $F^2$ times
\[
 B^{8(1-\sigma)/(5-3\sigma)}Y^{8\sigma/(5-3\sigma)}
       +B^{8/(5-4\sigma)}
 \le C(1+B^4+Y^2).
\]
Indeed the first term has degree $(2+2\sigma)/(5-3\sigma)<1$
relative to $B^4+Y^2$, and $8/(5-4\sigma)<4$.
Both follow from $\sigma<3/5$, which holds in \eqref{eq:indices}.
Consequently,
\[
 \frac d{ds}\int mf^2+c_R(Z^2+\|yf\|_2^2)
 \le C_R(1+B_\varepsilon+B^4+Y^2)\int mf^2.
\]
The coefficient has integral $\le C_R(1+H^{4/3})$ by
\eqref{cal:state}. Gronwall and $f=hz$, $r_*<\beta$, give
\eqref{cal:response}.

\emph{The adjoint.}
For $p_s=J_y(T,s)^*\varphi$, $-p'=L_{X+y}p$ and $p_T=\varphi$.
With $z=p/k$, the rough pairing is $\langle Xhk,vzP(hz)\rangle$.
The preceding bounds with exchanged weights give
\[
 -\frac d{ds}\int(h/k)p^2+c_R(\|z\|_{H^1}^2+\|yp\|_2^2)
 \le C_R(1+B_\varepsilon+B^4+Y^2)\int(h/k)p^2.
\]
Integrating from $s$ to $T$ gives \eqref{cal:adjoint}.
\end{proof}

\subsection{From state energy to the relative derivative}\label{co:section}

\subsubsection{Test the differentiated source against the response}
\label{co:current}

The adjoint estimate \eqref{cal:adjoint} controls $up$ in $L^2$.
This controls the pairing of the cubic source $u^3$ with $p$.

Use the OU realization of Appendix~\ref{ma:section} on $r\in[0,1]$.
Dots denote $\mathfrak D_t$, with $r$ and the frequency thresholds fixed.
Since $\mathfrak D_tZ_t(1)=tB_tZ_t(1)$, Lemma~\ref{common:ward} gives
\begin{equation}\label{co:reduced-defect}
 D_N(t)=\dot u_t(1)-tB_tu_t(1).
\end{equation}
Write $\mathcal R_{t,r}$ for the source at physical time $tr$. Differentiate
$\partial_rw_t=t[-Aw_t+\mathcal R_{t,r}(w_t)]$, then set $s=tr$:
\begin{equation}\label{co:source-equation}
 z'=(-A+D_w\mathcal R)z+F_0-Aw,\quad z=\dot w,\quad z(0)=0,
 \qquad F_0=\mathcal R+\dot{\mathcal R}|_w.
\end{equation}
Here primes denote $\partial_s$, $\mathcal R=\mathcal R_{t,s/t}$,
and $\dot{\mathcal R}|_w$ also holds $w$ fixed.

\begin{lemma}[Source tested in the dissipative norm]
\label{co:source-lemma}
For a smooth $E_N$-valued path $p$, put
\[
 \mathcal E_{u,t}(p)=\|p\|_{L^\infty(0,t;L^2)}
       +\|p\|_{L^2(0,t;H^{r_*})}
       +\|up\|_{L^2((0,t)\times\T^3)}.
\]
Uniformly in the cutoff, for $t\le T_g$,
\begin{equation}\label{co:source-spaces}
 \left|\int_0^t\langle F_0-Aw,\mathcal T_\tau^{-*}p_\tau\rangle\dd\tau\right|
       \le CM^Ct^\gamma\mathcal E_{u,t}(p).
\end{equation}
\end{lemma}
\begin{proof}
\emph{Interpolation.}
Set $v=\mathcal T^{-*}p$; $\mathcal T^{-*}$ is bounded on the test spaces.
By \eqref{co:position-energy},
\[
 \|w\|_{L^\infty H^{\epsilon_0}}+
 \|w\|_{L^2H^{1+\epsilon_0}}\le CM^Ct^\gamma,
 \qquad
 \|u\|_{L^4L^4}+\|w\|_{L^4L^4}\le CM^Ct^{\gamma/2}.
\]
Interpolating these norms, for $0\le\vartheta\le1$, gives
\begin{equation}\label{co:source-interpolation}
 \|w\|_{L^{2/\vartheta}H^{\epsilon_0+\vartheta}}
       \le CM^Ct^\gamma,
 \qquad
 \|p\|_{L^{2/\vartheta}H^{r_*\vartheta}}
       \le C\mathcal E_{u,t}(p).
\end{equation}
Also $\|h_0\|_{L^q}\le M$ by \eqref{up:integrated-model}.

\emph{The cubic uses the adjoint dissipation.}
Since $\mathcal T^*v=p$,
\[
 \left|\int_0^t\langle\mathcal TP(u^3),v\rangle\dd\tau\right|
 \le\|u\|_{L^4_{\tau,x}}^2\|up\|_{L^2_{\tau,x}}
 \le CM^Ct^\gamma\mathcal E_{u,t}(p).
\]
At fixed $w$, differentiating $u=\mathcal T^{-1}(w-q_0)$ gives
\[
 \chi:=\dot u|_w=-\mathcal T^{-1}(\dot Ku+\dot q_0),
 \qquad \|\chi\|_{L^4L^4}\le CM^Ct^{\gamma/2}.
\]
Indeed, coordinate smoothing and \eqref{ma:corrector-jet} give
$\|\chi\|_{L^4L^4}\le CM^C(t^{\gamma/2}+t^{1/4+\theta})$.
Since $\dot K^*:L^2\to H^{1-\epsilon_0}\subset L^4$,
\[
\begin{aligned}
 \left|\int_0^t\!\langle u^2\chi,p\rangle\dd\tau\right|
 &\le\|u\|_{L^4L^4}\|\chi\|_{L^4L^4}\|up\|_{L^2L^2}
 \le CM^Ct^\gamma\mathcal E_{u,t}(p),\\
 \left|\int_0^t\!\langle\dot KP(u^3),v\rangle\dd\tau\right|
 &\le Ct^{1/4}\|u\|_{L^4L^4}^3\|\dot K^*v\|_{L^\infty L^4}
 \le CM^Ct^{1/4+3\gamma/2}\mathcal E_{u,t}(p).
\end{aligned}
\]

\emph{The rough quadratic uses one spatial derivative.}
Write $r_0=w-Ku$ and $A_0=1+\|r_0\|_4$.
The product bound $L^4\cdot L^4\cdot L^2\subset L^1$ gives
\[
 |\langle Q_\zeta(r_0),p\rangle|
 \le Ch_0\bigl(A_0^2\|p\|_{B^\alpha_{2,1}}
       +A_0\|r_0\|_{B^\alpha_{2,1}}\|p\|_4\bigr).
\]
Since $\|A_0\|_{L^4}\le CM^Ct^{\gamma/2}$,
$H^\sigma\subset B^\alpha_{2,1}$ and $H^{3/4}\subset L^4$,
time H\"older gives
\[
\begin{aligned}
 \int_0^t h_0A_0^2\|p\|_{H^\sigma}\dd\tau
 &\le CM^Ct^{\gamma+1/2-1/q-\sigma/(2r_*)}\mathcal E_{u,t}(p),\\
 \int_0^t h_0A_0\|w\|_{H^\sigma}\|p\|_{H^{3/4}}\dd\tau
 &\le CM^Ct^{3\gamma/2+3/4-1/q-(\sigma-\epsilon_0)/2-3/(8r_*)}
          \mathcal E_{u,t}(p),\\
 \int_0^t h_0A_0\|Ku\|_{B^\alpha_{2,1}}\|p\|_{H^{3/4}}\dd\tau
 &\le CM^Ct^{\gamma+1/2-1/q-3/(8r_*)}\mathcal E_{u,t}(p).
\end{aligned}
\]
Each exponent exceeds $\gamma$ by \eqref{eq:indices}. Moreover,
\[
 \chi_0:=\dot r_0|_w=-\dot Ku-K\chi,\qquad
 \|Ku\|_{L^4B^\alpha_{2,1}}+
 \|\chi_0\|_{L^4(L^4\cap B^\alpha_{2,1})}
 \le CM^Ct^{\gamma/2}.
\]
by $K,\dot K:L^4\to B^{1-\epsilon_0}_{4,\infty}$ and
Lemma~\ref{ga:mark-stability}. Thus
\[
 \dot Q_\zeta(r_0)|_w
 =\dot\zeta_0r_0^2+2\zeta_0r_0\chi_0
       -2\dot\zeta_1r_0-2\zeta_1\chi_0+\dot\zeta_2
\]
obeys the same bounds; a derivative on the outer coordinate replaces
$p$ by a smoother test.

\emph{Linear and additive terms.}
By \eqref{up:linear-bound} and \eqref{co:source-interpolation},
\[
\begin{aligned}
 \int_0^t h_0\|w\|_{H^{1+\epsilon_0}}\|v\|_{H^{2\epsilon_0}}\dd\tau
 &\le CM^Ct^{\gamma+1/2-1/q-\epsilon_0/r_*}\mathcal E_{u,t}(p),\\
 \int_0^t|\langle f_0,v\rangle|\dd\tau
 &\le CM^Ct^{1-1/q-\sigma/(2r_*)}\mathcal E_{u,t}(p).
\end{aligned}
\]
Both hold after one coefficient derivative.

\emph{The initial layer is integrable against the response.}
The scalar terms are $(b+\dot b)(w+\ell)+b\dot\ell$.
Here $\ell,\dot\ell$ are bounded in $H^{-\sigma}$ and
$|b(\tau)|+|\dot b(\tau)|\lesssim\tau^{-1/2}$.
Since $\sigma<r_*$, \eqref{co:source-interpolation} gives
\[
 \int_0^t\tau^{-1/2}\|p_\tau\|_{H^\sigma}\dd\tau
 \le Ct^{1/2-\sigma/(2r_*)}\mathcal E_{u,t}(p),
 \qquad
 \frac12-\frac{\sigma}{2r_*}-\gamma
 =\frac{\epsilon_0(3/4-2\epsilon_0)}{1-\epsilon_0}>0.
\]
The additive-forcing gain is larger still. Terms containing $w$
have its additional factor $t^\gamma$.
Finally, $Aw\in L^2H^{-1+\epsilon_0}=L^2H^{-r_*}$, so
\[
 \left|\int_0^t\langle Aw,v\rangle\dd\tau\right|
 \le CM^C\|w\|_{L^2H^{1+\epsilon_0}}\|p\|_{L^2H^{r_*}}
 \le CM^Ct^\gamma\mathcal E_{u,t}(p).
\]
Summing gives \eqref{co:source-spaces}.
\end{proof}

\subsubsection{Moments of the linearized flow}

Let $J_N(t,s)$ be the linearized flow along $X_N$. Set
\begin{equation}\label{co:eq:backward-K}
 \mathcal K_{N,T}=\sup_{\|\varphi\|_2\le1}
             \mathcal E_{u,T}\bigl(J_N(T,\cdot)^*\varphi\bigr).
\end{equation}
\begin{proposition}\label{co:prop:backwardenergy}
For a deterministic $T_0>0$ and every finite $p>0$,
\begin{equation}\label{co:eq:backward-moments}
 \sup_N\sup_{0<T\le T_0}\E\mathcal K_{N,T}^p<\infty.
\end{equation}
\end{proposition}
\begin{proof}
Use the event $G_{N,T}$ of probability at least $1/2$ from
\eqref{cal:reference-bounds}--\eqref{cal:reference-form}. For
$\eta\in G_{N,T}$, $\mathfrak h\in\mathcal H=L^2(0,T;E_N)$ and
$y=X_N(\eta+\mathfrak h)-X_N(\eta)$, Proposition~\ref{cal:control}
applies with forcing $\sqrt{2a}\,\mathfrak h$.
Since $u(\eta+\mathfrak h)=y+u(\eta)-Z(\mathfrak h)$ and
$\|Z(\mathfrak h)\|_{L^2L^\infty}
 \lesssim\|Z(\mathfrak h)\|_{L^2H^2}\lesssim\|\mathfrak h\|_{\mathcal H}$,
\[
 \|u(\eta+\mathfrak h)p\|_{L^2L^2}
 \le\|yp\|_{L^2L^2}
       +C_R(1+\|\mathfrak h\|_{\mathcal H})\|p\|_{CL^2}.
\]
Hence \eqref{cal:adjoint} implies
\begin{equation}\label{co:shift-bound}
 \log(1+\mathcal K_{N,T}(\eta+\mathfrak h))
                  \le C(1+\|\mathfrak h\|_{\mathcal H}^{4/3}).
\end{equation}
Choose $T_0$ below the reference and state times. For large $L$, take
$r=cL^{3/4}$ with $C(1+r^{4/3})\le L$.
Gaussian isoperimetry~\cite{Borell} gives
\begin{equation}\label{co:eq:tan-tail}
 \mathbb P\{\log(1+\mathcal K_{N,T})>L\}
 \le1-\mathbb P_*(G_{N,T}+rB_{\mathcal H})
 \le1-\Phi(r)\le Ce^{-cL^{3/2}}.
\end{equation}
Therefore
$\E\mathcal K_{N,T}^p\le C_p+Cp\int_0^\infty
 e^{pL-cL^{3/2}}\dd L<\infty$.
\end{proof}

\subsubsection{Reconstruction of the relative derivative}

Conjugation gives
\[
 -A+D_w\mathcal R
 =\mathcal T'\mathcal T^{-1}+\mathcal T L_X\mathcal T^{-1},
 \qquad
 \mathcal U(t,s)=\mathcal T_tJ_N(t,s)\mathcal T_s^{-1},
\]
For $p_s=J_N(t,s)^*\varphi$, $\varphi\in E_N$,
\begin{equation}\label{co:eq:adjointU}
 v_s=\mathcal U(t,s)^*\mathcal T_t^{-*}\varphi
           =\mathcal T_s^{-*}p_s.
\end{equation}
Duhamel and Lemma~\ref{co:source-lemma} give
\[
 |\langle\mathcal T_t^{-1}z(t),\varphi\rangle|
 =\left|\int_0^t\langle F_0-Aw,v_s\rangle\dd s\right|
 \le CM^Ct^\gamma\mathcal K_{N,t}\|\varphi\|_2.
\]
Taking the supremum yields
\begin{equation}\label{co:eq:dualbound}
 \|\mathcal T_t^{-1}z(t)\|_2
                  \le CM^C\mathcal K_{N,t}t^\gamma.
\end{equation}
Reconstruction is exact:
$\dot u=\mathcal T^{-1}(z-\dot{\mathcal T}u-\dot q_0)$.
The coordinate bounds, \eqref{ma:corrector-jet}, and
\eqref{eq:positionneeded} imply
\[
 \|\dot u(t)\|_2
 \le CM^C\bigl(\mathcal K_{N,t}t^\gamma+t^\theta\bigr).
\]
By \eqref{co:eq:backward-moments},
$\|M^C\mathcal K_{N,t}\|_{L^p(\Omega)}
 \le\|M^C\|_{L^{2p}(\Omega)}\|\mathcal K_{N,t}\|_{L^{2p}(\Omega)}\le C_p$.
Thus $\gamma>\theta$, $L^2\subset H^{-3/4}$,
$\|tB_t\|_{L^2\to L^2}\le1/2$ and \eqref{co:reduced-defect} give
\begin{equation}\label{co:defect-rate}
 \sup_N\|D_N(t)\|_{L^p(\Omega;H_A^{-3/4})}\le C_pt^\theta
 \quad(1\le p<\infty),\qquad \sup_N e_N(t)^2\le Ct^{-3/2-2\varpi}.
\end{equation}

\subsection{Conclusion}\label{se:completion}
\begin{proof}[Completion of the proof of Theorem~\ref{th:velocity}]
Equations~\eqref{eq:positionneeded} and~\eqref{co:defect-rate} give
$\|u_N(t)\|_{L^2(\Omega;H)}+\|D_N(t)\|_{L^2(\Omega;H)}\le Ct^\theta$ uniformly in $N$.
Lemma~\ref{common:ward} gives
$e_N(t)\le Ct^{-3/4-\varpi}$. In the notation of Section~\ref{common:section}, the Weyl
exponent is $3/2$, the transport weight is $3/4$ and
$C_A=(6\pi^2\nu^{3/2})^{-1}$. The action exponent is therefore $7/4$,
and the gain is $\psi=\theta-1+7/8=1/8-\varpi$.
Theorem~\ref{common:principle} and Proposition~\ref{common:gaussian} give
\eqref{eq:pointrelative}--\eqref{eq:speedcomparison} and
\eqref{eq:universal}: the rectangular cutoff contains every eigenvalue
$\lambda_k\le L_N:=\mu+4\pi^2\nu N^2$, and $N\sqrt t\to\infty$
implies $tL_N\to\infty$.
For fixed $N$, the finite Gaussian sum gives
$\mathcal A_{0,N}(t)=O_N(t^{-1})$; \eqref{eq:speedcomparison}
therefore gives the stated zero limit at fixed cutoff.

Proposition~\ref{full:identification} passes the state bound to the full law:
$W_2(\rho_{g,t},\gamma_t)\le Ct^\theta$.
For its $C^1(L^2)$ realization $Y_t$ in
Proposition~\ref{ma:full-differentiation}, the chain rule gives, at every $t$,
\[
 \left|\partial_t\E\varphi(Y_t)-\E\langle B_tY_t,\nabla\varphi(Y_t)\rangle\right|
 =\left|\E\langle\partial_tY_t-B_tY_t,\nabla\varphi(Y_t)\rangle\right|
 \le Ct^{\theta-1}\|\nabla\varphi(Y_t)\|_2.
\]
Theorem~\ref{common:principle} proves \eqref{eq:full-comparison} at every
positive time and the ordinary limit \eqref{eq:full-limit}.
The Gaussian law and the state bound tend to $\delta_0$ at zero, so
Corollary~\ref{common:integrated} gives the $AC^1$ extension and the
length asymptotic. Choose $t_0$ below the deterministic horizons of the
state and response estimates.
\end{proof}

\section{Further remarks}\label{se:outlook}

\paragraph{Other transport norms.}
The position and derivative errors are bounded in $L^2$.
The same bounds therefore hold in every weaker norm with weights
$\lambda_j^{-s}$, $s\ge0$.

\begin{remark}[A family of weights]\label{out:family}
For KPZ, the leading Gaussian action holds in $H_s$ for every
$s\in[0,1/2)$, with $\beta=3/2-s$ and normalized comparison error
$O_{G,s}(g^2t^{1/2-s})$. Indeed, the bounds
\eqref{kpz:required-defect} and \eqref{kpz:first-chaos-defect} pass from
$L^2$ to $H_s$: the remainder and defect are $O_G(|g|t^{1/2})$, and
their first-chaos components are $O_G(g^2t^{3/4})$.
Apply \eqref{out:second-bound} in $H_s$ and use
$\mathcal A_0(t)\lesssim t^{-\beta}$ to obtain
\[
 t^\beta|\mathcal A_g(t)-\mathcal A_0(t)|
 \le C_{G,s}g^2\bigl(t^{1/2-s/2}+t^{1/2-s}\bigr)
 \le C_{G,s}g^2t^{1/2-s}.
\]
For $\Phi^4_3$, Theorem~\ref{common:principle} gives the leading Gaussian action
in $H_A^{-s}$ for every $s\in(1/2,1)$ and
$0<\varpi<\min\{1/8,(1-s)/2\}$, with action exponent $5/2-s$ and
normalized comparison error $O_\varpi(t^{1/2-s/2-\varpi})$.
In both cases, \eqref{common:gaussian-limit} gives the leading constant.
Use $\|x\|_{H_s}\le C\|x\|_2$ for $s\ge0$ to transfer the error bounds.
For $\Phi^4_3$, the gain is
$\theta-1+\beta/2=(1-s)/2-\varpi>0$.
For $0<\epsilon<s-1/2$,
\[
 \|f\|_{H^{-s}}^2
 \lesssim\sum_{j\ge-1}2^{-2sj}\|\Delta_jf\|_\infty^2
 \lesssim\|f\|_{C^{-1/2-\epsilon}}^2
           \sum_{j\ge-1}2^{-2j(s-1/2-\epsilon)}.
\]
Hence $X_N\to X$ in $H^{-s}$ by Proposition~\ref{full:identification}.
For $s>1/2$, the Gaussian field also belongs to $H_s$.
\end{remark}

\paragraph{The first nonlinear correction.}
Proposition~\ref{out:second} separates the first-chaos contribution
from the quadratic error. This gives the KPZ rate in
Theorem~\ref{kpz:action-limit}. The next questions concern its leading coefficient.

\emph{KPZ.} Theorem~\ref{kpz:action-limit} gives
$t^{3/2}|\mathcal A_g-\mathcal A_0|=O_G(g^2t^{1/2})$,
but does not determine the limiting coefficient.
Together with Remark~\ref{kpz:gaussian-expansion}, this suggests asking
whether some $c_{\rm KPZ}$ satisfies
\[
 t^{3/2}\mathcal A_g(t)
 =K+(c_{\rm KPZ}g^2-\sigma^2/4)t^{1/2}+o(t^{1/2}).
\]
\emph{$\Phi^4_3$.} The rate $1/8-\varpi$ comes from $L^2$ estimates of
quantities that are measured in $H^{-3/4}$. At first order in $g$,
the Wick-cubic forcing contributes a third-chaos term, invisible to
the linear pairing in \eqref{out:second-identity}. The first-chaos
contribution at that order comes from the initial-layer mass
in \eqref{up:corrector}. After removing the cutoff, its leading
singularity is $t^{-1/2}$ and is determined by Gaussian covariances.
A sharper rate, and a first correction computable from Gaussian data
alone, remain to be established.

\appendix

\section{Gaussian data: forcing, potential and their products}\label{ga:section}

To cancel the cubic Gaussian forcing and the quadratic potential,
we use their heat primitives. Three resonant products remain.

For an even nonnegative inhomogeneous dyadic partition $(\chi_j)$
with adjacent overlaps, put $S_j=\sum_{i\le j}\Delta_i$.
Bony's paraproducts~\cite{Bony} are
\[
 f\prec h=\sum_j(S_{j-2}f)\Delta_jh,\qquad
 f\circ h=\sum_{|i-j|\le1}\Delta_if\Delta_jh,
 \qquad fh=f\prec h+h\prec f+f\circ h.
\]
Write $C^s=B^s_{\infty,\infty}$, $H_J=\sum_{j\ge J}\Delta_j$,
$L_J=I-H_J$ and $\cI f(t)=\int_0^te^{-(t-s)A}f(s)\dd s$.
All constants below are uniform in $N$; time norms are on $[0,1]$.

\subsection{The two heat primitives and the three resonances}
For $Z=Z_N$, the time-dependent Wick polynomials are
\[
 W=Z^2-c_1,\qquad W_3=Z^3-3c_1Z,\qquad c_1(t)=\E Z(t,x)^2.
\]
The potential $W$ has regularity $-1-$. Its heat primitive and that of
the cubic forcing are
\[
 Y=\cI W\in C^{1-},\qquad \Psi=\cI P_NW_3\in C^{1/2-}.
\]
The equation fixes the projection in $\Psi$; $Y$ retains all output
frequencies. Put $c_2(t)=\E[Y(t,x)W(t,x)]$, independently of $x$.

The products $Y\circ W$, $\Psi\circ W$ and $Z\circ\Psi$ have formal
orders $0-$, $-1/2-$ and $0-$. Define
\[
 R_2=Y\circ W-c_2,\qquad
 R_3=\Psi\circ W-3c_2Z,\qquad S=Z\circ\Psi.
\]
\begin{proposition}[Enhancement estimates]\label{ga:model}
For every $\epsilon>0$ and $1\le p,q<\infty$, the following norms are
bounded in $L^p(\Omega)$ uniformly in $N$:
\begin{gather}
 \|Z\|_{C_tC^{-1/2-\epsilon}}+
 \|W\|_{C_tC^{-1-\epsilon}}+
 \|Y\|_{C_tC^{1-\epsilon}}+
 \|\Psi\|_{C_tC^{1/2-\epsilon}},\label{ga:base}\\
 \|R_2\|_{L^q_tC^{-\epsilon}}+
 \|R_3\|_{L^q_tC^{-1/2-\epsilon}}+
 \|S\|_{L^q_tC^{-\epsilon}}.\label{ga:rough-integrated}
\end{gather}
For $b,\theta\ge0$ with $b+2\theta<1/2$, the norm
$\|\Psi\|_{C_t^\theta C^b}$ also has all finite moments. In particular,
since $\Psi(0)=0$,
\begin{equation}\label{ga:psi-holder}
 \left\|\sup_{t>0}t^{-\theta}\|\Psi(t)\|_{C^b}\right\|_{L^p}
 \le C_{p,b,\theta}.
\end{equation}
\end{proposition}

\begin{proof}
Keep the subtraction inside the first-chaos kernel of $R_3$.

\emph{Positive kernels and norms.}
The covariance is
\begin{equation}\label{ga:covariance}
 \E[Z_k(t)\overline{Z_l(s)}]
 =\mathbf1_{k=l}\frac a{\lambda_k}e^{-\lambda_k|t-s|}
                  (1-e^{-2\lambda_k\min(t,s)}).
\end{equation}
The resonant multiplier $m(b,c)=\sum_{|i-j|\le1}\chi_i(b)\chi_j(c)$
satisfies $m(b,-b)=1$ and is supported where
$\langle b\rangle\asymp\langle c\rangle$.
Wick's product formula~\cite{Janson} splits $R_2,S$ into chaoses
$4,2$, and $R_3$ into chaoses $5,3,1$; the zero chaos of
$Y\circ W$ is exactly $c_2$ because $m(b,-b)=1$.
For all but $R_3^{(1)}$, square the kernels, pair the Gaussian slots,
and replace each covariance by
$C\lambda_k^{-1}e^{-c\lambda_k|t-s|}$.
The resulting integrands are nonnegative. Extend internal times to
$(-\infty,t]$ and remove the cutoff indicators. This gives the integrals in~\cite[Section~4.2, (4.24)--(4.41), and the
fifth- and third-chaos bounds on p.~34]{MWX};
$\lambda_k\asymp\langle k\rangle^2$ changes only the constants. Hence
\[
\begin{array}{c|rrrrrrr}
F&Z&W&Y&\Psi&R_2&S&R_3^{(5)}+R_3^{(3)}\\ \hline
\E|\widehat F(t,k)|^2\lesssim\langle k\rangle^{-a_F}
 &a_F=2&1&5&4&3&3&2.
\end{array}
\]
The $Y$ entry follows by heat integration of the quadratic covariance.
The quadratic bounds hold for any covariance dominated by the same envelope.

\emph{Space--time norms.} For $b=\sum_i p_i$ and
$D=\sum_i\lambda_{p_i}$,
\begin{equation}\label{gm:heat}
 \int_0^\infty\!\int_0^\infty
 e^{-\lambda_b(x+y)-D|x-y|}\,dx\,dy
 =[\lambda_b(\lambda_b+D)]^{-1}.
\end{equation}
On $[0,h]^2$ it is at most
$2\min(h,(2\lambda_b)^{-1})/(\lambda_b+D)$.
Splitting a heat integral at $s$ therefore multiplies its variance
bound by $C_r|t-s|^{2r}\lambda_b^{2r}$, $r\le1/2$.
For $Z$, split the stochastic integral. For $W$, telescope the two
Gaussian factors and use~\cite[Lemma~4.1]{MWX}. The increment variance is bounded by
\[
 C_r|t-s|^{2r}\langle k\rangle^{-3-2\varrho_F+4r},\qquad
 \begin{array}{c|rrrr}F&Z&W&Y&\Psi\\ \hline
 \varrho_F&-1/2&-1&1&1/2\\ r&<1/2&<1/4&<1/2&<1/2.
 \end{array}
\]
Finite-chaos moments and Bernstein~\cite[Propositions~3.3 and~3.6]{MWX}
give, for $p_0\ge\max(2,p,q)$,
\[
 \left\|\max_{\ell<2^n}\|\Delta_j[F((\ell+1)2^{-n})-F(\ell2^{-n})]\|_\infty
 \right\|_{L^{p_0}}
 \lesssim2^{-n(r-1/p_0)}2^{j(-\varrho_F+2r+3/p_0)}.
\]
Since $F(0)=0$, dyadic summation yields
\[
 \|F\|_{L^p(\Omega;C_t^\theta C^b)}
 \lesssim\sum_{j,n\ge0}
 2^{-j(\varrho_F-b-2r-3/p_0)}2^{-n(r-\theta-1/p_0)}<\infty
\]
when $b+2r+3/p_0<\varrho_F$ and $\theta+1/p_0<r$.
Take $(b,\theta)=(\varrho_F-\epsilon,0)$ for \eqref{ga:base};
for $\Psi$, choose $\theta<r<(1/2-b)/2$ and then $p_0$ large.
This also proves \eqref{ga:psi-holder}. At fixed time, if
$\E|\widehat F(t,k)|^2\lesssim\langle k\rangle^{-3-2\varrho_F}$,
Minkowski gives
\[
 \|F\|_{L^p(\Omega;L_t^qC^{\varrho_F-\epsilon})}
 \lesssim\sum_{j\ge0}2^{-j(\epsilon-3/p_0)}<\infty
 \qquad(p_0>3/\epsilon).
\]

\emph{The compensated first chaos.}
For the projected primitive, put
$c_2^P=\E[(\cI P_NW)W]$ and $R_3^P=\Psi\circ W-3c_2^PZ$.
For $k\in\Lambda_N$, write $K=\langle k\rangle$, $b=h+l$,
$\theta_h(v)=1-e^{-2\lambda_hv}$,
$M_k(b)=\mathbf1_{k+b\in\Lambda_N}m(k+b,-b)$ and
$M_0(b)=\mathbf1_{b\in\Lambda_N}$. The six double contractions give exactly
\begin{align}\label{ga:R3-first}
 (R_3^P)^{(1)}(t,k)=6a^2\sum_{h,l\in\Lambda_N}
 \frac1{\lambda_h\lambda_l}\int_0^t
 &e^{-(\lambda_h+\lambda_l)v}\theta_h(t-v)\theta_l(t-v)\\[-1mm]
 &\hspace{-35mm}\times
 [M_k(b)e^{-\lambda_{k+b}v}Z_k(t-v)
       -M_0(b)e^{-\lambda_bv}Z_k(t)]\,dv.\notag
\end{align}
For $k\notin\Lambda_N$ it vanishes. On
$R\le\max(\langle h\rangle,\langle l\rangle)<2R$, telescope the bracket:
replace $Z_k(t-v)$ by $Z_k(t)$, then $\lambda_{k+b}$ by $\lambda_b$,
$m(k+b,-b)$ by one, and the shifted indicator by the unshifted one.
The shell covariance weight is at most $CR^2$, its common heat rate
at least $cR^2$, and at fixed $b$ its weight is at most $CR^{-1}$.
For $R\gg K$, their respective $L^2(\Omega)$ bounds are
\[
\begin{aligned}
 K^{-1}R^2\int_0^\infty e^{-cR^2v}(K^2v)^{\eta_{\rm sh}}\dd v
 &\lesssim K^{-1}(K/R)^{2\eta_{\rm sh}},\\
 K^{-1}R^2(KR+K^2)\int_0^\infty ve^{-cR^2v}\dd v
 &\lesssim K^{-1}K/R,\\
 K^{-1}R^{-3}\sum_{\langle b\rangle\lesssim R}\min(1,K/\langle b\rangle)
 &\lesssim K^{-1}K/R,\\
 K^{-1}R^{-3}(KN^2)\mathbf1_{R\asymp N}
 &\lesssim K^{-1}K/R,
\end{aligned}
\qquad 0<\eta_{\rm sh}<1/2.
\]
Here $\|Z_k(t-v)-Z_k(t)\|_2\lesssim K^{-1}\min(1,(K^2v)^{\eta_{\rm sh}})$,
$|\lambda_{k+b}-\lambda_b|\lesssim KR+K^2$ and
$|m(k+b,-b)-1|\lesssim\min(1,K/\langle b\rangle)$.
For $R\lesssim K$, each term is bounded by $CK^{-1}$.
Consequently the shell contribution $F_{k,R}$ satisfies
\begin{equation}\label{ga:shell}
 \|F_{k,R}\|_2\le C_{\eta_{\rm sh}} K^{-1}\min(1,(K/R)^{2\eta_{\rm sh}}),
 \qquad \|(R_3^P)^{(1)}(t,k)\|_2\le CK^{-1}\log(e+K).
\end{equation}
Indeed, $\sum_{R\le K}1+\sum_{R>K}(K/R)^{2\eta_{\rm sh}}\lesssim\log(e+K)$.
By \eqref{ga:full-correction}, $R_3=R_3^P-3\Delta_NZ$ with
$|\Delta_N|\le C$. Absorb the logarithm in the spatial loss and apply
the preceding block estimates to obtain \eqref{ga:rough-integrated}.

\end{proof}

\subsection{A correction derived from the same data}
Differentiating the coordinate produces $(P_NW_3)\prec H_JY$.
Its heat integral is
\[
 \Omega_J=\cI P_N[(P_NW_3)\prec H_JY].
\]
\begin{lemma}[Integrated paraproduct correction]
For every $\epsilon>0$ and finite $p$,
\begin{equation}\label{ga:omega-norms}
 \sup_J\|\Omega_J\|_{L^\infty_tC^{3/2-\epsilon}}
 +\sup_{J,t>0}t^{-3/4+\epsilon}\|\Omega_J(t)\|_2
\end{equation}
has bounded $L^p(\Omega)$ norm uniformly in $N$.
\end{lemma}
\begin{proof}
Put $\mathscr L=\partial_t+A$, $Y_h=\cI h$ and
$\mathcal C(f,h)=\cI(f\prec h)-f\prec\cI h$.
The product rule and $Y_h(0)=0$ give
\begin{equation}\label{ga:Omega-identity}
 \cI P_N[(\mathscr Lf)\prec Y_h]
 =P_N\left[-\mathcal C(f,h)+\mu\cI(f\prec Y_h)
       +2\nu\sum_i\cI(\partial_i f\prec\partial_iY_h)\right].
\end{equation}
If $f\in C_tC^{\alpha_G}\cap C_t^{\alpha_G/2}L^\infty$,
$h\in C_tC^{\beta_G}$, $0<\alpha_G<1$, $-2<\beta_G<-\alpha_G$, put
$\sigma_G=\alpha_G+\beta_G+2>0$ and let $F,H$ denote these norms.
At block $j$, the time difference of $S_{j-2}f$ and the spatial
commutator are bounded respectively by
\[
 CFH\tau^{\alpha_G/2}e^{-c\tau2^{2j}}2^{-\beta_G j},\qquad
 CFHe^{-c\tau2^{2j}}2^{-(\alpha_G+\beta_G)j}.
\]
The second follows from the localized heat-kernel moment
$\int|K_{j,\tau}(x)|d(x,0)^{\alpha_G} dx
\lesssim2^{-\alpha_G j}e^{-c\tau2^{2j}}$.
Thus
\[
 \|\Delta_j\mathcal C(f,h)(t)\|_\infty
 \lesssim FH2^{-(\alpha_G+\beta_G)j}\min(t,2^{-2j}),\qquad
 \sum_{j\ge0}2^{-(\alpha_G+\beta_G)j}\min(t,2^{-2j})
 \lesssim t^{\sigma_G/2}.
\]
The derivative product has block size $CFH2^{-(\alpha_G+\beta_G)j}$
and hence the same integrated bounds; the mass term is smoother.
Finally
\begin{equation}\label{ga:projection}
 \|P_Nv\|_{C^{b-\eta_{\rm pr}}}\le C_{b,\eta_{\rm pr}}\|v\|_{C^b},\qquad
 \|P_Nv\|_2\le\|v\|_2,\qquad\eta_{\rm pr}>0,
\end{equation}
since only blocks $2^j\asymp N$ are affected, and their Dirichlet
kernel norm is at most $C(1+j)^3\lesssim2^{\eta_{\rm pr}j}$.
Take $f=\Psi$, $h=H_JW$, let $\alpha_G\uparrow1/2$ and
$\beta_G\uparrow-1$, and use a strict projection loss.
Equation~\eqref{ga:Omega-identity} gives \eqref{ga:omega-norms}, uniformly in $J$.
\end{proof}

\subsection{Projection boundaries and the prescribed counterterms}

\begin{lemma}[Covariance envelope and spectral boundary]\label{ga:quadratic-envelope}
Let $\widetilde Z$ be a jointly measurable, centered, spatially homogeneous
Gaussian field supported in $\Lambda_N$, with covariance bounded by
$C\mathbf1_{k=l}\lambda_k^{-1}e^{-c\lambda_k|t-s|}$.
Set $\widetilde W=\widetilde Z^2-\E\widetilde Z^2$,
$\widetilde Y=\cI\Pi\widetilde W$, for a fixed real-preserving
multiplier with $|\pi|\le1$, and
$\widetilde R=\widetilde Y\circ\widetilde W-\E[\widetilde Y\circ\widetilde W]$.
The variances of $(\widetilde W,\widetilde Y,\widetilde R)$ are bounded
by $C(\langle k\rangle^{-1},\langle k\rangle^{-5},\langle k\rangle^{-3})$.
For $\epsilon>0$, a coefficient $B(t)$ with all finite moments uniformly
in $N,t$ bounds
\begin{equation}\label{ga:envelope-boundary}
 \mathcal Bf=P_N\{[(I-P_N)(f\prec\widetilde Y)]\circ\widetilde W\},
 \qquad
 \|\mathcal Bf\|_{H^{-\epsilon}}\le B(t)\|f\|_2,\quad
 \|\mathcal Bf\|_{C^{-\epsilon}}\le B(t)\|f\|_\infty.
\end{equation}
For finite $p,q$, $\widetilde R\in L^p(\Omega;L^q_tC^{-\epsilon})$;
the operator norms in \eqref{ga:envelope-boundary} belong to $L^p(\Omega;L^q_t)$.
\end{lemma}
\begin{proof}
The variance bounds follow from the positive-kernel argument above.
A boundary term has $2^j\asymp2^l\asymp N$, since the product has
frequency $2^j$ and $\widetilde Y,\widetilde W$ are supported in $\Lambda_{2N}$.
Choose $q_0>4$ with $6/q_0<\epsilon$ and take
\[
 B(t)=C\sum_{2^j\asymp N}\|\Delta_j\widetilde Y(t)\|_{q_0}
          \sum_{2^l\asymp N}\|\Delta_l\widetilde W(t)\|_{q_0}.
\]
Cube projections are uniformly bounded on each finite $L^r$,
$1<r<\infty$, by the tensor-product Riesz projection theorem~\cite{Riesz}.
For $1/s=1/2+2/q_0$, H\"older and the embeddings
$L^s\hookrightarrow H^{-\epsilon}$, $L^{q_0/2}\hookrightarrow C^{-\epsilon}$ give
\[
 \|\mathcal Bf\|_{H^{-\epsilon}}\lesssim\|\mathcal Bf\|_s\le B\|f\|_2,
 \qquad
 \|\mathcal Bf\|_{C^{-\epsilon}}\lesssim\|\mathcal Bf\|_{q_0/2}\le B\|f\|_\infty.
\]
Moreover,
\[
 \|B(t)\|_{L^p(\Omega)}
 \lesssim\sum_{2^j\asymp2^l\asymp N}
 \|\Delta_j\widetilde Y\|_{L^{2p}(\Omega;L^{q_0})}
 \|\Delta_l\widetilde W\|_{L^{2p}(\Omega;L^{q_0})}
 \lesssim\sum_{2^j\asymp2^l\asymp N}2^{-j}2^l\lesssim1.
\]
Minkowski gives the time norms.
\end{proof}

\paragraph{Full and projected primitives.}
With $d_{pq}=\lambda_p+\lambda_q+\lambda_{p+q}$ and
$I_{pq}(t)=\int_0^te^{-d_{pq}(t-s)}\theta_p(s)\theta_q(s)ds$,
Wick's rule gives
$c_2^P=2a^2\sum_{p,q,p+q\in\Lambda_N}I_{pq}/(\lambda_p\lambda_q)$.
The full primitive drops only the last restriction. Hence
\begin{equation}\label{ga:full-correction}
 \Delta_N=c_2-c_2^P
 =2a^2\sum_{\substack{p,q\in\Lambda_N\\p+q\notin\Lambda_N}}
 \frac{I_{pq}}{\lambda_p\lambda_q},
 \qquad 0\le\Delta_N\le C,\quad |t\Delta_N'|\le C.
\end{equation}
For $K=\langle k\rangle$, $I_{pq}\le d_{pq}^{-1}$ and
\[
 \sum_{p+q=k}\frac1{\lambda_p\lambda_qd_{pq}}
 \lesssim K^{-4}\sum_{|p|\le2K}\langle p\rangle^{-2}
       +\sum_{|p|>2K}\langle p\rangle^{-6}\lesssim K^{-3},
 \qquad \sum_{N<|k|_\infty\le2N}\langle k\rangle^{-3}\lesssim1.
\]
For $F=\theta_p\theta_q$, $I_{pq}'=e^{-d_{pq}\cdot}*F'$,
$F'\ge0$, $\int F'\le1$ and $sF'(s)\le C$.
Hence
\[
 tI_{pq}'=\int_0^t[(t-s)+s]e^{-d_{pq}(t-s)}F'(s)\dd s
 \le\frac C{d_{pq}}\int_0^tF'(s)\dd s
       +C\int_0^te^{-d_{pq}(t-s)}\dd s\le\frac C{d_{pq}}.
\]

The prescribed constants therefore satisfy, for $0<t\le1$,
\begin{equation}\label{ga:initial-layer}
\begin{aligned}
 &0\le C_{1,N}-c_1(t)\le Ct^{-1/2},\qquad |tc_1'(t)|\le Ct^{-1/2},\\
 &|C_{2,N}-c_2(t)|+|tc_2'(t)|\le C\log(e/t).
\end{aligned}
\end{equation}
The first line follows by summing $\lambda_k^{-1}e^{-2t\lambda_k}$
and $t e^{-2t\lambda_k}$. For the second, split at $t/2$:
\[
 \frac1{d_{pq}}-I_{pq}(t)
 =\frac{e^{-td_{pq}}}{d_{pq}}
  +\int_0^te^{-d_{pq}(t-s)}[1-\theta_p(s)\theta_q(s)]\dd s
 \lesssim\frac{e^{-ctd_{pq}}+e^{-ct\lambda_p}+e^{-ct\lambda_q}}{d_{pq}}.
\]
The same bound holds for $tI_{pq}'$. Now
\[
 \sum_q(\lambda_qd_{pq})^{-1}\lesssim\langle p\rangle^{-1},
 \qquad \sum_p\langle p\rangle^{-3}e^{-ct\lambda_p}\lesssim\log(e/t),
\]
and \eqref{ga:full-correction} give \eqref{ga:initial-layer}.

\subsection{Differentiating the renormalized products}

\begin{lemma}[Gaussian differentiation]\label{ga:mark-stability}
Use the OU realization of Lemma~\ref{gaalign:ou}.
The bounds above for the Gaussian fields and boundary operator also hold
after one $\mathfrak D_t$ derivative, in the same physical-time norms.
The constants are uniform in cutoff and horizon.
Keep inputs, multipliers and thresholds fixed.
\end{lemma}
\begin{proof}
By \eqref{gaalign:marked-factor}, differentiating the $n$ Gaussian
factors inserts $n$ operators of norm at most one. These commute with
spatial blocks and time differences.
On a contracted pair,
\[
 \langle\mathsf A_th_1,h_2\rangle+\langle h_1,\mathsf A_th_2\rangle=0.
\]
For \eqref{ga:covariance}, put $z=t\lambda$, $m=\min(r,v)$ and
$d=|r-v|$. Then
\[
 \dot C_t(r,v)=\frac a\lambda e^{-zd}
 [-zd(1-e^{-2zm})+2zm e^{-2zm}],\qquad
 |\dot C_t(r,v)|\lesssim\frac a\lambda e^{-zd/2}(1-e^{-2zm}),
\]
since $2xe^{-2x}\le1-e^{-2x}$ and $xe^{-x}\lesssim e^{-x/2}$.
Differentiated heat kernels satisfy the same bounds.

Differentiate before estimating:
\[
 \mathfrak D_tR_3=(\mathfrak D_t\Psi)\circ W+\Psi\circ\mathfrak D_tW
       -3\dot c_2Z-3c_2\mathfrak D_tZ.
\]
In \eqref{ga:R3-first}, Gaussian derivatives preserve the increment bound.
The multipliers are fixed, and
\[
 |\mathfrak D_t(e^{-tA_1x}-e^{-tA_0x})|
 \lesssim t|A_1-A_0|x[1+t(A_1+A_0)x]e^{-ct\min(A_1,A_0)x}.
\]
Using $x^me^{-cx}\lesssim_m e^{-cx/2}$ in the four shell estimates
preserves \eqref{ga:shell} and \eqref{ga:full-correction}.
Apply the same bound to \eqref{ga:Omega-identity} and the bilinear
boundary estimate, then sum blocks and time increments as above.
\end{proof}

\section{Correct the source, then estimate the state}
\label{po:section}

The additive corrector removes the cubic Gaussian forcing from
$u=X_N-Z_N$. We estimate $r=u+q$ by entropy, then use
$w=\mathcal Tr$ for heat regularity and source differentiation.

Write $P=P_N$, $\eta=3g$ and $\mathscr H=\partial_t+A$.
Keep frequency thresholds fixed under every coefficient derivative $\mathfrak d$.

\subsection{The coordinate and its single cancellation}
Let $M_0\ge1$ bound \eqref{ga:base}, \eqref{ga:omega-norms} with loss
$\epsilon_0/4$, including their time weights and first coefficient derivatives.
Include also the backward primitive \eqref{cal:primitives}.
Then $M_0$ has all finite moments, uniformly in $N,T$.
Put $Y_h=H_JY$ and define
\begin{equation}\label{po:transform}
 Kf=\eta P(f\prec Y_h),\qquad \mathcal T=I+K,\qquad
 q_0=g\Psi+\eta g\Omega_J,\qquad w=\mathcal Tu+q_0.
\end{equation}
Then
\[
 \mathscr Hq_0=g\mathcal TPW_3.
\]
Choose $J$ as in Appendix~\ref{ac:section}, large enough also for the
weights in Appendix~\ref{ac:reference-section}. Then $2^J\lesssim M_0^C$
and $\mathcal T^{\pm1}$ are bounded on the spaces $L^p$, $C^a$, $H^a$
used below, with finite $p$ and $a<\beta$. The maps
\begin{equation}\label{po:smoothing}
 K,K^*:L^p\to H^{a,p}\ (a<\beta),\qquad
 K:L^4\to L^\infty,\qquad K:H^{-3/4}\to L^2
\end{equation}
have polynomial bounds in $M_0$. Extend $K$ by $PKP$ when taking adjoints.
Moreover,
\begin{equation}\label{po:q0}
 \|q_0(t)\|_{C^{1/2-\epsilon_0/4}}+
 t^{-\theta}\|q_0(t)\|_{C^\varpi}\lesssim M_0^C,
\end{equation}
also after one coefficient derivative. For $\Psi$ this is
$\varpi+2\theta<1/2$ in \eqref{ga:psi-holder}. For $\Omega_J$, put
$\lambda=(3/2-\epsilon_0/4-\varpi)/(3-\epsilon_0/4)$. Then
\[
 \|\Omega_J(t)\|_{C^\varpi}
 \lesssim\|\Omega_J(t)\|_{L^2}^{\lambda}
 \|\Omega_J(t)\|_{C^{3/2-\epsilon_0/4}}^{1-\lambda}
 \lesssim M_0^Ct^{\lambda/2},\qquad \lambda/2>\theta.
\]

The linear remainder is
\begin{equation}\label{up:normalform}
 \mathcal Af=[\mathscr H,K]f-\eta\mathcal TP(Wf)
 +\eta\mathcal TP((\mathcal Tf)\circ W)-\eta^2c_2\mathcal Tf.
\end{equation}
The heat commutator cancels $f\prec W$. The added resonance cancels
$f\circ W$; the last term subtracts the scalar contraction of
$(Kf)\circ W$.
Applying the same identity to the additive corrector leaves the forcing
\[
 F=\eta\mathcal TP(q_0\circ W-\eta c_2Z)
   =\eta g\mathcal TP(R_3+\eta\Omega_J\circ W).
\]
The only projection remainder in this cancellation is
\begin{equation}\label{up:boundary}
 \mathcal Bf=P\{[(P-I)(f\prec Y)]\circ W\}.
\end{equation}
Let $\mathfrak m(t)$ be one plus the norms of $(R_2,R_3,S)$ in
$(C^{-\epsilon_0/4},C^{-1/2-\epsilon_0/4},C^{-\epsilon_0/4})$ and the operator norms of
$\mathcal B:L^2\to H^{-\epsilon_0/4}$ and
$\mathcal B:L^\infty\to C^{-\epsilon_0/4}$, including their derivatives. Put
\begin{equation}\label{up:integrated-model}
 h_0(t)=CM_0^C\mathfrak m(t),\qquad
 M=M_0+\|h_0\|_{L^q(0,1)},\qquad
 \sup_N\E M^p<\infty\quad(p<\infty).
\end{equation}
These bounds hold on every shorter time interval. Increase the fixed
power of $M_0$ in $h_0$ to include further polynomial factors.
The letters $h,k$ denote the positive entropy weights.

\begin{lemma}[What remains after cancellation]\label{up:operators}
For $\epsilon_0<a<\beta$ and $E^a=C^a$ or $H^a$,
\begin{equation}\label{up:order}
 \|\mathcal Af\|_{E^{a-1-\epsilon_0/2}}\lesssim h_0\|f\|_{E^a},
 \qquad \|F\|_{C^{-\alpha}}\lesssim h_0.
\end{equation}
Set
\[
 \mathcal L=-\eta\mathcal TP(\,\cdot\,\circ W)
                  +\mathcal A\mathcal T^{-1},\qquad
 f_0=F-\mathcal A\mathcal T^{-1}q_0.
\]
Then
\begin{equation}\label{up:linear-bound}
 \|\mathcal Lw\|_{H^{-2\epsilon_0}}\lesssim h_0\|w\|_{H^{1+\epsilon_0}},
 \qquad \|f_0\|_{C^{-\alpha}}+\|f_0\|_{H^{-\alpha}}
                                                    \lesssim h_0.
\end{equation}
The bounds hold after one coefficient derivative, with the input fixed.
\end{lemma}
\begin{proof}
Write
$\mathcal A=\mathsf D-\eta\mathcal TP(W\prec\,\cdot\,)
+\eta^2\mathcal T\mathsf R$. Since
\[
 [\mathscr H,K]f=\eta P\Bigl(f\prec H_JW-\mu f\prec Y_h
                  -2\nu\sum_i\partial_if\prec\partial_iY_h\Bigr),
\]
\begin{equation}\label{up:heatcomm}
 \mathsf Df=-\eta P\Bigl(f\prec W_l+\mu f\prec Y_h
       +2\nu\sum_i\partial_if\prec\partial_iY_h\Bigr)
       -\eta KP(f\prec W),
\end{equation}
where $W_l=(I-H_J)W$. The resonant commutator gives
\begin{equation}\label{up:rescomm}
 \mathsf Rf=P\{fR_2+\mathfrak C(f,Y,W)
                      -[P(f\prec Y_l)]\circ W\}+\mathcal Bf,
\end{equation}
where $Y_l=Y-Y_h$ and
$\mathfrak C(f,Y,W)=(f\prec Y)\circ W-f(Y\circ W)$.
Paraproduct and commutator bounds, with projection loss $<\epsilon_0/8$, give
\[
\begin{aligned}
 \|P(\partial_if\prec\partial_iY_h)\|_{E^{a-1-3\epsilon_0/8}}
 +\|P\mathfrak C(f,Y,W)\|_{E^{a-5\epsilon_0/8}}&\lesssim h_0\|f\|_{E^a},\\
 \|KP(f\prec W)\|_{E^{-3\epsilon_0/4}}
 +\|P(fR_2)\|_{E^{-3\epsilon_0/4}}
 +\|\mathcal Bf\|_{E^{-3\epsilon_0/4}}&\lesssim h_0\|f\|_{E^a}.
\end{aligned}
\]
Each output index exceeds $a-1-\epsilon_0/2$ since $a<\beta$.
The low-frequency and mass terms are smoother. Moreover,
\[
 \|F\|_{C^{-\alpha}}\lesssim M_0^C
 \bigl(\|R_3\|_{C^{-1/2-\epsilon_0/4}}
 +\|\Omega_J\|_{C^{3/2-\epsilon_0/4}}
       \|W\|_{C^{-1-\epsilon_0/4}}\bigr)\lesssim h_0.
\]
For $\mathcal Lw$, use $a=1-3\epsilon_0/4$ in \eqref{up:order};
$\mathcal A\mathcal T^{-1}w\in H^{-5\epsilon_0/4}$, while
$w\circ W\in H^{3\epsilon_0/4}$.
For $f_0$, use $a=1/2-\epsilon_0/2$ in both the H\"older and Sobolev versions.

For every multilinear expression $B$,
$\mathfrak dB(a_1,\ldots,a_n)=\sum_iB(a_1,\ldots,\mathfrak da_i,\ldots,a_n)$, and
\begin{equation}\label{up:inverse-jet}
 \mathfrak d(\mathcal T^{-1})
       =-\mathcal T^{-1}(\mathfrak dK)\mathcal T^{-1}.
\end{equation}
The inverse is bounded. Apply Lemma~\ref{ga:mark-stability} to the
renormalized expressions; no smallness of $\mathfrak dK$ is needed.
\end{proof}

\subsection{Equations for energy and regularity}
Set
\begin{equation}\label{state:corrector}
 q=\mathcal T^{-1}q_0,\qquad r=u+q=\mathcal T^{-1}w.
\end{equation}
The identity $q-g\Psi=\eta g\Omega_J-Kq$ shows that
$q-g\Psi\in C^{1-\epsilon_0/2}$ with polynomial bounds. Consequently the
products $Zq$ and $Zq^2$ use only $S=Z\circ\Psi$ and ordinary products.
Moreover $\|q(t)\|_{C^\varpi}\lesssim M_0^Ct^\theta$.

The two scalar coefficients are
\[
 b_0=\eta(C_1-c_1)-\eta^2C_2,\qquad
 b=b_0+\eta^2c_2.
\]
Keep $C_2$ inside the potential. For $b$, \eqref{ga:initial-layer} gives
\begin{equation}\label{up:corrector}
 |b(t)|+|\mathfrak db(t)|\lesssim t^{-1/2}.
\end{equation}
Let $L_0^P=\nu\Delta-\mu-\eta P(W\,\cdot\,)+b_0$.
Then
\begin{equation}\label{state:additive-equation}
 r'=L_0^Pr-gP\bigl((r-q)^3+3Z(r-q)^2\bigr)
                 +e+b(Z-q),\qquad e=\mathcal T^{-1}f_0.
\end{equation}
Indeed, $\mathcal Tq=q_0$ implies
$\mathscr Hq=gPW_3-\mathcal T^{-1}[\mathscr H,K]q$.
Also, \eqref{up:normalform} gives
\[
 e=\eta PWq-\mathcal T^{-1}[\mathscr H,K]q
                         -\eta^2c_2(Z-q).
\]
Adding this to
$u'=L_0^Pu-gP(u^3+3Zu^2)-gPW_3+b_0Z$
gives \eqref{state:additive-equation}, with
$\|e\|_{C^{-\alpha}}\lesssim h_0$.

For the equation of $w$, write $r_0=u+q_0=w-Ku$ and put
\begin{equation}\label{up:zeta}
 \zeta=(Z,Zq_0,Zq_0^2),\qquad
 Q_\zeta(r_0)=\zeta_0r_0^2-2\zeta_1r_0+\zeta_2=Zu^2,
\end{equation}
and set $\ell=\mathcal TZ-q_0$. Then
\begin{align}
 \mathscr Hw&=\mathcal P_t(w)+b(t)(w+\ell),\qquad w(0)=0,\label{po:A8}\\
 \mathcal P_t(w)&=-g\mathcal TP(u^3+3Q_\zeta(r_0))
                                  +\mathcal Lw+f_0.\label{up:polynomial-def}
\end{align}
The identities
\[
 Z\Psi=Z\prec\Psi+\Psi\prec Z+S,\qquad
 Z\circ\Psi^2=2\Psi S+2\mathfrak C(\Psi,\Psi,Z)+Z\circ(\Psi\circ\Psi)
\]
give $\|\zeta\|_{C^{-1/2-\epsilon_0/4}}+
\|\mathfrak d\zeta\|_{C^{-1/2-\epsilon_0/4}}\lesssim h_0$;
the remaining resonances have positive regularity sum.

\subsection{Two ordinary entropies}

The quadratic entropy controls $r$; the quartic entropy gives $u^3\in L^2_{t,x}$.
Use the weights of Proposition~\ref{cal:construction}, put $v=r/h$,
and include their bounds in $M$. Fix one threshold satisfying both
the coordinate and weight estimates.

\begin{lemma}[Quadratic entropy]\label{state:quadratic}
There is a deterministic $T_g>0$ such that, for $T\le T_g$,
\begin{equation}\label{state:quadratic-bound}
 \sup_{t\le T}\|v(t)\|_2^2+
 \int_0^T\bigl(\|v\|_{H^1}^2+\|v\|_4^4\bigr)\dd t
 \le CM^CT^{1-\sigma}.
\end{equation}
Consequently $\|r(t)\|_2+\|w(t)\|_2
\le CM^Ct^{1/4-\epsilon_0}$ and $\|u(t)\|_2\le CM^Ct^\theta$.
\end{lemma}
\begin{proof}
\emph{The energy and its dissipation.}
Let $E_T$ denote the left side of \eqref{state:quadratic-bound}.
Apply \eqref{cal:general-entropy} with $\Phi(v)=v^2/2$ to
\eqref{state:additive-equation}. Since $r\in E_N$ and $k/h$ is close
to one, the leading cube gives
\[
 \langle P((k/h)r),r^3\rangle
 \ge\bigl(1-\|P\|_{L^4\to L^4}\|k/h-1\|_\infty\bigr)\|r\|_4^4
 \ge\tfrac12\|r\|_4^4.
\]
The constant $\mu>0$ in the weight defects supplies the $L^2$ term.
The available dissipation is therefore $\|v\|_{H^1}^2+\|v\|_4^4$.

\emph{One degree comparison absorbs the rough terms.}
Expand the lower powers:
\[
 (r-q)^3+3Z(r-q)^2
 =r^3+3(Z-q)r^2+3(q^2-2Zq)r+3Zq^2-q^3.
\]
By the product bounds after \eqref{state:corrector}, the coefficients
belong to $L^q_tC^{-\alpha}$. Bound their norms and the rough weight errors
by $a(t)\ge1$, with $\|a\|_{L^q(0,T)}\le CM^C$.
Choose $\alpha<\sigma<\beta$, with $\sigma$ just above $1/2$.
Fractional differences with spatial exponents $4,4,2$ give
\[
 \|h^2v^2P(kv)\|_{B^\alpha_{1,1}}
 \lesssim M^C\|v\|_4^2\|v\|_{H^\sigma}
 \lesssim M^C\|v\|_4^{3-\sigma}\|v\|_{H^1}^{\sigma}.
\]
Thus the quadratic pairing is bounded by
$Ca(t)\|v\|_4^{3-\sigma}\|v\|_{H^1}^\sigma$.
Its degree relative to the two dissipations is
\[
 \frac{3-\sigma}{4}+\frac\sigma2
 =\frac{3+\sigma}{4}<1.
\]
Consequently Young's inequality gives, for every $\varepsilon>0$,
\[
 Ca(t)\|v\|_4^{3-\sigma}\|v\|_{H^1}^{\sigma}
 \le\varepsilon(\|v\|_4^4+\|v\|_{H^1}^2)
       +C_\varepsilon a(t)^{4/(1-\sigma)}.
\]
The linear and additive pairings cost respectively
$a\|v\|_4^{2-\sigma}\|v\|_{H^1}^\sigma$ and
$a\|v\|_4^{1-\sigma}\|v\|_{H^1}^\sigma$.
Their degrees are $(2+\sigma)/4$ and $(1+\sigma)/4$.
The weight errors pair with $v^2$ and have the first of these degrees.
They are therefore absorbed by the same bound. Set $m=4/(1-\sigma)$
and take $q\ge m/\sigma$. Then
\[
 \int_0^T a(t)^m\dd t
 \le T^{1-m/q}\|a\|_{L^q}^m\le CM^CT^{1-\sigma}.
\]

\emph{The projection error.}
Equation~\eqref{cal:quadratic-quotient-projection} bounds its error by
$\Gamma_{N,T}(t)(\|v\|_4\|v\|_{H^1}+\|v\|_4^2)$.
Here $T$ is the horizon of the backward weight. The degrees are $3/4$
and $1/2$, so Young leaves $C_\varepsilon(1+\Gamma_{N,T}^4)$.
The pointwise moments proved in Appendix~\ref{cal:projection-section}
give, for every finite $p$,
\[
 \bigl\|\|\Gamma_{N,T}\|_{L^8(0,T)}\bigr\|_{L^p(\Omega)}
 \le C_pT^{1/8},\qquad 0<T\le1,
\]
uniformly in $N$. For each $T$, include this norm in $M$.
Its moment bounds are uniform in $N,T$; no supremum over $T$ is used. Thus
\[
 \int_0^T\Gamma_{N,T}^4\dd t
 \le T^{1/2}\|\Gamma_{N,T}\|_{L^8(0,T)}^4
 \le CM^CT^{1-\sigma},\qquad \sigma>1/2.
\]
\emph{The initial layer determines the rate.}
The spatially constant $t^{-1/2}$ weight error costs $C\sqrt T E_T$.
For $bZ$, interpolation and time H\"older give
\[
 \int_0^T t^{-1/2}\|v(t)\|_{H^\sigma}\dd t
 \le C T^{(1-\sigma)/2}
 \left(\sup_{t\le T}\|v(t)\|_2^2+
                  \int_0^T\|v\|_{H^1}^2\dd t\right)^{1/2}.
\]
Indeed $\|v\|_{H^\sigma}\le
\|v\|_2^{1-\sigma}\|v\|_{H^1}^\sigma$; the remaining integral is
$\int_0^Tt^{-1/(2-\sigma)}\dd t$, finite because $\sigma<1$.
Also,
$\int_0^T|b|\|q\|_2\|v\|_2\dd t
\lesssim M^CT^{1/2+\theta}E_T^{1/2}$.
The logarithmic errors in \eqref{cal:defect-bounds} belong to every finite $L^q_t$. We obtain
\[
 E_T\le C\sqrt T E_T+\tfrac14 E_T+CM^CT^{1-\sigma}.
\]
Choose a deterministic $T_g$ with $C\sqrt{T_g}\le1/4$ and absorb.
The same weights give the bound on $[0,t]$ for every $t\le T$:
the logarithmic term has bounded time norms, and no terminal condition
on $k(t)$ is needed.
Finally $r=hv$, $w=\mathcal Tr$ and $u=r-q$, with
$\|q(t)\|_2\le CM^Ct^\theta$, give the stated pointwise bounds.
\end{proof}

\begin{lemma}[Quartic entropy]\label{state:quartic}
On the same deterministic interval,
\begin{equation}\label{state:quartic-bound}
 \sup_{t\le T}\|v(t)\|_4^4+
 \int_0^T\bigl(\|v\nabla v\|_2^2+
                         \|P(r^3)\|_2^2\bigr)\dd t
 \le CM^C.
\end{equation}
\end{lemma}
\begin{proof}
\emph{The quartic test.}
Apply
\eqref{cal:general-entropy} with $\Phi(v)=v^4/4$. Its weight error is
$-\frac14\langle 3a_++a_-,hk v^4\rangle$.
The gradient dissipation is $3\nu\int hk v^2|\nabla v|^2$.
At a fixed time write
\[
 U=\|v\|_4,\qquad D=\|v\nabla v\|_2,\qquad
 H=\|P(r^3)\|_2,\qquad R=\|r\|_2.
\]
Since $r\in E_N$,
\begin{equation}\label{state:projected-coercivity}
 \|r\|_4^4=\langle P(r^3),r\rangle\le HR.
\end{equation}
Hence, for $0<p<2$,
\begin{equation}\label{state:quartic-absorption}
 aU^{4p}\le\varepsilon H^2+
       C_\varepsilon a^{2/(2-p)}R^{2p/(2-p)}.
\end{equation}

\emph{The leading cube.}
The ratio $k/h^3$ is close to one by \eqref{cal:gauge-bounds}.
Writing $r^3=P(r^3)+Q(r^3)$ gives
\[
\begin{aligned}
 \langle P((k/h^3)r^3),r^3\rangle
 &=H^2+\langle (k/h^3-1)r^3,P(r^3)\rangle\\
 &\ge\tfrac12H^2-C\|k/h^3-1\|_\infty^2\|Qr^3\|_2^2.
\end{aligned}
\]
Since $r$ has frequencies at most $N$, Bernstein and a high-frequency
Poincar\'e inequality imply
\[
 \|Qr^3\|_2^2
 \le CN^{-2}\|\nabla(r^3)\|_2^2
 \le CN^{-1/2}\|r\|_4^2\|r\nabla r\|_2^2.
\]
The exponential formula for $h$ gives
$\|\nabla h\|_\infty\le CM^CN^{1-\beta}$. Hence the last bound is
\begin{equation}\label{state:cubic-tail}
 CN^{-1/2}U^2D^2+CM^CN^{3/2-2\beta}U^6.
\end{equation}
The second term is absorbed by \eqref{state:quartic-absorption}, since
$6<8$ and $\beta>3/4$. The first is absorbed in $D^2$ while $U^4\le cN$,
for a fixed small $c$.

\emph{The rough lower powers.}
The potential projection error is controlled by
\eqref{cal:quartic-projection}. Treat the rough linear coefficient
$q^2-2Zq$ as another potential: its local term pairs with $hk v^4$,
and its projection term is covered by the same estimate.
For the local term use
\[
 \|v^4\|_{B^\alpha_{1,1}}
 \le C\bigl(U^{4-2\alpha}D^\alpha+U^4\bigr).
\]
Indeed, $\|v|v|\|_2=U^2$, $\|\nabla(v|v|)\|_2=2D$, and
\[
 aU^{4-2\alpha}D^\alpha
 \le\varepsilon D^2+C_\varepsilon a^{2/(2-\alpha)}U^4.
\]
Apply \eqref{state:quartic-absorption}; the weight errors have the same test.

For the rough quadratic coefficient put $F(v)=v|v|$.
The inequality
$|a^3-b^3|\le C(|a|+|b|)|F(a)-F(b)|$
and the boundedness of $P$ on $L^2$ and $L^{3/2}$ give
\[
 \|h^2v^2P(kv^3)\|_{B^\alpha_{1,1}}
       \le C\|v\|_6^3\|F(v)\|_{H^\sigma}.
\]
For a spatial increment $y$,
\[
 \|\delta_y[h^2v^2P(kv^3)]\|_1
 \lesssim M^C\|v\|_6^3
 \bigl(\|\delta_yF(v)\|_2+|y|^\beta\|F(v)\|_2\bigr).
\]
Integrate against $|y|^{-3-\alpha}\dd y$, using $\alpha<\sigma<\beta$. Also,
\[
 \|v\|_6\le C(U^{1/2}D^{1/4}+U),\qquad
 \|F(v)\|_{H^\sigma}\le C(U^{2-2\sigma}D^\sigma+U^2).
\]
The largest mixed term is $a(t)U^{7/2-2\sigma}D^{3/4+\sigma}$.
Use $U^4\lesssim H R$ from \eqref{state:projected-coercivity}.
Its degree relative to $H^2+D^2$ is
\[
 \frac{7/2-2\sigma}{8}+\frac{3/4+\sigma}{2}
 =\frac{13}{16}+\frac\sigma4<1,\qquad \sigma<3/4.
\]
Young leaves a polynomial in $a(t)$ and $R$. The other monomials have
smaller degree. The quadratic entropy bounds $R$; the assumed time norms
of $a$ are sufficient.

\emph{The additive source and the initial layer.}
Every additive source in \eqref{state:additive-equation} lies in
$E_N$, so its pairing with $P(kv^3)$ equals its pairing with $kv^3$.
The same cubic difference inequality gives
\[
 \|v^3\|_{B^\alpha_{1,1}}
 \le C\|v\|_2\bigl(U^{2-2\alpha}D^\alpha+U^2\bigr).
\]
By \eqref{state:projected-coercivity}, a source of norm $K(t)$ gives
\[
 KRU^{2-2\alpha}D^\alpha
 \lesssim KR^{(3-\alpha)/2}H^{(1-\alpha)/2}D^\alpha
 \le\varepsilon(D^2+H^2)+
 C_\varepsilon K^{4/(3-\alpha)}R^2.
\]
The term $KRU^2$ satisfies the same bound with $K$ replaced by $1+K$.
For $bZ$ and $bq$, $K(t)\le CM^Ct^{-1/2}$; this is integrable since
$2/(3-\alpha)<1$. The other additive norms belong to $L^q_t$.
The scalar weight error is bounded by
$\varepsilon H^2+C_\varepsilon(d_+^2+d_-^2)R^2$.
Its initial singularity is integrable because
$R^2(t)\le CM^Ct^{1-\sigma}$ by
\eqref{state:quadratic-bound}; its terminal logarithm is integrable.

\emph{Closing the cutoff restriction.}
The preceding estimates prove \eqref{state:quartic-bound}, with a bound
$CM^A$ independent of $N$, up to the first time $U^4=cN$.
If $N>2CM^A/c$, continuity prevents that first time.
For the remaining cutoffs $N\le 2CM^A/c$, all spatial coefficient
norms are polynomial in $M$ by Bernstein. Testing
\eqref{state:additive-equation} with $r^3$ gives, with $a\le CM^C$,
\[
\begin{aligned}
 \tfrac14\partial_t\|r\|_4^4+3\nu\|r\nabla r\|_2^2+gH^2
 &\le aH(\|r\|_4^2+R+1)\\
 &\le aH^{3/2}R^{1/2}+aH(R+1)\\
 &\le\tfrac g2H^2+C_g(1+a)^4(1+R^2).
\end{aligned}
\]
Here \eqref{eq:constants} bounds the scalar coefficients. Since
$D^2\lesssim\|r\nabla r\|_2^2+M^C\|r\|_4^4$,
integration proves \eqref{state:quartic-bound} also for these cutoffs.
\end{proof}

\subsection{The heat equation supplies the missing derivatives}

\begin{proposition}[State estimate]\label{po:position}
For a deterministic $T_g\in(0,1]$ and $T\le T_g$,
\begin{equation}\label{po:energy}
 \sup_{t\le T}\|w(t)\|_{H^{\epsilon_0}}^2+
 \int_0^T\bigl(\|w\|_{H^{1+\epsilon_0}}^2+\|w\|_4^4+\|u\|_4^4\bigr)\dd t
 \le CM^CT^{2\gamma},\qquad
 \|u(t)\|_2\le CM^Ct^\theta.
\end{equation}
\end{proposition}
\begin{proof}
\emph{The nonlinear source is now square integrable.}
The quartic entropy gives
\[
 \|v^2\|_{L^{10/3}_{t,x}}
 \lesssim\|v^2\|_{L^\infty_tL^2}^{2/5}
          \|v^2\|_{L^2_tH^1}^{3/5}\le CM^C,
 \qquad \|u^3\|_{L^2_{t,x}}\le CM^C,
\]
since $u=hv-q$ and $h,q$ are bounded.

Write $r_0=hv+(q_0-q)$. The entropy bounds and $q_0-q\in C_TC^\beta$
give $r_0,r_0^2\in L^{2/\sigma}_TH^\sigma$.
Since $\sigma>\alpha$ and $-1/2-\epsilon_0/4>-\alpha$,
for $j=0,1,2$,
\[
 \|\zeta_jr_0^{2-j}\|_{L^2_TH^{-\alpha}}
 \lesssim T^{1/2-1/q-\sigma/2}
 \|\zeta_j\|_{L^q_TC^{-1/2-\epsilon_0/4}}
 \|r_0^{2-j}\|_{L^{2/\sigma}_TH^\sigma}\le CM^C.
\]

The covariance gives $\E\|Z(t)\|_{H^{-\alpha}}^2\le Ct^{\epsilon_0}$.
Gaussian increments give $\|Z(t)\|_{H^{-\alpha}}\le CMt^{\epsilon_0/4}$
for all $t$, with $M$ having all finite moments. Thus
\[
 \int_0^T\|bZ\|_{H^{-\alpha}}^2\dd t
 \lesssim M^2\int_0^Tt^{-1+\epsilon_0/2}\dd t\lesssim M^2,
 \qquad \int_0^T|b|^2\|w\|_2^2\dd t
 \lesssim M^C\int_0^Tt^{-\sigma}\dd t\lesssim M^C.
\]
The bounds on $\mathcal T,q_0$ likewise control $b(\mathcal TZ-q_0)$.

\emph{The linear term.} Set
\[
 X=\|w\|_{C_TH^{1-\alpha}}+\|w\|_{L^2_TH^{2-\alpha}},
 \qquad W_0=\|w\|_{C_TL^2},\qquad
 \vartheta=\frac{1+\epsilon_0}{2-\alpha}<1-\frac2q.
\]
Linear heat regularity, \eqref{up:linear-bound}, time H\"older and
Sobolev interpolation give
\[
 X\le CM^C+
 CM^CT^{1/2-1/q-\vartheta/2}X^\vartheta W_0^{1-\vartheta}.
\]
Young gives $X\le X/2+CM^C$, hence $X\le CM^C$.

\emph{Interpolation.}
The quadratic entropy gives
$W_0\le CM^CT^{1/4-\epsilon_0}$ and
$\|w\|_{L^2_TL^2}\le CM^CT^{3/4-\epsilon_0}$.
With $a=(1+\epsilon_0)/(3/2-\epsilon_0)$ and
$c=\epsilon_0/(1/2-\epsilon_0)$,
\[
\begin{aligned}
 \|w\|_{L^2_TH^{1+\epsilon_0}}
 &\le\|w\|_{L^2_TL^2}^{1-a}X^a
 \le CM^CT^{\gamma+\frac{5\epsilon_0}{4(3/2-\epsilon_0)}},\\
 \|w\|_{C_TH^{\epsilon_0}}
 &\le W_0^{1-c}X^c
 \le CM^CT^{\gamma+\frac{\epsilon_0}{4(1/2-\epsilon_0)}}.
\end{aligned}
\]
Both corrections to $\gamma$ are positive.
The $L^4_{t,x}$ bounds follow from the quadratic entropy and bounded
reconstruction, since $1-\sigma>2\gamma$.
Together with Lemma~\ref{state:quadratic}, this proves \eqref{po:energy}.
\end{proof}

\subsection{Local stability of the corrected equation}

\begin{lemma}[Weighted heat gain]\label{up:heat-gain}
Suppose $\|F(t)\|_{C^r}\le H(t)t^{-d}$, with
$H\in L^p(0,T)$, $1<p\le\infty$, and $T\le1$.
Set $a=(\zeta-r)_+/2$. If
\[
 d,a<1-1/p,\qquad \chi=\lambda+1-1/p-d-a>0,
\]
then
\[
 \sup_{t\le T}t^\lambda\|\cI F(t)\|_{C^\zeta}
 \lesssim T^\chi\|H\|_{L^p(0,T)}.
\]
\end{lemma}
\begin{proof}
The heat bound and H\"older reduce the estimate to
\[
 \int_0^t(t-r)^{-ap'}r^{-dp'}\dd r
 =t^{1-(a+d)p'}B(1-dp',1-ap').
\]
It is finite because $ap',dp'<1$.
\end{proof}

\begin{lemma}[A common local solution map]\label{up:local}\label{po:local}
In the regime \eqref{eq:indices}, the mild map of \eqref{po:A8} and
its state derivative have norm at most $1/2$ on the unit ball of
\[
 \mathcal X_T=C_TC^\varpi\cap
 \{w:\sup_{t>0}t^\omega\|w(t)\|_{C^{1+\epsilon_0}}<\infty\}
\]
whenever $T\le cM^{-q_{\rm loc}}$, for deterministic
$c,q_{\rm loc}>0$. The source estimates also hold after one coefficient
derivative.
\end{lemma}
\begin{proof}
Set
$\mathcal F(w)=\cI[\mathcal P_\cdot(w)+b(w+\ell)]$ and
$\|w\|_{\mathcal X_T}=\|w\|_{C_TC^\varpi}
+\sup_{t>0}t^\omega\|w(t)\|_{C^{1+\epsilon_0}}$.
On its unit ball, reconstruction and $\alpha<\sigma<\beta$ give
\begin{equation*}
 \|u\|_{C^\varpi}+\|r_0\|_{C^\varpi}
 +t^\omega\|r_0\|_{C^\sigma}\lesssim M^C.
\end{equation*}
The four heat inputs are
\begin{equation}\label{up:local-source}
\begin{array}{c|ccc}
 \text{source, including outer operators}&p&d&r\\\hline
 \mathcal Au+F-g\mathcal TPu^3&q&0&\varpi-1-\epsilon_0/2\\
 \mathcal TPQ_\zeta(r_0)&q&\omega&-\alpha\\
 \mathcal TP(w\circ W)&\infty&\omega&\epsilon_0/2\\
 b(w+\ell)&\infty&1/2&-\alpha
\end{array}
\end{equation}
A row means $\|\text{source}(t)\|_{C^r}\le H(t)t^{-d}$ with
$\|H\|_{L^p}\lesssim M^C$. Apply Lemma~\ref{up:heat-gain} to
$(\zeta,\lambda)=(\varpi,0)$ and $(1+\epsilon_0,\omega)$.
The strongest kernel restriction is
$1+3\epsilon_0/4-\varpi/2<1-1/q$. The smallest time gain is
\[
 \omega+\tfrac12-\tfrac12(1+\epsilon_0+\alpha)=\varpi-\epsilon_0>0.
\]
Both follow from \eqref{eq:indices}. Moreover,
\[
 \delta u=\mathcal T^{-1}\delta w,\quad
 \delta r_0=(I-K\mathcal T^{-1})\delta w,\quad
 \delta Q_\zeta(r_0)=2(\zeta_0r_0-\zeta_1)\delta r_0.
\]
These obey the same table; coefficient derivatives use
\eqref{up:inverse-jet} and Lemma~\ref{ga:mark-stability}. Hence
\[
 \sup_{\|w\|_{\mathcal X_T}\le1}
 \max\{\|\mathcal F(w)\|_{\mathcal X_T},
       \|D\mathcal F(w)\|_{\mathcal L(\mathcal X_T)}\}
 \le CM^CT^{\varpi-\epsilon_0}\le\tfrac12.
\]
Banach's theorem and finite-dimensional uniqueness identify the fixed point.
\end{proof}

\subsection{The boundary contraction and the continuum field}
\label{full:identification-section}
The scalar $\Delta_N=c_2-c_2^P$ cancels the projection boundary.
We keep both terms until this cancellation is made.

\begin{lemma}[Convergence of the compensated boundary operator]
\label{full:boundary-limit}
Let $\mathcal B_N$ and $\Delta_N$ be as in \eqref{up:boundary}
and \eqref{ga:full-correction}.
For $0<T\le1$, $\alpha_0,\eta_0,\epsilon>0$ and finite $p,q\ge1$,
\[
 \mathcal B_N+\Delta_N I\longrightarrow0
 \quad\text{in }L^p\!\left(\Omega;L^q\!\left(0,T;
 \mathcal L(C^{\alpha_0},C^{-\epsilon})
 \cap\mathcal L(H^{\eta_0},H^{-\epsilon})\right)\right).
\]
\end{lemma}
\begin{proof}
Put $B_N=-\mathcal B_N$. For a fixed Fourier mode $e_\ell$ and large
$N$, the primitive frequency $b$ in $B_Ne_\ell$ has size $N$.
The low--high multiplier is then one. Wick's rule gives
$\E B_Ne_\ell=d_{N,\ell}e_\ell$, where
\[
 d_{N,\ell}=2a^2\sum_{p,q\in\Lambda_N}
 \mathbf1_{p+q+\ell\notin\Lambda_N}
 m(p+q+\ell,-p-q)\frac{I_{pq}}{\lambda_p\lambda_q}.
\]
The translated cube differs from the untranslated cube at at most
$C_\ell N^2$ frequencies. On $|b|\asymp N$,
\[
 |m(b+\ell,-b)-1|\le C_\ell/N,\qquad
 \sum_{p+q=b}(\lambda_p\lambda_qd_{pq})^{-1}
       \le C\langle b\rangle^{-3}.
\]
Since $I_{pq}\le d_{pq}^{-1}$,
\[
 \sup_{t\le T}|d_{N,\ell}-\Delta_N|
 \le C_\ell N^2N^{-3}
 +\frac{C_\ell}{N}\sum_{|b|\asymp N}\langle b\rangle^{-3}
 \le C_\ell/N.
\]

At fixed output $k$, write
$\widehat{B_Ne_\ell-\E B_Ne_\ell}=I_2(f_{N,2})+I_4(f_{N,4})$.
Each kernel contains $\mathbf1_{b+\ell\notin\Lambda_N}$.
For $|b|\gg|k|+|\ell|$, smoothness and $m(b,-b)=1$ bound the shifted
multiplier by twice the unshifted one. The $R_2$ tensors of
Proposition~\ref{ga:model} therefore give, also after contraction,
\[
 |f_{N,n}|\le CF_n\in L^2,\quad f_{N,n}\to0\quad(n=2,4),
 \qquad
 \E|\widehat{B_Ne_\ell-\E B_Ne_\ell}|^2
 =2!\|f_{N,2}\|_2^2+4!\|f_{N,4}\|_2^2\longrightarrow0.
\]
Hypercontractivity gives every finite moment. Put $T_N=B_N-\Delta_NI$.
For $E^a=C^a$ or $H^a$ and a smooth cutoff $S_L$,
\[
 \|T_Ne_\ell\|_{E^{-\epsilon}}
 \lesssim\|S_LT_Ne_\ell\|_{E^{-\epsilon}}
       +L^{-\epsilon/2}\|T_Ne_\ell\|_{E^{-\epsilon/2}}.
\]
Lemma~\ref{ga:quadratic-envelope} bounds the second norm uniformly in
$L^p(\Omega;L^q_t)$; the finite-dimensional first term tends to zero.
On the input side, the same lemma gives
\[
\begin{aligned}
 \|T_N\|_{\mathcal L(C^{\alpha_0},C^{-\epsilon})}
 &\le\|T_NS_L\|_{\mathcal L(C^{\alpha_0},C^{-\epsilon})}
       +C(1+B(t))L^{-\alpha_0},\\
 \|T_N\|_{\mathcal L(H^{\eta_0},H^{-\epsilon})}
 &\le\|T_NS_L\|_{\mathcal L(H^{\eta_0},H^{-\epsilon})}
       +C(1+B(t))L^{-\eta_0}.
\end{aligned}
\]
Take $L^p(\Omega;L^q_t)$ norms, then $N\to\infty$ and $L\to\infty$.
Finally $\mathcal B_N=-B_N$.
\end{proof}

The limiting scalar coefficients are
\begin{equation}\label{full:finite-part-coefficients}
 \kappa_1(t)=a\sum_k\lambda_k^{-1}e^{-2t\lambda_k},\qquad
 \kappa_2(t)=2a^2\sum_{p,q}
       \frac{d_{pq}^{-1}-I_{pq}(t)}{\lambda_p\lambda_q},\qquad
 \overline b=\eta\kappa_1-\eta^2\kappa_2.
\end{equation}
The sums have no output restriction; $d_{pq},I_{pq}$ are defined in
\eqref{ga:full-correction}. With these coefficients and zero initial data,
the paracontrolled equation has a unique local solution \cite{CC},
global for $g>0$ \cite{MW}. Its corrected remainder belongs locally to
$\mathcal X_T$ of Lemma~\ref{po:local} and to $C^{1+\epsilon_0}$ on
compact positive-time intervals.

\begin{proposition}[Identification with the canonical field]
\label{full:identification}
Let $X$ be this unprojected solution. For $0<\tau<T\le T_g$ and every finite $p$,
\[
 X_N\longrightarrow X
 \quad\text{in }L^p\bigl(\Omega;C([\tau,T];H^{-3/4})\bigr).
\]
Consequently $\sup_{t\in[\tau,T]}W_2(\rho_{g,N,t},\Law(X(t)))\to0$.
\end{proposition}
\begin{proof}
\emph{Cancel the scalar correction.}
Define
\[
 \overline{\mathcal A}_N=\mathcal A_N+\eta^2\Delta_N\mathcal T_N,
 \quad \overline F_N=F_N+\eta^2\Delta_N\mathcal T_NZ_N,
 \quad \overline b_N=b_N-\eta^2\Delta_N.
\]
These use $\mathcal B_N+\Delta_N I$, $R_{3,N}^P$, and
$\eta(C_{1,N}-c_{1,N})-\eta^2(C_{2,N}-c_{2,N}^P)$, respectively.
With $\ell_N=\mathcal T_NZ_N-q_{0,N}$,
\[
 \overline{\mathcal L}_N-\mathcal L_N=\eta^2\Delta_NI,
 \qquad \overline f_{0,N}-f_{0,N}=\eta^2\Delta_N\ell_N.
\]
Thus
\begin{equation}\label{full:source-cancellation}
 \overline{\mathcal P}_N(w)+\overline b_N(w+\ell_N)
 =\mathcal P_N(w)+b_N(w+\ell_N),\qquad
 \overline{\mathcal P}_N=\mathcal P_N+\eta^2\Delta_N(w+\ell_N).
\end{equation}
The drift is unchanged. The centered boundary tends to zero by
Lemma~\ref{full:boundary-limit}.

\emph{Convergence of the remaining coefficients.}
For each Fourier chaos kernel, Proposition~\ref{ga:model} and
\eqref{ga:R3-first}--\eqref{ga:shell} give
$|f_N|+|f|\le F\in L^2$ and $f_N\to f$. Hence
\[
 \|I_n(f_N)-I_n(f)\|_{L^p(\Omega)}
 \le C_{n,p}\|f_N-f\|_2\longrightarrow0.
\]
The preceding smooth-cutoff argument, with the time-increment bounds
of Proposition~\ref{ga:model}, gives convergence in the norms
\eqref{ga:base}, \eqref{ga:rough-integrated}, with any extra spatial loss.
For the integrated norms, Minkowski gives the same conclusion in every
finite $L^p(\Omega;L^q_t)$.
Equation~\eqref{ga:Omega-identity} gives convergence of $\Omega_J$
and its initial weight in \eqref{po:q0}.

For the scalars in \eqref{full:finite-part-coefficients},
\[
 C_{1,N}-c_{1,N}(t)\longrightarrow\kappa_1(t),\qquad
 C_{2,N}-c_{2,N}^P(t)\longrightarrow\kappa_2(t).
\]
The summable deficit bound in \eqref{ga:initial-layer} is uniform for
$t\ge\tau>0$; it gives locally uniform convergence and
$|\kappa_1(t)|\lesssim t^{-1/2}$,
$|\kappa_2(t)|\lesssim\log(e/t)$.
With \eqref{full:source-cancellation}, these are the coefficients of the limiting equation.

\emph{Stability, including the initial layer.}
Choose $0<\delta<\epsilon_0/8$ and work in
\[
 \mathcal X_T^\delta=C([0,T];C^{\varpi-\delta})\cap
 \{w:\sup_{0<t\le T}t^\omega\|w(t)\|_{C^{1+\epsilon_0-\delta}}<\infty\}.
\]
Pass to an almost surely convergent coefficient subsequence.
Its norms are bounded; choose a common threshold $J$. Then
\[
 \mathcal T_N^{-1}-\mathcal T^{-1}
 =\mathcal T_N^{-1}(K-K_N)\mathcal T^{-1}\longrightarrow0.
\]
The product and heat inequalities remain strict:
$1+\epsilon_0-\delta-(1+\epsilon_0/4)>\epsilon_0/2$
after a projection loss $<\epsilon_0/8$; equal input and output losses
leave the affine heat exponent unchanged.
On $[t_0,T]$, let $\varepsilon_N(t_0)$ be the sum of the coefficient differences
in the norms of \eqref{up:local-source}, including operator norms.
Then $\varepsilon_N(t_0)\to0$.

For the initial interval $[0,t_0]$, a row of
\eqref{up:local-source} contributes at most
\[
 C_Rt_0^{\min\{1-1/p-d,\,\lambda+1-1/p-d-a\}},
 \qquad a=(\zeta-r)_+/2,
\]
also after a state derivative. For $t\le2t_0$ this is
Lemma~\ref{up:heat-gain}; for $t>2t_0$, use
$(t-r)^{-a}\lesssim t^{-a}$ before integrating over $[0,t_0]$.
Both powers are positive. Hence, for some $\chi>0$,
\[
 \|\mathcal F_N-\mathcal F\|_{C^1(B_R;\mathcal X_T^\delta)}
 \le C_R\{\varepsilon_N(t_0)+t_0^\chi\},\qquad \varepsilon_N(t_0)\longrightarrow0.
\]
Let $N\to\infty$, then $t_0\downarrow0$.

On the common contraction interval,
\[
 \|w_N-w\|_{\mathcal X_T^\delta}
 \le2\|\mathcal F_N(w)-\mathcal F(w)\|_{\mathcal X_T^\delta}
 \longrightarrow0.
\]
On a compact positive-time interval, choose
$R>1+\sup_t\|w(t)\|_{C^{1+\epsilon_0}}$.
The same estimates give a common restart length $c_RM^{-q_{\rm loc}}$;
finitely many restarts cover $[\tau,T]$.
Reconstruction and the subsequence criterion give
$X_N\to X$ in probability in $C([\tau,T];H^{-3/4})$.

Finally, \eqref{po:energy}, \eqref{po:q0} and the Gaussian bounds give
\[
 \sup_N\E\sup_{t\le T_g}\|X_N(t)\|_{H^{-3/4}}^m<\infty
 \qquad(m<\infty).
\]
Put $D_N=\sup_{t\in[\tau,T]}\|X_N(t)-X(t)\|_{H^{-3/4}}$.
For $m>p$ and $\delta>0$,
\[
 \E D_N^p\le\delta^p+(\E D_N^m)^{p/m}
                   \mathbb P(D_N>\delta)^{1-p/m}\longrightarrow\delta^p,
 \qquad
 \sup_{t\in[\tau,T]}W_2(\rho_{g,N,t},\Law(X(t)))^2\le\E D_N^2.
\]
Let $\delta\downarrow0$.
\end{proof}

\begin{remark}[Which finite part]\label{full:finite-part}
The prescribed $C_{2,N}$ restricts $p+q$ to $\Lambda_N$.
The coefficients $\kappa_1,\kappa_2$ have no output restriction.
The difference $C_{2,N}^{\rm full}-C_{2,N}$ sums over
$p,q\in\Lambda_N$ with $p+q\notin\Lambda_N$.
This difference and $\Delta_N(t)$ both converge to
\[
 \Delta_\infty=\frac{2a^2}{(4\pi^2\nu)^3}
 \int_{\substack{x,y\in[-1,1]^3\\|x+y|_\infty>1}}
 \frac{dx\,dy}{|x|^2|y|^2(|x|^2+|y|^2+|x+y|^2)}\in(0,\infty).
\]
The second convergence is uniform on each $[\tau,T]$ with $\tau>0$.
To justify the singular Riemann sum, fix $0<\delta<1/4$.
If $0<|p|\le\delta N$ and $p+q\notin\Lambda_N$, then $|q|\asymp N$
and $q$ lies in a boundary strip containing at most $C|p|N^2$ points.
Thus
\[
 \sum_{\substack{p,q\in\Lambda_N,\ p+q\notin\Lambda_N\\
                         0<|p|\le\delta N}}
 \frac1{\lambda_p\lambda_qd_{pq}}
 \le \frac C{N^2}\sum_{0<|p|\le\delta N}
               \frac{|p|}{\langle p\rangle^2}
 \le C\delta^2.
\]
The term $p=0$ is absent. The same bound holds near $q=0$ and for the
corresponding regions of the integral. Away from these axes, scaling
$p=Nx$, $q=Ny$ gives an ordinary Riemann sum; its boundary has measure
zero and $\mu/N^2\to0$. Let $N\to\infty$, then $\delta\downarrow0$.
The integral is positive because its domain contains an open set.
Finally, the deficit bound used in \eqref{ga:initial-layer} gives
\[
 \sup_{t\in[\tau,T]}
 \bigl|(C_{2,N}^{\rm full}-C_{2,N})-\Delta_N(t)\bigr|
 \le C\delta^2+C e^{-c\tau\delta^2N^2}.
\]
The same two limits prove the asserted uniform convergence.

By Lemma~\ref{full:boundary-limit}, the projection boundary
$\mathcal B_N$ converges to $-\Delta_\infty I$. It therefore contributes
a mass term of size $\eta^2\Delta_\infty$, and
\eqref{full:source-cancellation} shows that this term cancels the
difference of the two constants. The projection in $P_N(X_N^3)$ thus
cancels the output restriction in \eqref{eq:constants}.
The limit $X$ uses the finite part of $C_{2,N}^{\rm full}$.
Another regularization may change the finite mass, but the leading
constant $K(a,\nu)$ is independent of $\mu$.
\end{remark}

\section{Gaussian endpoint alignment}\label{ma:section}

Let $Z=\sqrt q\,W(u)$, with $q>0$ and $\|u\|=1$. Then
\[
 \dot Z=\frac{\dot q}{2q}Z+\sqrt q\,W(\dot u).
\]
The first term changes the variance; the second changes the realization. Since $\langle u,\dot u\rangle=0$, the skew operator
$\mathsf A=u\otimes\dot u-\dot u\otimes u$ satisfies
\begin{equation}\label{gaalign:row-principle}
 \mathsf Au=-\dot u,\qquad \|\mathsf A\|=\|\dot u\|,
 \qquad \dot Z+\mathscr W_{\mathsf A}Z=\frac{\dot q}{2q}Z.
\end{equation}
The rotation removes the second term without changing the law.
If $\dot S=ST$, differentiating the endpoint row $\sqrt q\,u^*$ gives
$T^*u=(\dot q/2q)u+\dot u$; orthogonal projection yields
$\|\dot u\|\le\|T\|$.
Since $\Gamma(O)W(v)=W(Ov)$, the rotated factor is $SO^*$
and its derivative is $\dot S-S\mathsf A$.

Dots initially mean $t\partial_t$ at fixed normalized time $r\in[0,1]$.
On $\mathcal K=L^2(0,1)$ put, for $z=t\lambda\ge0$,
\[
 (V_zf)(r)=\int_0^r e^{-z(r-v)}f(v)\,dv,\qquad
 u_z(r)=\sqrt{\frac{2z}{1-e^{-2z}}}e^{-z(1-r)},\quad u_0=1.
\]
The zero-start OU path is realized by $S_t=\sqrt{2at}\,V_z$.
Its endpoint row is $\sqrt{q_t}u_z$,
where $q_t=a\lambda^{-1}(1-e^{-2t\lambda})$, with $q_t=2at$ at $\lambda=0$.

\begin{lemma}[OU differentiation]\label{gaalign:ou}\label{ma:factor}
Define
\begin{equation}\label{gaalign:connection}
 \mathsf A_t=u_z\otimes\dot u_z-\dot u_z\otimes u_z,
 \qquad \mathfrak D_t=t\partial_t+\mathscr W_{\mathsf A_t}.
\end{equation}
Then $\|\mathsf A_t\|\le1/2$, and
\begin{align}
 \dot S_t&=S_tT_z^0,\quad T_z^0=\tfrac12I-zV_z,\quad
                \|T_z^0\|\le\tfrac12,\label{gaalign:cayley}\\
 \mathfrak D_tS_t&:=\dot S_t-S_t\mathsf A_t
       =S_t(T_z^0-\mathsf A_t),\quad
       \|T_z^0-\mathsf A_t\|\le1,\label{gaalign:marked-factor}\\
 \mathfrak D_t Z_t(1)&=tB_tZ_t(1).\label{ma:endpoint}
\end{align}
Set $tB_t=1/2$ when $\lambda=0$. These operator bounds are uniform over spatial modes.
At finite spatial cutoff, the paths are locally $C^1$ in $t$ and in the
rotation parameter, in every $L^p(C^\alpha[0,1])$ with
$p<\infty$, $\alpha<1/2$.
\end{lemma}
\begin{proof}
The identity $\partial_zV_z=-V_z^2$ gives \eqref{gaalign:cayley}. For $v=V_zf$,
$v'+zv=f$, $v(0)=0$, and integration by parts gives
\[
 \|(I-2zV_z)f\|_2^2=\|f\|_2^2-2z|v(1)|^2.
\]
Thus $\|T_z^0\|\le1/2$. Apply \eqref{gaalign:row-principle} to obtain
$\|\mathsf A_t\|\le1/2$ and the factor correction $-S_t\mathsf A_t$.
Since $\dot q_t/(2q_t)=z/(e^{2z}-1)$, the same identity proves
\eqref{ma:endpoint}.

For $g_{t,r}(v)=\sqrt{2at}\mathbf1_{v\le r}e^{-t\lambda(r-v)}$,
\[
 \partial_tg=t^{-1}(T_z^0)^*g,\qquad
 \partial_t^2g=t^{-2}\bigl[((T_z^0)^*)^2-(T_z^0)^*
                  +\dot T_z^{0*}\bigr]g,
 \qquad \dot T_z^0=-zV_z+z^2V_z^2.
\]
Thus, for $I\Subset(0,\infty)$ and $j\le2$,
\begin{equation}\label{gaalign:increments}
 \|\partial_t^j(g_{t,r}-g_{t,v})\|_2^2
 \le C_I\|g_{t,r}-g_{t,v}\|_2^2\le C_Ia|r-v|.
\end{equation}
For $p$ large enough, the dyadic increments give
\[
 \|\partial_t^jZ_t\|_{L^p(\Omega;C^\alpha)}
 \le C_{I,p}\sum_{n\ge0}2^{-n(1/2-\alpha-1/p)}<\infty,
 \qquad
 Z\in L^p\bigl(\Omega;W^{2,p}(I;C^\alpha)\bigr)
 \hookrightarrow L^p\bigl(\Omega;C^1(I;C^\alpha)\bigr).
\]
Use a larger H\"older index below $1/2$ to obtain the separable
little-H\"older space. The same bounds hold for the first two derivatives
of $e^{\varepsilon\mathsf A_t}g_{t,r}$ and for finite spatial sums.
\end{proof}

\begin{lemma}[Marked coefficients]\label{ma:model}
The coefficient bounds in Appendix~\ref{po:section} hold on $[0,t]$,
also after one $\mathfrak D_t$ derivative, with $M=M_{N,t}\ge1$ satisfying
\begin{align}
 \sup_{N,\,0<t\le1}\mathbb E M_{N,t}^p&<\infty\quad(p<\infty),
       \label{ma:moments}\\
 \|q_0(t)\|_2+\|\mathfrak D_tq_0(t)\|_2
       &\le CM_{N,t}^Ct^\theta.\label{ma:corrector-jet}
\end{align}
Differentiate at fixed thresholds, then choose them.
\end{lemma}
\begin{proof}
Apply Lemma~\ref{ga:mark-stability} using \eqref{gaalign:marked-factor},
then use \eqref{po:q0}. The expectation is taken at each fixed horizon.
\end{proof}
For the deterministic scalar, \eqref{ga:initial-layer} gives
\begin{equation}\label{ma:initial}
 |b(s)|+|sb'(s)|\le C_gs^{-1/2},\qquad
 \mathfrak D_tb(tr)=tr\,b'(tr).
\end{equation}

At fixed $N,t$, use noise factor $S_te^{-\varepsilon\mathsf A_t}$ and
solve $\partial_ru=-t[Au+F_N(u+Z)]$, $u(0)=0$.
The endpoint law is independent of $\varepsilon$. Testing against $u$ gives
\[
 \tfrac12\partial_r\|u\|_2^2+t\langle u,Au\rangle
       +c_gt\|u\|_4^4
 \le C_Nt\bigl(1+\|Z\|_4^4+\|u\|_2^2\bigr).
\]
For $c_N=\max(0,3gC_{1,N}-9g^2C_{2,N})$, each parameter derivative satisfies
\[
 \partial_rh=-t(A+DF_N(U))h+f,\quad h(0)=0,
 \qquad \|h\|_{C_rL^2}\le e^{c_Nt}\|f\|_{L^1_rL^2}.
\]
Here $f$ is polynomial in $u,Z$ and the corresponding derivative of $Z$.
Gronwall gives finite $L^p$ norms for these derivatives. In physical time,
\[
 \partial_s(\mathfrak D_tu)=-(A+DF_N(U))\mathfrak D_tu
 -Au-F_N(U)-DF_N(U)\mathfrak D_tZ,\qquad \mathfrak D_tu(0)=0.
\]
Thus Lemma~\ref{common:ward} applies. Equation~\eqref{ma:endpoint}
cancels the Gaussian term; Section~\ref{co:section} bounds the remainder uniformly in $N$.

\begin{proposition}[A differentiable realization of the full law]
\label{ma:full-differentiation}
For some $t_0>0$, the laws $\rho_{g,t}$ admit a realization
$Y\in C^1_{\rm loc}((0,t_0];L^2(\Omega;H^{-3/4}))$ such that
\[
 \|t\partial_tY_t-tB_tY_t\|_{L^2(\Omega;H^{-3/4})}
 \le Ct^\theta.
\]
\end{proposition}
\begin{proof}
Choose $t_0$ below the state and response horizons and fix $t_*\in(0,t_0)$.
On the direct sum of the mode noise spaces, solve
$tO_t'=O_t\mathsf A_t$, $O_{t_*}=\Id$ on $(0,t_0]$.
Work on any $I\Subset(0,t_0]$.
For the normalized solution $X_{N,t}(r)$, put
$Y_{N,t}=\Gamma(O_t)X_{N,t}(1)$.
The operators $O_t$ are orthogonal and
\begin{equation}\label{ma:integrated-alignment}
 t\partial_tY_{N,t}
 =\Gamma(O_t)\mathfrak D_tX_{N,t}(1),\qquad
 t\partial_tY_{N,t}-tB_tY_{N,t}=\Gamma(O_t)D_N(t).
\end{equation}
Indeed, for $z=t\lambda$ and $\xi=z(1-r)$,
\[
 t\partial_tu_z=\left(\tfrac12-\frac z{e^{2z}-1}-\xi\right)u_z,
 \qquad \int_0^1\xi^m|u_z|^2\,dr\le C_m,
 \qquad \|\mathsf A_t\|+\|t\partial_t\mathsf A_t\|\le C.
\]

Use the compensated coefficients of \eqref{full:source-cancellation}.
For $j=0,1,2$, differentiation of the OU and heat factors gives
\[
 (t\partial_t)^jg_{t,r}(v)
   =P_j(t\lambda(r-v))g_{t,r}(v),\qquad
 |(t\partial_t)^j(te^{-t\lambda a})|
   \le C_jt e^{-t\lambda a/2}.
\]
Here $P_0=1$, $P_1(x)=1/2-x$, $P_2(x)=(1/2-x)^2-x$.
Keep each compensated contraction together. For $a,b,x\ge0$,
\[
 |(t\partial_t)^j(e^{-tax}-e^{-tbx})|
 \le C_j\min\{t|a-b|x,1\}e^{-t\min(a,b)x/2}.
\]
This preserves the four sums in \eqref{ga:shell}.
Thus, on each fixed Fourier block, the differentiated kernels converge
in $L^2$ by dominated convergence, also for
$\mathcal B_N+\Delta_N I$. For its mean, the counting in
Lemma~\ref{full:boundary-limit} gives
\[
 |(t\partial_t)^jI_{pq}(tr)|\le C_jd_{pq}^{-1}
 \quad\Longrightarrow\quad
 |(t\partial_t)^j[d_{N,\ell}(tr)-\Delta_N(tr)]|\le C_\ell N^{-1}.
\]
Rotation adds bounded operators on the tensor slots and preserves these
convergences. A spatial loss $0<\delta<\epsilon_0/8$ makes the remaining
block tails summable; the input truncation in that lemma handles the
operator norm.

More precisely, for each $r_0>0$ this gives convergence in
$L^p(\Omega;C^1(I))$ in the norms \eqref{ga:base},
\eqref{ga:rough-integrated}, \eqref{ga:omega-norms} and \eqref{po:q0},
with every spatial index decreased by $\delta$, and in both operator
norms of Lemma~\ref{full:boundary-limit}, on $r\in[r_0,1]$.
Use $W^{2,p}(I)\hookrightarrow C^1(I)$, $p>1$.
The stochastic norms and their first horizon derivatives remain uniformly
bounded in $L^p(\Omega)$ on $0\le r\le1$.
The scalar coefficients converge there in $C^1(I;C_r)$; the same
exponential bounds give
\[
 |(t\partial_t)^j\overline b_N(tr)|\le C_Ir^{-1/2},
 \qquad j=0,1,2,
\]
for the scalar in \eqref{full:source-cancellation}.

At a fixed horizon, take an almost surely convergent coefficient
subsequence. Freeze the thresholds and use the contraction intervals
of Proposition~\ref{full:identification}. Localize to norms at most $R$
and a horizon neighborhood $U$ of radius at least $R^{-1}$.
On a contraction interval $[0,\ell]$, let $\mathcal X_\ell^\delta$ be
that proposition's weighted space in normalized time.
The positive heat gains of Lemma~\ref{po:local} give
\[
 \|\mathcal F_N-\mathcal F\|_{C^1(U;C^1(B_R;\mathcal X_\ell^\delta))}
 \le C_R\{\varepsilon_N(r_0)+r_0^\chi\},
 \qquad \varepsilon_N(r_0)\longrightarrow0,\quad\chi>0.
\]
The term $r_0^\chi$ bounds the initial interval, including its parameter
derivative; no uniform convergence of the weighted scalar layer is used.
Let $N\to\infty$, then $r_0\downarrow0$.
For $h_N=\partial_tw_N$, $h=\partial_tw$, differentiation of the fixed point gives
\[
 h_N-h=(I-D\mathcal F_N(w_N))^{-1}
 \bigl[\partial_t\mathcal F_N(w_N)-\partial_t\mathcal F(w)
 +(D\mathcal F_N(w_N)-D\mathcal F(w))h\bigr]\longrightarrow0,
\]
with inverse norm at most two. Finite restarts propagate this identity,
including the derivative of the restart data. The limiting trajectory
at the fixed horizon supplies such a neighborhood almost surely.
Let $R\to\infty$ and use the subsequence criterion. Endpoints and
derivatives converge in probability at each horizon; the limiting
derivative $V_t$ is continuous in probability.

By \eqref{ma:integrated-alignment}, \eqref{co:defect-rate},
\eqref{po:energy} and the Gaussian moments,
$\sup_{N,t\in I}\|\partial_tY_{N,t}\|_{L^p}<\infty$ for every finite $p$.
Uniform integrability gives $\partial_tY_{N,t}\to V_t$ in $L^2$ and
$V\in C(I;L^2)$. Passing to the limit in
$Y_{N,t}-Y_{N,s}=\int_s^t\partial_rY_{N,r}\,dr$ gives
$Y_t-Y_s=\int_s^tV_r\,dr$. Hence $Y\in C^1(I;L^2)$ and
\eqref{ma:integrated-alignment} yields the asserted defect bound.
Finally $\Law(Y_t)=\rho_{g,t}$ by Proposition~\ref{full:identification}.
\end{proof}

\section{The coordinate and its adjoint}\label{ac:section}

Let $P=P_N$ be the cube projection. Its $L^p$ norm is bounded uniformly
in $N$ for $1<p<\infty$.

Let $Y\in C^\beta$, $0<\beta<1$, and put
\[
 K_Jf=P\sum_{j\ge J}(S_{j-2}Pf)\Delta_jY,
 \qquad \mathcal T=I+K_J.
\]
\begin{lemma}[Smoothing and bounded inversion]
Write $B=\|Y\|_{C^\beta}$. For $1<p<\infty$ and $0\le a<\beta$,
\begin{equation}\label{ac:c:eq:K-smoothing}
 \|K_J\|_{L^p\to H^{a,p}}+
 \|K_J^*\|_{L^p\to H^{a,p}}
 \le CB2^{-(\beta-a)J}.
\end{equation}
For $b\ge0$ and $a<\beta-b$,
\begin{equation}\label{ac:c:eq:negative-inverse}
 \|K_Jf\|_{H^a}
 \le CB2^{-(\beta-b-a)J}\|f\|_{H^{-b}}.
\end{equation}
The corresponding Besov estimates hold with a strict regularity loss.
A single choice $2^J\le C(1+B)^C$ makes $\mathcal T$ and $\mathcal T^*$
invertible on the finitely many spaces used below.
In particular, when $\beta>3/4$, the maps
$K_J:L^4\to L^\infty$ and $K_J:H^{-3/4}\to L^2$ are bounded.
\end{lemma}
\begin{proof}
The frequency support of $(S_{j-2}Pf)\Delta_jY$ gives
\[
 \|K_Jf\|_{H^{a,p}}
 \le C\sum_{j\ge J}2^{aj}\|S_{j-2}Pf\|_p\|\Delta_jY\|_\infty
 \le CB\|f\|_p\sum_{j\ge J}2^{-(\beta-a)j}.
\]
For negative input,
\[
 \|S_{j-2}Pf\|_2\le C2^{bj}\|f\|_{H^{-b}},\qquad
 \|K_Jf\|_{H^a}
 \le CB\|f\|_{H^{-b}}\sum_{j\ge J}2^{-(\beta-b-a)j}.
\]
The adjoint satisfies
\[
 K_J^*f=P\sum_{j\ge J}S_{j-2}(\Delta_jY\,Pf),\qquad
 \|K_J^*f\|_{H^{a,p}}
 \le CB\|f\|_p\sum_{j\ge J}2^{-(\beta-a)j}.
\]
Take $a=-b$ for $K_J$ on negative Sobolev spaces, then use duality
for $K_J^*$. The first estimate treats positive indices below $\beta$.
For the finitely many spaces $E$ used below, choose $J$ so that
\[
\begin{gathered}
 \max_E\bigl(\|K_J\|_{E\to E}+\|K_J^*\|_{E\to E}\bigr)
 \le CB2^{-\delta J}\le\tfrac12,\\
 (I+K_J)^{-1}=\sum_{n\ge0}(-K_J)^n,\qquad
 \|(I+K_J)^{-1}\|_{E\to E}\le2,
\end{gathered}
\]
Here $\delta>0$ is the smallest exponent in these bounds.
The adjoint has the same Neumann series.
Finally choose $3/4<a<\beta$ in \eqref{ac:c:eq:K-smoothing} and use
$H^{a,4}\hookrightarrow L^\infty$; take $b=3/4$, $a=0$ in
\eqref{ac:c:eq:negative-inverse} for the other assertion.
\end{proof}

\section{Explicit positive weights}\label{ac:reference-section}

Exponentials of the heat primitives cancel the rough potential.
Differentiating them produces a gradient square. The second counterterm
cancels the divergent part of its mean. We use $\eta=3g$ and the exponents in
\eqref{eq:indices}.

\subsection{The gradient square and its scalar part}
For a horizon $T\le1$, use the full-output heat primitives
\begin{equation}\label{cal:primitives}
 Y_+(t)=\int_0^t e^{-(t-r)A}W(r)\dd r,
 \qquad Y_-(t)=\int_t^T e^{-(r-t)A}W(r)\dd r.
\end{equation}
Both primitives use the same $W$ and the covariance bound
\eqref{ga:covariance}.

\begin{lemma}[Centered gradient square]\label{cal:gradient-square}
For every $\varepsilon>0$ and finite $p,q$, uniformly in $N,T$ and each
fixed dyadic threshold $J$,
\begin{equation}\label{cal:gradient-moments}
 \left\|\nu|\nabla H_JY_\pm|^2
       -\nu\E|\nabla H_JY_\pm|^2\right\|_{L^p(\Omega;L^q_TC^{-\varepsilon})}
 \le C_{p,q,\varepsilon}T^{1/q}.
\end{equation}
With $H_J$ omitted, the scalar parts satisfy
\begin{align}
 |C_{2,N}-\nu\E|\nabla Y_+(t)|^2|
 &\le C[1+\log(e/t)],\label{cal:scalar-forward}\\
 |C_{2,N}-\nu\E|\nabla Y_-(t)|^2|
 &\le C[1+\log(e/t)+\log(e/(T-t))].\label{cal:scalar-backward}
\end{align}
Inserting $H_J$ adds at most $C(1+J)$ to these bounds.
\end{lemma}
\begin{proof}
\emph{The fluctuating part.}
Put $F=\nu|\nabla Y_\pm|^2-\nu\E|\nabla Y_\pm|^2=F^{(4)}+F^{(2)}$.
By \eqref{gm:heat},
$\E|\widehat{\nabla Y_\pm}(k)|^2\lesssim\langle k\rangle^{-3}$.
The product formula and Cauchy--Schwarz give
\[
 \E|\widehat F^{(4)}(k)|^2
 \lesssim\sum_b\E|\widehat{\nabla Y_\pm}(b)|^2
                   \E|\widehat{\nabla Y_\pm}(k-b)|^2
 \lesssim\langle k\rangle^{-3}\log(e+\langle k\rangle).
\]
For the second chaos, let $a+b=k$ be the unpaired frequencies
and $\ell$ the contracted frequency.
Put $u_\ell=\langle a+\ell\rangle$, $v_\ell=\langle b-\ell\rangle$.
Its kernel norm is bounded by $C S(a,b)/(\langle a\rangle\langle b\rangle)$, where
\[
\begin{aligned}
 S(a,b)&=\sum_\ell\frac{u_\ell v_\ell}{\langle\ell\rangle^2}
 \int_0^\infty\!\int_0^\infty
 e^{-u_\ell^2x-v_\ell^2y-\langle\ell\rangle^2|x-y|}\dd x\dd y\\
 &=\sum_\ell\frac1{\langle\ell\rangle^2}
 \frac{u_\ell v_\ell}{u_\ell^2+v_\ell^2}
 \left(\frac1{u_\ell^2+\langle\ell\rangle^2}
       +\frac1{v_\ell^2+\langle\ell\rangle^2}\right).
\end{aligned}
\]
For $K=\max(\langle a\rangle,\langle b\rangle)=\langle b\rangle$,
\[
 S(a,b)\lesssim
 \sum_{|\ell|>2K}\langle\ell\rangle^{-4}+K^3K^{-4}
 +K^{-1}\!\sum_{|\ell|<K/2}\langle\ell\rangle^{-3}
 +K^{-2}\!\sum_{|\ell|<K/2}\langle\ell\rangle^{-2}
 \lesssim K^{-1}\log(e+K).
\]
Consequently,
\[
 \E|\widehat{(\nu|\nabla Y_\pm|^2)^{(2)}}(k)|^2
 \le C\sum_{a+b=k}\frac{\log^2(e+K)}{
              \langle a\rangle^2\langle b\rangle^2K^2}
 \le C\langle k\rangle^{-3}\log^2(e+\langle k\rangle).
\]
For $r\ge\max(p,q)$ and $3/r<\varepsilon$, hypercontractivity and Bernstein give
\[
 \|\Delta_jF(t)\|_{L^r(\Omega;L^\infty)}\lesssim_r(1+j)2^{3j/r},
 \qquad
 \|F\|_{L^p(\Omega;L^q_TC^{-\varepsilon})}
 \lesssim_r T^{1/q}\sum_{j\ge0}(1+j)2^{-j(\varepsilon-3/r)}
 \lesssim T^{1/q}.
\]
The same bounds hold with $H_J$, in either time orientation.

\emph{The mean cancels the counterterm.}
For $k=p+q$, write $d_{pq}=\lambda_k+\lambda_p+\lambda_q$ and
$\theta_\ell(r)=1-e^{-2\lambda_\ell r}$. Then
\[
 \E|\widehat Y_\pm(t,k)|^2
 =2a^2\!\sum_{\substack{p+q=k\\p,q\in\Lambda_N}}
 \frac{\displaystyle\int_{D_\pm^2}
 e^{-\lambda_k(x+y)-(\lambda_p+\lambda_q)|x-y|}
 \theta_p(r_\pm)\theta_q(r_\pm)\dd x\dd y}{\lambda_p\lambda_q},
\]
where $D_+=[0,t]$, $D_-=[0,T-t]$,
$r_+=t-\max(x,y)$ and $r_-=t+\min(x,y)$.
The stationary integral is $(\lambda_kd_{pq})^{-1}$, obtained by
replacing $D_\pm$ by $[0,\infty)$ and $\theta_p\theta_q$ by $1$. Hence
\[
 0\le\E|\widehat Y_{\rm stat}(k)|^2-\E|\widehat Y_\pm(t,k)|^2
 \lesssim\sum_{p+q=k}
 \frac{e^{-c|D_\pm|\lambda_k}+e^{-ct\lambda_p}+e^{-ct\lambda_q}}
      {\lambda_p\lambda_q\lambda_kd_{pq}}.
\]
Multiply by $4\pi^2\nu|k|^2$, sum, and use
\eqref{cal:sunset-identity}, \eqref{ga:full-correction} and
$\sum_p\langle p\rangle^{-3}e^{-ct\lambda_p}\lesssim\log(e/t)$.
This proves \eqref{cal:scalar-forward}--\eqref{cal:scalar-backward}.
Finally, removing modes below $2^J$ changes the scalar by
$C\sum_{|k|\lesssim2^J}\langle k\rangle^{-3}\le C(1+J)$.
\end{proof}

\subsection{What the weights cancel}
Put $b_0=\eta(C_{1,N}-c_1)-\eta^2C_{2,N}$ and
$L_0=\nu\Delta-\mu-\eta W+b_0$.

\begin{proposition}[Explicit weights and equation errors]
\label{cal:construction}\label{cal:reference}
For every fixed $\delta>0$ sufficiently small, a threshold $J$ can be
chosen with $2^J\le CM^C$, where $M$ has all finite moments uniformly
in $N,T$, so that
\begin{equation}\label{cal:equations}
 h=e^{-\eta H_JY_+},\qquad k=e^{-\eta H_JY_-},\qquad
 h(0)=k(T)=1
\end{equation}
and
\begin{equation}\label{cal:gauge-bounds}
 \tfrac12\le h,k\le2,\qquad
 \|h-1\|_{C_TC^\beta}+\|k-1\|_{C_TC^\beta}\le\delta.
\end{equation}
Writing $Y_{\pm,h}=H_JY_\pm$ and $W_l=(I-H_J)W$, their equation
errors are exactly
\begin{align}
 h'&=L_0h+a_+h,&-k'&=L_0k+a_-k,\notag\\
 a_\pm&=\mu+\eta\mu Y_{\pm,h}+\eta W_l-\eta(C_{1,N}-c_1)
       +\eta^2(C_{2,N}-\nu|\nabla Y_{\pm,h}|^2).
       \label{cal:explicit-defects}
\end{align}
Each error is the sum of a scalar $d_\pm$ and a distribution $r_\pm$,
where for every finite $q$ and $\varepsilon>0$,
\begin{equation}\label{cal:defect-bounds}
 \|r_\pm\|_{L^q_TC^{-\varepsilon}}\le CM^C,
 \qquad
 |d_\pm(t)|\le C[1+J+t^{-1/2}+\log(e/t)+\log(e/(T-t))].
\end{equation}
Both ratios $k/h,h/k$, and also $k/h^3$, are uniformly close to
one when $\delta$ is small.
\end{proposition}
\begin{proof}
\emph{Choice of the weights.}
The covariance bounds in either orientation give
\[
 \|H_JY_\pm\|_{C_TC^\beta}\lesssim M2^{-\epsilon_0J/4},
 \qquad
 J=\min\{j\ge0:CM2^{-\epsilon_0j/4}\le\delta\}.
\]
Composition by the exponential gives \eqref{cal:gauge-bounds} and $2^J\le CM^C$.

\emph{Equation errors.}
For $h=e^{-\eta Y_{+,h}}$, the chain rule gives
\[
 \frac{h'-\nu\Delta h}{h}
 =-\eta(\partial_t-\nu\Delta)Y_{+,h}
       -\eta^2\nu|\nabla Y_{+,h}|^2.
\]
Substitute $(\partial_t+A)Y_{+,h}=H_JW$; for $k$, substitute
$(-\partial_t+A)Y_{-,h}=H_JW$. This gives \eqref{cal:explicit-defects}.

\emph{Scalar and centered terms.}
For deterministic $J$, put
$d_\pm=-\eta(C_{1,N}-c_1)
 +\eta^2(C_{2,N}-\nu\E|\nabla H_JY_\pm|^2)$.
Set $r_\pm=a_\pm-d_\pm$. Since
$\|L_JW\|_{C^{-\varepsilon}}\lesssim2^{CJ}\|W\|_{C^{-1-\varepsilon}}$,
Lemma~\ref{cal:gradient-square} and \eqref{ga:initial-layer} give
\eqref{cal:defect-bounds}. Enlarge $M$ by
\[
 \sum_{j\ge0}(1+j)^{-2}\sum_\pm
 \left\|\nu|\nabla H_jY_\pm|^2-\nu\E|\nabla H_jY_\pm|^2
 \right\|_{L^q_TC^{-\varepsilon}}.
\]
At the chosen $J$, the centered norm is at most $(1+J)^2M\le CM^C$.
Each expectation above is evaluated before choosing $J$.

\end{proof}

\begin{lemma}[Weight tails]\label{cal:weight-tail-lemma}
Let $a_{\ge cN}$ denote a smooth high-frequency cutoff. For finite $p>1$,
there is $\mathcal G_{N,p}(t)$ such that
\begin{equation}\label{cal:weight-tail}
 \|W(t)\|_p\le N\mathcal G_{N,p}(t),\qquad
 \|a_{\ge cN}(t)\|_p
 \le N^{-1}\mathcal G_{N,p}(t)+CM^CN^{-2\beta},
 \quad a\in\{h,k,k/h,h/k\}.
\end{equation}
For every finite $q,r$, the norms
$\|\mathcal G_{N,p}\|_{L^r(\Omega;L^q(0,T))}$ are bounded uniformly in $N,T$.
\end{lemma}
\begin{proof}
The Fourier bounds and hypercontractivity give
\[
 \|W\|_{L^r(\Omega;L^p)}
 \lesssim_{p,r}\Bigl(\sum_{n\in\Lambda_{2N}}\langle n\rangle^{-1}\Bigr)^{1/2}
 \lesssim N,
 \quad
 \|\Delta_jY_\pm\|_{L^r(\Omega;L^p)}
 \lesssim_{p,r}\Bigl(\sum_{|n|\asymp2^j}\langle n\rangle^{-5}\Bigr)^{1/2}
 \lesssim2^{-j}.
\]

\emph{Paralinearization.} Write $a=F(Y_{+,h},Y_{-,h})$.
The arguments of this smooth exponential are bounded by
\eqref{cal:gauge-bounds}. Bony's formula gives
\[
 a=\sum_\pm\partial_\pm F(Y_{+,h},Y_{-,h})\prec Y_{\pm,h}+R_a,
 \qquad \|R_a\|_{C^{2\beta}}\le CM^C.
\]
The remainder satisfies
$\|(R_a)_{\ge cN}\|_\infty\le CM^CN^{-2\beta}$. A block of the
paraproduct of frequency $2^j\ge cN$ has $L^p$ norm at most
$C\|\Delta_{j'}Y_{\pm,h}\|_p$ with $|j-j'|\le2$. Put
\[
 \mathcal G_{N,p}=N^{-1}\|W\|_p
 +N\sum_\pm\sum_{2^j\ge c'N}\|\Delta_jY_\pm\|_p .
\]
Its moments are bounded, since $N\sum_{2^j\ge c'N}2^{-j}\le C$.
Minkowski gives the time norms.

\emph{The random threshold.} If $2^J\le c'N$, the blocks above $c'N$
of $H_JY_\pm$ are bounded by a fixed finite sum of neighboring blocks of $Y_\pm$, uniformly in $J$, and the bound above applies. If
$2^J>c'N$, then $N\le CM^C$ and the trivial bound
$\|a_{\ge cN}\|_p\le C\le CM^{C}N^{-2\beta}$ holds after enlarging
the power of $M$.
\end{proof}

For the response, $L=L_0-(6gZu+3gu^2)$, so the errors become
$a_\pm+6gZu+3gu^2$. On $\{M\le R\}$, Lemma~\ref{po:local} gives
$\|u\|_{C_TC^\varpi}\le C_R$ for deterministic $T$ small enough, and
\begin{equation}\label{cal:reference-input}
 \|6gZu(t)+3gu(t)^2\|_{C^{-\alpha}}
 \le C_R(1+t^{-d})+\kappa(t),\qquad
 d=\omega\frac{\sigma-\varpi}{1+\epsilon_0-\varpi}<\tfrac14,
 \qquad \|\kappa\|_{L^q_T}\le C_R.
\end{equation}
Write $u=(w-Ku)-q_0$ and $\chi=(\sigma-\varpi)/(1+\epsilon_0-\varpi)$. Then
\[
\begin{aligned}
 \|w-Ku\|_{C^\sigma}
 &\lesssim\|w\|_{C^\varpi}^{1-\chi}\|w\|_{C^{1+\epsilon_0}}^\chi
          +M^C\|u\|_{C^\varpi}
 \le C_R(1+t^{-\omega\chi}),\\
 \|Z(w-Ku)\|_{C^{-\alpha}}
 &\lesssim\|Z\|_{C^{-\alpha}}\|w-Ku\|_{C^\sigma}.
\end{aligned}
\]
Use $Z\Psi=Z\prec\Psi+\Psi\prec Z+S$ in $Zq_0$;
Proposition~\ref{ga:model} controls this term and the smoother $Z\Omega_J$.
Also $\|u^2\|_{C^{-\alpha}}\lesssim\|u\|_{C^\varpi}^2$.

Since $d<1/4$, the rough errors belong to $L^4_TC^{-\alpha}$;
the scalar errors belong to $L^1_T$. For a rough error $a$,
\[
 |\langle a,hkv^2\rangle|
 \le C\|a\|_{C^{-\alpha}}\|v\|_2^{2-\alpha}\|v\|_{H^1}^{\alpha}
 \le\varepsilon\|v\|_{H^1}^2
   +C_\varepsilon\|a\|_{C^{-\alpha}}^{2/(2-\alpha)}\|v\|_2^2.
\]
Since $2/(2-\alpha)<4$, the growth coefficient is integrable.
Choose $R$ with $\mathbb P(M>R)\le R^{-p}\E M^p\le1/4$ and then
$T_*=cR^{-C}$ in Lemma~\ref{po:local}. Intersecting with the projection
event $\Omega_{\rm proj}$ of Appendix~\ref{cal:projection-section} gives
\[
 \mathbb P(\{M\le R\}\cap\Omega_{\rm proj})
 \ge1-\mathbb P(M>R)-\mathbb P(\Omega_{\rm proj}^c)\ge\tfrac12.
\]

\section{Errors from the Galerkin projection}\label{cal:projection-section}

A weighted entropy test need not belong to $E_N$.
For the potential error, one derivative supplies $N^{-1}$ and cancels
the size $N$ of the Wick square. The cubic error is confined to a
narrow frequency strip.

\begin{lemma}[The potential error]\label{cal:projection-form}\label{cal:projection}
Use the weights and reference event of Proposition~\ref{cal:construction}.
Write $V=\eta W+a+c$, where $\eta=3g$,
$a=6gZu+3gu^2$ and $c$ is spatially constant.
For every $\varepsilon>0$ there is a nonnegative coefficient
$B_\varepsilon$, with
$\int_0^T B_\varepsilon(s)\,ds\le C_{R,\varepsilon}$ on an event of
probability at least $1/2$, such that
\begin{equation}\label{cal:projection-Hr}
 |\langle Q(mf),Vf\rangle|
 \le\varepsilon\|f/h\|_{H^1}^2
       +B_\varepsilon(s)\|f/h\|_2^2,
 \qquad m=k/h,\quad f\in E_N.
\end{equation}
The constants are uniform in $N,T$. The estimate also holds with $h,k$ interchanged.
\end{lemma}
\begin{proof}
Put $v=f/h$, $F=\|v\|_2$ and $Z=\|v\|_{H^1}$.
Since $Q(mf)=Q(kv)$, smoothly split $k=k_l+k_h$ and $v=v_l+v_h$
at a sufficiently small fixed multiple of $N$.
The low--low product lies in $E_N$, so
\[
 Q(kv)=Q(k_lv_h)+Q(k_hv),\qquad
 \|v_h\|_2\le CN^{-1}Z.
\]
For the Wick square, covariance and hypercontractivity give
\[
 \|W(s)\|_{L^{12}}\le N G_N(s),\qquad
 \sup_{N,T,s}\|G_N(s)\|_{L^p(\Omega)}<\infty
 \quad(p<\infty).
\]
Indeed $\sum_{n\in\Lambda_{2N}}\E|\widehat W(n)|^2\lesssim N^2$.
By \eqref{cal:weight-tail} and $2\beta>1$,
\[
 \|k_h\|_{12}\le N^{-1}\mathcal G_{N,12}+C_RN^{-2\beta}
 \le C_RN^{-1}\widetilde G_N.
\]
Since $\|v\|_{12/5}\lesssim F^{3/4}Z^{1/4}$,
\begin{align*}
 |\langle Q(k_lv_h),Whv\rangle|
 &\lesssim_R\|v_h\|_2\|W\|_{12}\|v\|_{12/5}
 \lesssim_R G_NF^{3/4}Z^{5/4},\\
 |\langle Q(k_hv),Whv\rangle|
 &\lesssim_R\|k_h\|_{12}\|W\|_{12}\|v\|_{12/5}^2
 \lesssim_R G_N\widetilde G_NF^{3/2}Z^{1/2}.
\end{align*}
Young gives
\[
 G_NF^{3/4}Z^{5/4}+G_N\widetilde G_NF^{3/2}Z^{1/2}
 \le\varepsilon Z^2+C_\varepsilon
       \bigl(G_N^{8/3}+(G_N\widetilde G_N)^{4/3}\bigr)F^2
 \le\varepsilon Z^2+C_\varepsilon(1+\mathcal G_N^4)F^2,
\]
where $\mathcal G_N$ has all finite moments.

For $a$, frequency support in $\Lambda_{2N}$ gives
$\|a\|_\infty\le CN^\alpha\|a\|_{C^{-\alpha}}$.
Since $\|k_h\|_\infty\le C_RN^{-\beta}$ and
$\alpha<\beta<1$, the same decomposition yields
\[
 |\langle Q(kv),ahv\rangle|
 \le C_R\|a\|_{C^{-\alpha}}
       \bigl(N^{\alpha-1}FZ+N^{\alpha-\beta}F^2\bigr)
 \le\varepsilon Z^2+
       C_{R,\varepsilon}(1+\|a\|_{C^{-\alpha}}^2)F^2.
\]
The constant term vanishes exactly:
$\langle Q(mf),cf\rangle=c\langle mf,Qf\rangle=0$.
Thus we may take
\[
 B_\varepsilon=C_{R,\varepsilon}
       (1+\mathcal G_N^4+\|a\|_{C^{-\alpha}}^2).
\]
By \eqref{cal:reference-input}, the last term is integrable on $\{M\le R\}$.
For both orientations, let $L=\sum_\pm\int_0^T\mathcal G_{N,\pm}^4\dd s$.
Then
\[
 \E L\le C T\le C,\qquad
 \mathbb P(L>4C)\le\tfrac14,\qquad
 \mathbb P(\{M\le R\}\cap\{L\le4C\})\ge\tfrac12.
\]
Take $\Omega_{\rm proj}=\{L\le4C\}$ and interchange $h,k$ for the backward bound.
\end{proof}

\paragraph{The two state entropies.}
For the explicit weights, replace $C_R$ by $CM^C$ in the preceding
estimates. Let $r=hv\in E_N$.
For the test $kv$, the two terms satisfy
\[
 |\langle Q(k_lv_h),Wr\rangle|
 \lesssim M^C\|v_h\|_2\|W\|_4\|v\|_4,
 \qquad
 |\langle Q(k_hv),Wr\rangle|
 \lesssim M^C\|k_h\|_4\|W\|_4\|v\|_4^2.
\]
Hence
\begin{equation}\label{cal:quadratic-quotient-projection}
 |\langle Q(kv),Wr\rangle|
 \le\Gamma_N(t)\bigl(\|v\|_4\|v\|_{H^1}+\|v\|_4^2\bigr),
\end{equation}
All moments of $\Gamma_N(t)$ are bounded uniformly in $N,T,t$.
For $p\ge8$, Minkowski gives
\[
 \bigl\|\|\Gamma_N\|_{L^8(0,T)}\bigr\|_{L^p(\Omega)}
 \le\left(\int_0^T\|\Gamma_N(t)\|_{L^p(\Omega)}^8\dd t\right)^{1/8}
 \le C_pT^{1/8}.
\]
For $p<8$, use the estimate with $p=8$. Any other finite time exponent
is obtained in the same way.

Put $U=\|v\|_4$ and $Y=\|v\|_{H^1}$. The two products in
\eqref{cal:quadratic-quotient-projection} have degrees
$1/4+1/2=3/4$ and $2/4=1/2$ relative to $U^4+Y^2$. Hence
\[
 \Gamma_N(UY+U^2)
 \le\varepsilon(U^4+Y^2)+C_\varepsilon(\Gamma_N^4+\Gamma_N^2).
\]
The initial rate is unchanged: $\int_0^T\Gamma_N^4
\le T^{1/2}\|\Gamma_N\|_{L^8}^4$ and $1/2>1-\sigma$.

For the test $kv^3$, put $U=\|v\|_4$ and $G=\|v\nabla v\|_2$.
Let $c$ be supported in $\Lambda_{2N}$ with values in $C^{-\sigma}$,
where $\sigma<\beta$. The same decomposition gives
\begin{equation}\label{cal:quartic-projection}
 |\langle Q(kv^3),Wr\rangle|
 +|\langle Q(kv^3),cr\rangle|
 \le\varepsilon G^2+A_\varepsilon(t)U^4.
\end{equation}
Here $A_\varepsilon$ is polynomial in the Gaussian coefficients and
$\|c\|_{C^{-\sigma}}$; choose $q$ large enough for its time norm.
The low--low product vanishes:
\[
 Q(kv^3)=Q\bigl(k_l(v^3)_{\ge cN}\bigr)+Q(k_hv^3),
 \qquad
 \|(v^3)_{\ge cN}\|_{3/2}
 \le CN^{-1}\|v\|_6G.
\]
Since $\|\nabla(v^3)\|_{3/2}\le3\|v\|_6G$,
\begin{align*}
 |\langle Q(k_l(v^3)_{\ge cN}),Wr\rangle|
 &\lesssim M^C\|(v^3)_{\ge cN}\|_{3/2}\|W\|_{12}U
 \le\Gamma_NU\|v\|_6G,\\
 |\langle Q(k_hv^3),Wr\rangle|
 &\lesssim M^C\|k_h\|_{24}\|W\|_{24}\|v^3\|_{3/2}U
 \le\Gamma_NU\|v\|_{9/2}^3.
\end{align*}
Sobolev and interpolation give
\[
 \|v^2\|_6\lesssim G+U^2,\qquad
 \|v\|_6\lesssim U^{1/2}G^{1/4}+U,\qquad
 \|v\|_{9/2}\lesssim U^{5/6}G^{1/12}+U.
\]
Thus the sum is bounded by
$C\Gamma_N(U^{3/2}G^{5/4}+U^2G+U^{7/2}G^{1/4}+U^4)$.
For each monomial, $a+2b=4$ and $b<2$, so
\[
 \Gamma_NU^aG^b\le\varepsilon G^2
       +C_\varepsilon\Gamma_N^{2/(2-b)}U^4.
\]
For $c$, the same decomposition gives directly
\[
 |\langle Q(kv^3),cr\rangle|
 \lesssim M^C\|c\|_{C^{-\sigma}}
 \bigl(N^{\sigma-1}U\|v\|_6G+N^{\sigma-\beta}U^4\bigr).
\]
Use $\sigma<\beta<1$ and the preceding Young bound.

\begin{lemma}[The cubic error]\label{cal:boundary}
Under \eqref{cal:reference-bounds}, let $f\in E_N$, $v=f/h$ and $m=k/h$.
Then
\begin{equation}\label{cal:boundary-four}
 \|Q(mf)\|_4
 \le C_R\|v\|_2^{1/2}\|v\|_{H^1}^{1/2}.
\end{equation}
The same statement holds with $h,k$ interchanged.
\end{lemma}
\begin{proof}
Put $F=\|v\|_2$, $Z=\|v\|_{H^1}$ and decompose $m$ into smooth dyadic
blocks $m_j$ of frequency $L=2^j$.
For $L\le cN$, $Q(m_jf)$ lies in an exterior strip of $\Lambda_N$
with at most $CN^2L$ frequencies. Only the part of $f$ with
$|n|_\infty\ge c'N$ contributes. For $z=Q(m_jf)$,
\[
 \|z\|_4\le\|z\|_2^{1/2}\|z\|_\infty^{1/2}
 \le(\#\operatorname{supp}\widehat z)^{1/4}\|z\|_2
 \le CN^{1/2}L^{1/4}\|m_j\|_\infty\|f_{\ge c'N}\|_2.
\]
For $f=hv$, smooth low--high decomposition gives
\[
 (hv)_{\ge c'N}=(h_lv_h+h_hv)_{\ge c'N},\qquad
 \|f_{\ge c'N}\|_2\lesssim_R
 \min\{F,\|v_h\|_2+\|h_h\|_\infty F\}
 \lesssim_R\min\{F,N^{-1}Z+N^{-\beta}F\}.
\]
Summing $L^{1/4-\beta}$ over dyadic $L\le cN$ gives
\[
 C_RN^{1/2}\min\{F,N^{-1}Z+N^{-\beta}F\}
 \le C_R\bigl(F^{1/2}Z^{1/2}+N^{1/2-\beta}F\bigr)
 \le C_RF^{1/2}Z^{1/2}.
\]
For $L\ge cN$, use the uniform $L^4$ bound for $Q$ and Bernstein on
$f\in E_N$:
\[
 \|Q(m_{\ge cN}f)\|_4
 \le C_RN^{-\beta}\|f\|_4
 \le C_RN^{3/4-\beta}F\le C_RF.
\]
Use $\beta>3/4$ and $Z\ge F$.
\end{proof}

\section{The KPZ realization and its derivative}\label{kpz:analytic}

We differentiate the logarithm of the stochastic heat solution.
Inverse moments control the denominators. Two coupling derivatives
bound the remainder; a third bounds its first chaos. Set
$P_t=e^{\nu t\partial_x^2}$ and $\mathcal K=L^2((0,1)\times\mathbb T)$.

\subsection{Inverse moments of the stochastic heat solution}

\begin{lemma}\label{kpz:inverse}
If $dZ_b=\nu\partial_x^2Z_b\,dt+bZ_b\,dW$,
$Z_b(0)=1$, then, for every finite $T,K$ and $r>0$,
\[
 \sup_{|b|\le K,\,0<t\le T,\,x\in\mathbb T}
       \mathbb E Z_b(t,x)^{-r}<\infty.
\]
\end{lemma}
\begin{proof}
Use continuous Walsh solutions, comparison, and restart at stopping
times \cite{Walsh,MuellerNualart}.

\emph{One crossing.} For predictable
$|w|\le B$, the convolution $I_w(s,x)=\int_0^sP_{s-u}(x-y)w(u,y)W(du,dy)$
satisfies
\begin{equation}\label{kpz:stopped-tail}
 \mathbb P\{\sup_{s\le\tau,x}|I_w(s,x)|>1/4\}
 \le C e^{-c\tau^{-1/2}},\qquad 0<\tau\le T.
\end{equation}
For $\tau\le1$, put $\delta=|s-s'|$, $d=d(x,x')$ and
$\Delta I_w=I_w(s,x)-I_w(s',x')$. Its increment martingale satisfies
\[
 \begin{aligned}
 \langle\Delta I_w\rangle
 &\lesssim B^2\left(\delta+\sum_{k\ne0}\min\{\delta,k^{-2}\}
       +\min\left\{\sqrt\tau,\sum_{k\ne0}\min\{d^2,k^{-2}\}\right\}\right)
 \\&\lesssim B^2(\sqrt\delta+\min\{\sqrt\tau,d\}).
 \end{aligned}
\]
On $(\tau2^{-4j},\sqrt\tau2^{-2j})$, choose
$a_j=(1-2^{-1/2})2^{-j/2}/4$, so that $\sum_ja_j=1/4$. Then
\[
 \mathbb P\!\left(\|I_w\|_\infty>\tfrac14\right)
 \le C\tau^{-1/2}\sum_{j\ge0}2^{6j}
       \exp\!\left(-\frac{a_j^2}{CB^2\sqrt\tau\,2^{-2j}}\right)
 \le Ce^{-c/\sqrt\tau}.
\]
Continuity extends the bound from the grids; increasing $C$ covers $1<\tau\le T$.

\emph{Successive crossings.} Set $\tau_{-1}=0$.
At $\tau_{n-1}<\infty$, start a solution $v_n$ from
the constant $2^{-n}$ using the same future noise, and let $\tau_n$ be its
first exit from $[2^{-n-1},2^{1-n}]$ in space. For $n=0$ take $v_0=Z_b$;
after an infinite stopping time all subsequent times are infinite.
At either an upper or lower exit, $\inf_xv_n(\tau_n,x)\ge2^{-n-1}$.
Comparison therefore gives $v_n\le v_{n-1}\le Z_b$ after each restart.
Until its exit, $2^nv_n\in[1/2,2]$, so its stopped noise coefficient is
bounded by $2K$. By \eqref{kpz:stopped-tail}, conditionally on the past,
\[
 \mathbb P\{\tau_n-\tau_{n-1}<u\mid\mathcal F_{\tau_{n-1}}\}
 \le 1\wedge Ce^{-cu^{-1/2}}
 \quad\text{on }\{\tau_{n-1}<\infty\}.
\]
At the exit the stopped convolution equals $2^nv_n-1$.
On $\{\tau_n<T\}$, at least $(n+1)/2$ durations are below $2T/(n+1)$;
successive conditioning gives
\[
 \mathbb P(\tau_n<T)
 \le 2^{n+1}\left(1\wedge Ce^{-c\sqrt{(n+1)/(2T)}}\right)^{(n+1)/2}
 \le Ce^{-cn^{3/2}}.
\]
Since $\{\inf_xZ_b(t,x)<2^{-n-1}\}\subset\{\tau_n<t\}$,
\[
 \mathbb E Z_b(t,x)^{-r}
 \le C_r\left(1+\sum_{n\ge0}2^{r(n+1)}\mathbb P(\tau_n<t)\right)
 \le C_r\left(1+\sum_{n\ge0}2^{r(n+1)}Ce^{-cn^{3/2}}\right)<\infty.
\]
\end{proof}

\subsection{Coupling derivatives and the number operator}

Put $\mathsf T=t\partial_t$ and let
\[
 d\mathcal Z_r^{t,\kappa}=\nu t\partial_x^2\mathcal Z_r^{t,\kappa}\,dr
       +\kappa\sigma\sqrt t\,\mathcal Z_r^{t,\kappa}\,dW_r,
 \qquad \mathcal Z_0^{t,\kappa}=1,\qquad
 F_\kappa(t)=\log\mathcal Z_1^{t,\kappa}.
\]
Then $\mathcal Z_1^{t,\kappa}$ has the law of the heat solution at time $t$.
For bounded $\kappa$ and $0<t\le t_0$,
\begin{equation}\label{kpz:second-jet}
 \begin{aligned}
 \sup_x\|\mathsf T^\ell\partial_\kappa^mF_\kappa(t,x)\|_p
 &\le C_p t^{m/4},\qquad m=2,3,\quad \ell=0,1,\\
 \sup_x\|DF_\kappa(t,x)\|_{L^p(\Omega;\mathcal K)}
 &\le C_p|\kappa|t^{1/4}.
 \end{aligned}
\end{equation}
These derivatives are locally continuous for $t>0$.

\emph{Heat expansion.}
The ordered $n$th chaos kernel is
\[
 g_{n;t,x}=t^{n/2}\mathbf1_{0<r_1<\cdots<r_n<1}
       \prod_{i=1}^nP_{t(r_{i+1}-r_i)}(y_{i+1}-y_i),
 \qquad (r_{n+1},y_{n+1})=(1,x).
\]
Set $f_n=\operatorname{Sym}g_n$. For $0<s\le t_0$, Parseval gives
\[
 \|P_s\|_2^2+\|s\partial_sP_s\|_2^2
 =\sum_{k\in\mathbb Z}\bigl[1+(4\pi^2\nu sk^2)^2\bigr]
       e^{-8\pi^2\nu sk^2}\le Cs^{-1/2}.
\]
Since $n!\|f_n\|^2=\|g_n\|^2$, the product rule, spatial integration and
\[
 \int_{0<r_1<\cdots<r_n<1}
       \prod_{i=1}^n(r_{i+1}-r_i)^{-1/2}\,dr
 =\frac{\Gamma(1/2)^n}{\Gamma(1+n/2)}
\]
give
\begin{equation}\label{kpz:kernel-estimate}
 n!\|\mathsf T^\ell f_{n;t,x}\|^2
 \le\frac{C^n(1+n)^3t^{n/2}}{\Gamma(1+n/2)},\qquad \ell=0,1,
\end{equation}
In $\mathcal Z=1+\sum_{n\ge1}(\kappa\sigma)^nI_n(f_n)$,
hypercontractivity and the Malliavin formula \cite{Janson,Nualart} therefore give
\[
 \|D^j\mathsf T^\ell\partial_\kappa^m[(\kappa\sigma)^nI_n(f_n)]\|_p
 \le C_p^n(1+n)^C\mathbf1_{n\ge m}(1+|\kappa|)^{n-m}
       \frac{t^{n/4}}{\sqrt{\Gamma(1+n/2)}},
 \quad j\le1,\ \ell\le1,\ m\le3.
\]
Summing gives, locally uniformly,
\[
 \begin{gathered}
 \|\mathsf T^\ell\partial_\kappa^m\mathcal Z\|_p\le C_pt^{m/4}
 \quad(1\le m\le3,\ \ell\le1),\\
 \|\mathsf T\mathcal Z\|_p\le C_pt^{1/4},\quad
 \|D\mathcal Z\|_p\le C_p|\kappa|t^{1/4}.
 \end{gathered}
\]
\emph{Derivatives of the logarithm.}
The formulas
\[
 \partial_\kappa^2\log\mathcal Z
 =\frac{\mathcal Z_{\kappa\kappa}}{\mathcal Z}
       -\frac{\mathcal Z_\kappa^2}{\mathcal Z^2},\qquad
 D\log\mathcal Z=\frac{D\mathcal Z}{\mathcal Z},
\]
and
\[
 \partial_\kappa^3\log\mathcal Z
 =\frac{\mathcal Z_{\kappa\kappa\kappa}}{\mathcal Z}
 -3\frac{\mathcal Z_{\kappa\kappa}\mathcal Z_\kappa}{\mathcal Z^2}
 +2\frac{\mathcal Z_\kappa^3}{\mathcal Z^3},
\]
together with Lemma~\ref{kpz:inverse} and
\[
 \mathsf T(\mathcal Z^{-b}A)
 =\mathcal Z^{-b}\mathsf TA-b\mathcal Z^{-b-1}(\mathsf T\mathcal Z)A,
 \qquad
 \|\mathcal Z^{-b}A\|_p
 \le\|\mathcal Z^{-1}\|_{2bp}^{b}\|A\|_{2p},
\]
prove \eqref{kpz:second-jet}. Apply these formulas first to
$\log(\mathcal Z+\varepsilon)$; for $0<\theta\le1$,
\[
 \|(\mathcal Z+\varepsilon)^{-b}-\mathcal Z^{-b}\|_p
 \le C_{b,\theta}\varepsilon^\theta
       \|\mathcal Z^{-1}\|_{(b+\theta)p}^{b+\theta}\longrightarrow0.
\]
Closedness of $D$ and the integral identities for $\partial_\kappa$ and
$\partial_t$ identify the limits. Kernel continuity and the summable
bound \eqref{kpz:kernel-estimate} give local continuity.

\emph{The number operator.}
The chaos bound gives
\[
 \sum_{n\ge1}n^2(\kappa\sigma)^{2n}n!\|f_n\|^2<\infty,
 \qquad \mathcal N\mathcal Z=\kappa\partial_\kappa\mathcal Z.
\]
The generator chain rule, first for smooth extensions of
$\log(z+\varepsilon)$ below zero, gives in $L^p$
\[
 \mathcal N\log(\mathcal Z+\varepsilon)
 =\frac{\kappa\partial_\kappa\mathcal Z}{\mathcal Z+\varepsilon}
       +\frac{\|D\mathcal Z\|_{\mathcal K}^2}{(\mathcal Z+\varepsilon)^2}
 \longrightarrow \kappa\partial_\kappa F_\kappa
       +\|DF_\kappa\|_{\mathcal K}^2.
\]
Closedness of $\mathcal N$ proves $F_\kappa\in\operatorname{Dom}(\mathcal N)$ and
\begin{equation}\label{kpz:number-identity}
 \mathcal NF_\kappa
 =\kappa\partial_\kappa F_\kappa+\|DF_\kappa\|_{\mathcal K}^2.
\end{equation}

\subsection{The remainder and its derivative}

The norms below are in $L^2(\Omega;H)$.
Let $X_t^0=\Pi_\circ\partial_\kappa F_\kappa(t)|_{\kappa=0}$ and
$Y_{g,t}=\Pi_\circ\kappa^{-1}F_\kappa(t)$, with the removable value at zero.
Taylor's formula gives $Y_{g,t}=X_t^0+R_{g,t}$ and
\[
 R_{g,t}=\kappa\Pi_\circ\int_0^1(1-\theta)
              \left.\partial_b^2F_b(t)\right|_{b=\theta\kappa}\,d\theta.
\]
By \eqref{kpz:second-jet},
\[
 \|R_{g,t}\|_2+\|\mathsf TR_{g,t}\|_2\le C_G|g|t^{1/2}.
\]
Since $X_t^0$ is in the first chaos, \eqref{kpz:number-identity} also gives
\[
 \mathcal NR_{g,t}=\Pi_\circ\left[
 \partial_\kappa F_\kappa-\partial_\kappa F_\kappa|_{\kappa=0}
                  +\kappa^{-1}\|DF_\kappa\|_{\mathcal K}^2\right].
\]
Hence
\[
 \|\mathcal NR_{g,t}\|_2
 \le |\kappa|\sup_{|b|\le|\kappa|}\|\partial_b^2F_b\|_2
       +|\kappa|^{-1}\|DF_\kappa\|_{L^4(\Omega;\mathcal K)}^2
 \le C_G|g|t^{1/2}.
\]

For $k\ne0$, Lemma~\ref{gaalign:ou} gives
\[
 \begin{gathered}
 \widehat X_t^0(k)=\sigma\sqrt t\int_0^1
       e^{-4\pi^2\nu tk^2(1-r)}\,d\widehat W_r(k),\\
 (\mathsf T+d\Gamma(\mathsf A_t))X_t^0=tB_tX_t^0.
 \end{gathered}
\]
Put
$D_{g,t}=\mathsf TY_{g,t}-tB_tY_{g,t}+d\Gamma(\mathsf A_t)Y_{g,t}$.
By \eqref{common:ward-number},
\[
 \begin{gathered}
 D_{g,t}=(\mathsf T-tB_t+d\Gamma(\mathsf A_t))R_{g,t},\\
 \|D_{g,t}\|_2\le\|\mathsf TR_{g,t}\|_2
       +\tfrac12\|R_{g,t}\|_2+\tfrac12\|\mathcal NR_{g,t}\|_2.
 \end{gathered}
\]
Thus
\begin{equation}\label{kpz:required-defect}
 \|R_{g,t}\|_2+\|D_{g,t}\|_2\le C_G|g|t^{1/2}.
\end{equation}
The realization $Y_{g,t}$ has law $\mu_{g,t}$ and is locally $C^1$
in $L^2(\Omega;H)$. Lemma~\ref{common:ward} gives
$e(t)\le t^{-1}\|D_{g,t}\|_2$ for every $t>0$.

\emph{The first chaos starts one order later.}
Let $P_1$ be the orthogonal projection onto the first chaos of the fixed
Gaussian space and write $F^{(1)}=P_1F$. At zero coupling,
\[
 \left.\partial_\kappa^2F_\kappa(t,x)\right|_{\kappa=0}
 =2\sigma^2I_2(f_{2;t,x})-\sigma^2I_1(f_{1;t,x})^2.
\]
Since $I_1(f)^2=I_2(f\otimes f)+\|f\|^2$,
$P_1\partial_\kappa^2F_\kappa|_{\kappa=0}=0$. Taylor's formula gives
\[
 R_{g,t}^{(1)}
 =\frac{\kappa^2}{2}\Pi_\circ P_1\int_0^1(1-\theta)^2
       \left.\partial_b^3F_b(t)\right|_{b=\theta\kappa}\,d\theta.
\]
By \eqref{kpz:second-jet},
\[
 \|R_{g,t}^{(1)}\|_2+\|\mathsf TR_{g,t}^{(1)}\|_2
 \le C_Gg^2t^{3/4}.
\]
Since $[P_1,B_t]=[P_1,d\Gamma(\mathsf A_t)]=0$ and
$\|\mathsf A_t\|,\|tB_t\|\le1/2$,
\[
 \|D_{g,t}^{(1)}\|_2
 \le\|\mathsf TR_{g,t}^{(1)}\|_2+\|R_{g,t}^{(1)}\|_2,
\]
whence
\begin{equation}\label{kpz:first-chaos-defect}
 \|R_{g,t}^{(1)}\|_2+\|D_{g,t}^{(1)}\|_2\le C_Gg^2t^{3/4}.
\end{equation}

\clearpage
\section*{Notation and proof dependencies}

For Section~\ref{app:phi4} and
Appendices~\ref{ga:section}--\ref{cal:projection-section};
regularity indices are fixed in \eqref{eq:indices}.
\begin{center}\small
\begin{tabular}{@{}ll@{}}
\toprule
Symbol & Meaning\\\midrule
$\varpi,\ \epsilon_0,\ q$ & loss in the rate; small index; large time exponent\\
$\alpha,\ \sigma,\ \beta,\ r_*$ & coefficient order, test order, weight regularity, response regularity\\
$\theta,\ \gamma,\ \omega$ & state rate, state gain, local heat weight\\
$Z,\ W,\ W_3$ & Gaussian field, Wick square, Wick cube\\
$Y,\ \Psi,\ \Omega_J$ & heat primitives of $W$, $P_NW_3$ and of the paraproduct correction\\
$R_2,\ R_3,\ S$ & renormalized resonant products (Proposition~\ref{ga:model})\\
$K,\ \mathcal T=I+K$ & smoothing operator and change of variables\\
$u=X_N-Z_N,\ q,\ q_0$ & nonlinear remainder; additive corrector and its coordinate form\\
$r=u+q,\ w=\mathcal Tr,\ v=r/h$ & corrected state, its coordinate, its weighted quotient\\
$h,\ k,\ a_\pm$ & forward and backward weights, their equation defects\\
$M,\ h_0$ & random coefficient bound; time-dependent coefficient bound\\
$\mathcal K_{N,T}$ & norm of the backward flow \eqref{co:eq:backward-K}\\
$\eta=3g,\ P=P_N,\ Q=I-P$ & cubic coupling; spectral projection and its complement\\
\bottomrule
\end{tabular}
\end{center}
Elsewhere $\sigma$ is the KPZ noise amplitude, $r_t$ the residual
current of Theorem~\ref{common:principle}, and $K$ the constant
\eqref{common:gaussian-limit}.

\paragraph{Order of the argument.}
The estimates are used in this order:
\begin{enumerate}[label=(\roman*)]
\item Gaussian fields and their derivatives
 (Appendix~\ref{ga:section}, Lemma~\ref{ga:mark-stability});
\item the coordinate $\mathcal T$ and its adjoint
 (Appendix~\ref{ac:section});
\item the cancellation identities and the local solution map
 (Appendix~\ref{po:section}, Lemmas~\ref{up:operators} and~\ref{po:local});
\item the explicit weights and the reference event
 (Appendix~\ref{ac:reference-section});
\item the projection errors (Appendix~\ref{cal:projection-section});
\item the two entropies and the heat estimate
 (Lemmas~\ref{state:quadratic}, \ref{state:quartic},
 Proposition~\ref{po:position});
\item the controlled state and the response
 (Proposition~\ref{cal:control});
\item the source bound, response moments and reconstruction
 (Section~\ref{co:section});
\item identification of the continuum field
 (Proposition~\ref{full:identification}), its differentiable realization
 (Proposition~\ref{ma:full-differentiation}) and completion
 (Section~\ref{se:completion}).
\end{enumerate}

\section*{Use of AI tools}
The author set the research goals and developed the main ideas, concepts, and proof strategies. ChatGPT (OpenAI) assisted the author with some technical calculations and developing certain conceptual arguments into technical proofs. The author takes responsibility for the final manuscript and its mathematical claims.

\end{document}